\UseRawInputEncoding

\documentclass{amsart}
\usepackage{amsthm,amsmath,amssymb,amscd,graphics,enumerate, stmaryrd,xspace,verbatim, epic, eepic,url}

\usepackage[usenames,dvipsnames]{color}

\usepackage[colorlinks=true]{hyperref}

\usepackage[active]{srcltx}
\usepackage[all]{xypic}
\SelectTips{cm}{}

{

   \newtheorem{theorem}[subsubsection]{Theorem}
      \newtheorem*{theorem*}{Theorem}
   \newtheorem{proposition}[subsubsection]{Proposition}

   \newtheorem{lemma}[subsubsection]{Lemma}

   \newtheorem{corollary}[subsubsection]{Corollary}
   
   \newtheorem*{conjecture*}{Conjecture}

}
{\theoremstyle{definition}
          \newtheorem*{exercise*}{Exercise}
   
   \newtheorem{example}[subsubsection]{Example}
   \newtheorem*{example*}{Example}
   \newtheorem{definition}[subsubsection]{Definition}
   
   \newtheorem*{definition*}{Definition}
   
   \newtheorem{remark}[subsubsection]{Remark}

}
\newcommand{\CC}{{\mathbb{C}}}
\newcommand{\QQ}{{\mathbb{Q}}}
\newcommand{\NN}{{\mathbb{N}}}
\newcommand{\MM}{{\mathbb{M}}}
\newcommand{\PP}{{\mathbb{P}}}
\newcommand{\ZZ}{{\mathbb{Z}}}

\newcommand{\GG}{{\mathbb{G}}}

\renewcommand{\AA}{{\mathbb{A}}}

\newcommand{\m}{{\mathfrak{m}}}

\newcommand{\cA}{{\mathcal A}}

\newcommand{\cC}{{\mathcal C}}
\renewcommand{\cD}{{\mathcal D}}

\newcommand{\cI}{{\mathcal I}}
\newcommand{\cJ}{{\mathcal J}}

\newcommand{\cO}{{\mathcal O}}

\def\<{\langle}
\def\>{\rangle}
\newcommand{\inn}{\operatorname{in}}

\newcommand{\Spec}{\operatorname{Spec}}

\newcommand{\Cox}{\operatorname{Cox}}
\newcommand{\Proj}{\operatorname{Proj}}

\newcommand{\Sing}{{\operatorname{Sing}}}
\newcommand{\Star}{\operatorname{Star}}

\newcommand{\Hom}{{\operatorname{Hom}}}

\newcommand{\Cl}{{\operatorname{Cl}}}
\newcommand{\divv}{{\operatorname{div}}}
\newcommand{\Div}{{\operatorname{Div}}}

\newcommand{\codim}{\operatorname{codim}}

\newcommand{\dar}{\downarrow}

\newcommand{\Ver}{{\operatorname{Vert}}}
\newcommand{\irr}{{\operatorname{irr}}}
\newcommand{\inte}{{\operatorname{int}}}

\newcommand{\Bl}{{\operatorname{Bl}}}

\def\:{{\colon}}
\def\.{{,\dots,}}

\newcommand{\double}{\genfrac..{0pt}1
{\raise -1pt\hbox{$\scriptstyle\longrightarrow$}}{\raise 3pt\hbox
{$\scriptstyle\longrightarrow$}}}

\renewcommand{\setminus}{\smallsetminus}

\def\sat{{\rm sat}}
\def\nor{{\rm nor}}

\def\tototi{\mathbin{\mathop{\otimes}\limits^{\raise-1pt\hbox
{$\scriptscriptstyle {\rm L}$}}}}

\def\indlim{\mathop{\vrule width0pt height7pt depth
4pt\smash{\lim\limits_{\raise 1pt\hbox to 14.5pt
{\rightarrowfill}}}}}
\def\projlim{\mathop{\vrule width0pt height7pt depth
4pt\smash{\lim\limits_{\raise 1pt\hbox to 14.5pt
{\leftarrowfill}}}}}

\newcommand\displaceamount{3pt}

\newcommand{\doubledown}{\ar@<\displaceamount>[d]\ar@<-\displaceamount>[d]}

\newcommand{\doubleup}{\ar@<\displaceamount>[u]\ar@<-\displaceamount>[u]}

\newcommand{\doubleright}{\ar@<\displaceamount>[r]\ar@<-\displaceamount>[r]}

\newcommand{\con}{{\operatorname{conv}}}

\def\id{\text{id}}

\newcommand{\val}{{\operatorname{val}}}
\newcommand{\cha}{{\operatorname{char}}}
\newcommand{\Prin}{{\operatorname{Prin}}}

\newcommand{\ord}{{\operatorname{ord}}}

\newcommand{\gr}{{\operatorname{gr}}}

\newcommand{\inv}{{\operatorname{inv}}}

\def\supp{{\operatorname{supp}}}

\begin{document}
\title{ Cox rings of  morphisms and resolution of singularities}

%

\author[J. W{\l}odarczyk] {Jaros{\l}aw W{\l}odarczyk}
\address{Department of Mathematics, Purdue University\\
150 N. University Street,\\ West Lafayette, IN 47907-2067}
\email{wlodarcz@purdue.edu}

\thanks{This research is supported by  BSF grant 2014365} 

\date{\today}

\begin{abstract} We extend the Cox-Hu-Keel construction of the Cox rings to any proper birational morphisms of normal noetherian schemes. 
It allows the representation of any proper birational morphism   by a map of schemes with mild singularities with torus actions.

In a  particular case, the notion generalizes the combinatorial construction of Satriano \cite{Satriano} and the recent construction of multiple weighted blow-ups on Artin-stacks by Abramovich-Quek \cite{AQ}.

The latter can be viewed as an extension of stack theoretic  blow-ups by  Abramovich, Temkin and W\l odarczyk \cite{ATW-weighted}, a similar construction  of McQuillan \cite{Marzo-McQuillan} and and the author's recent cobordant recent cobordant blow-ups \cite{Wlodarczyk22} at weighted centers to a more general situation of arbitrary  locally monomial centers.  
 
 We show some applications of this operation to the resolution of singularities over a field of any characteristic.

\end{abstract}
\maketitle
\setcounter{tocdepth}{1}


\section{Introduction}
The importance of
$G_m$-actions in birational geometry and their connection with the Mori
theory was already discovered by Reid, Thaddeus,  and many
others (see \cite{Tha1}, \cite{Tha2}, \cite{Tha3}, \cite{R}, \cite{DH}).  
This was also reflected in 
the proof of the Weak Factorization theorem, which relied on the notion 
of {\it birational cobordism} and a critical role of $G_m$-action
\cite{W-Cobordism}, \cite{W-Factor}, \cite{AKMW}. 

The idea of  the birational cobordism from \cite{W-Cobordism} is to construct a smooth variety with $G_m$-action which represents a proper birational  morphism  and parametrizes possible birational elementary modifications such as blow-ups, blow-downs, and flips. This allows decomposing the proper birational maps of smooth varieties into a sequence of blow-ups and blow-downs with smooth centers.


A similar idea was considered shortly after by Hu-Keel \cite{Hu-Keel}, who constructed their {\it Mori dream space}, parametrizing  possible birational modifications in the Mori program via torus actions.
The Mori dream space plays a vital role in the Mori theory.
One of the key ingredients in constructing the Mori dream space is the Cox rings.

Recall that the Cox rings for toric varieties were considered first by Cox in \cite{Cox}. The main idea of the construction comes from the convex geometry: Any  polyhedral complex can be realized as the image of the simplicial complex. Similarly, any fan in toric geometry can be  represented as the image of the subfan of a regular  cone.
This simple observation leads to the fundamental formula
describing the {\it Cox coordinate ring} of tor the toric variety $X$ as
$$\cC(X):=\bigoplus_{D\in \Cl(X)} H^0(X,\cO_X(D)),$$
where $\Cl(X)$ is the Weil divisor class group.
The action of torus $T=\Spec \ZZ[\Cl(X)]$ naturally occurs 
in the construction, and is determined by the  $\Cl(X)$-gradation.

The Cox formula  generalizes the construction of the {\it coordinate ring} of the projective scheme $X=\PP^n_\ZZ$, namely $$\ZZ[x_0,\ldots,x_n]=\bigoplus_{n\in \ZZ} H^0(X,\cO_X(n))$$

The projective space $X=\PP^n$  can be seen as the geometric quotient of the {\it characteristic space} 
$$\hat{X}=\Spec_X(\bigoplus_{D\in \Cl(X)} \cO_X(D)\to X,$$
introduced
in \cite{Cox2}. The characteristic space $\hat{X}$ comes with the natural embedding $\hat{X}\hookrightarrow \overline{X}$ into the coordinate space:
$$\overline{X}:=\Spec(\bigoplus_{D\in \Cl(X)} H^0(X,\cO_X(D))$$

In particular, for $X=\PP^n$ we obtain $$\hat{X}=\AA^{n+1}_{\ZZ}\setminus \{0\} \hookrightarrow  \overline{X}=\AA^{n+1}_{\ZZ}$$

This leads to the standard $Proj$ -construction:
$$
\begin{array}{cccc} Proj(\ZZ[x_0,\ldots,x_n]) & = & (\Spec(\ZZ[x_0,\ldots,x_n]\setminus V(x_0,\ldots,x_n))/G_m  \\
  \parallel &   & \parallel \ \\ X &= & \hat{X}/T,  
   \end{array}.
$$

The Cox rings have found many different applications in algebraic geometry. (See  \cite{AH06}, \cite{AIPSV12}, \cite{AP12}, \cite{ACDL21}, \cite{AHL10}, \cite{ABHW18}, \cite{BP04}, \cite{BHHN16}, \cite{DHHKL15}, \cite{HT04}, \cite{HKL16}, \cite{HAM21}\cite{HM15}, \cite{HMT20}.)
In particular, they were  used to study minimal resolutions of singularities (see \cite{DB16}, \cite{DBG17}, \cite{DBK16}, \cite{FGL11},  \cite{G20}). 
In the recent paper \cite{BM21} the Cox rings were defined for the morphisms of log pairs.

In this paper, we introduce the idea of  Cox rings of the proper birational morphisms of any normal  noetherian schemes and propose a more general  approach to embedded resolution problems in the language of torus actions, extending the ideas of McQuillan \cite{Marzo-McQuillan} and Abramovich-Temkin-W\l odarczyk  \cite{ATW-weighted} of the weighted resolution, and  Abramovich- Quek \cite{AQ} of the multiple weighted resolutions.

The idea of utilizing group actions to resolve singularities is ancient and should be traced back to Newton.
In the method that he developed, known later as Newton-Puiseau theorem, he shows that any polynomial function $f(x,y)$ on $X=\CC^2$ with expansion containing the term $y^r$ can be, in fact, upon a coordinate change, resolved by a Newton-Puiseau series $y=g(x^{1/k})$. In other words one considers  the space $X'=\CC^2$ with the group action of $\mu_k=\langle\xi\rangle$, $\xi(x,y)=(\xi\cdot x,y)$, giving the quotient $X'\to X$, $(x,y)\mapsto (x^k,y)$, and a  smooth  holomorphic branch  $V(y-g(x^{1/k}))$ on  $X'$ parametrizing subspace $V(f)$ on $X$.


Originally in the Hironaka embedded resolution, only smooth centers were used (see \cite{Hironaka},\cite{Villamayor},\cite{Bierstone-Milman},\cite{Wlodarczyk},\cite{Encinas-Hauser}, \cite{Encinas-Villamayor}, \cite{Kollar}).
In the recent papers \cite{ATW-principalization}, \cite{ATW-relative} in the resolution process of logarithmic schemes and morphisms, we considered  the stack-theoretic blow-ups  of the centers of the form $$\cJ=(u_1,\ldots, u_k, m_1^{1/w_{1}},\ldots,m_r^{1/w_r}),$$ in the context  of  Kummer \'etale topology on the logarithmic stacks. The functorial properties of the algorithm of logarithmic resolution of morphisms dictated such general centers.

Then in \cite{ATW-weighted}, we developed the formalism of the stack-theoretic blow-ups of the weighted centers of the form $(u_1^{1/w_1},\ldots, u_k^{{1/w_k}})$.
This approach allows 
to simplify the resolution procedure in characteristic zero. The  algorithm is more efficient and avoids many unnecessary blow-ups reducing    technicalities. It uses a very simple geometric invariant, which improves after each step and is independent of the logarithmic structure. 
A similar result was obtained by McQuillan in \cite{Marzo-McQuillan}.
 More general centers were considered in the  paper\cite{Quek} of  Quek in the logarithmic context.

In work \cite{AQ} of Abramovich-Quek, the authors introduce {\it multi-weighted blow-ups}, further extending the 
results in \cite{Quek}. The multiple weighted blow-ups generalize the weighted blow-ups and are used to obtain a smooth and toroidal resolution version of Artin stacks (see Section \ref{AQ}). The Abramovich-Quek weighted blow-up  generalizes the Satriano toroidal construction on Artin logarithmically smooth stacks in \cite{Satriano} to  locally monomial ideals. 

Subsequently in the  paper \cite{Wlodarczyk22} the operation of {\it cobordant blow-up} $B_+\to X$ with weighted centers $\cJ=(u_1^{1/w_1},\ldots, u_k^{{1/w_k}})$ was  introduced, where $u_1,\ldots, u_k$ is a  partial 
system of local parameters \begin{align} B=\Spec_X({\cO_X}[{ t}^{-1}, { t}^{w_1}{x_1},\ldots, { t}^{w_k}{x_k}]),\quad B_+=B\setminus V(x_1{ t}^{w_1},\ldots,x_k{ t}^{w_k}), \end{align}  where $t$ is an introduced unknown.
A similar formula was discovered by Rydh in the paper of \cite{Rydh-proj} and studied in the context of the stack-theoretic blow-ups. Moreover, a certain relation between toric Cox construction and toric weighted cobordant blow-ups was already observed in \cite{Rydh-proj} and \cite{Wlodarczyk22}.

The operation of cobordant blow-up allows representing stack-theoretic weighted blow-ups and more general  Kummer blow-ups in the language of smooth varieties with torus action without stack theoretic language. Moreover, apart from fast functorial resolution with SNC  divisors in characteristic zero, the approach leads to the resolution of some classes of singularities in positive and mixed characteristic (see \cite{Wlodarczyk22}).


In the present paper, we  associate the {\it Cox coordinate ring} to arbitrary proper birational morphisms $\pi: Y\to X$ of normal  noetherian schemes as follows: 
$$\cA_{Y/X}:= \pi_*(\bigoplus_{E\in \Cl(Y/X)} \cO_Y({E})),$$ 
where $\Cl(Y/X)\subset \Cl(Y)$ is a free group generated by the exceptional divisors. It comes with the coaction of the associated torus $T=\Spec(\ZZ[\Cl(Y/X)])$.

Per analogy with the standard Cox construction, we call the space \begin{align} B:=\Spec_X(\pi_*(\bigoplus_{E\in \Cl(Y/X)} \cO_Y({E}))\end{align} the {\it relative Cox coordinate space}.
The  scheme \begin{align} B_+:=\Spec_Y(\bigoplus_{E\in \Cl(Y/X)} \cO_Y({E}) \end{align} will be called the {\it relative Cox characteristic space}.

In this language, any proper birational  morphism $\pi: Y\to X$ can be represented by a $T$-equivariant morphism $B_+\to B$ such that the induced morphism of the good quotient coincides with $\pi: Y\to X$: $$
\begin{array}{cccc} B_+\sslash T & \to & B\sslash T  \\
  \parallel &   & \parallel \ \\ Y &\stackrel{\pi}\to & X,  
   \end{array}.
$$


 As in the standard construction, the morphism $B_+\subset B$ is an open immersion upon some reasonable assumptions.

If $Y\to X$ is the  the blow-up of the ideal $\cJ$ on $X$ the associated presentation  $B_+\sslash T\to X$  can be thought as  the normalized extended $Proj$ introduced by Swanson-Huneke  \cite{HS}:  $$B_+\sslash T=Proj_X(\cO_X[\cJ t,{\bf t}^{-1}])^{\nor}.$$
(See Section \ref{blow3}.)

Note that the morphism $B\to X$ is affine and is locally described by a  single chart.
The spaces  $B_+$ and $B$ usually have nicer singularities and simpler descriptions, and the morphism $B_+\subset B$ is  way simpler    than the original $\pi: Y\to X$. 
As in  the standard Cox construction, the semiinvariant functions on $B_+$ and $B$ can be interpreted as forms on $Y$  and are convenient for the computations.

For instance, the construction can be applied to normalized blow-ups  of locally monomial centers, leading to general classes of  modifications of singularities of subschemes and ideals that preserve regular ambient schemes.

Given a locally toric or simply regular scheme  $X$ over a field and  any locally toric proper birational morphism $\pi: Y\to X$, one associates with $\pi$ a morphism of Cox regular spaces
$B_+\subset B$, where 
$$B=\Spec(\cO_X[t_1^{-1},\ldots,t_k^{-1},u_1{\bf t}^{\alpha_1},\ldots,u_k{\bf t}^{\alpha_k}])$$

 In particular, one  represents  the normalized blow-up of any locally monomial $\cJ$ by  a {\it smooth cobordant blow-up $B_+\to X$ of $\cJ$} equipped  with torus action. The formula generalizes the weighted cobordant blow-up introduced in \cite{Wlodarczyk22} with $B_+=B\setminus V(\cJ {\bf t}^\alpha),$ for the corresponding multi-indexes $\alpha,\alpha_1,\ldots,\alpha_k$. It also leads to  a version of {\it the multi-weighted blow-up} of \cite[Definition 2.1.6]{AQ} when considering the stack theoretic quotient $[B_+\sslash T]$.

 
One can think of this approach as an extension of the resolution by  cobordant blow-ups with weighted centers to more general locally monomial ideals or $\QQ$-ideals. 

When replacing the group $\Cl(Y/X)$ with a subgroup $\Gamma\subset \Cl(Y/X)\otimes \QQ$ in the formulas (2) and (3), one further 
generalizes the construction. We obtain

\begin{align*} B^\Gamma:=\Spec_X(\pi_*(\bigoplus_{E\in \Gamma} \cO_Y({E})),\quad B^\Gamma_+:=\Spec_Y(\bigoplus_{E\in \Gamma} \cO_Y({E}). \end{align*} 

This generalized construction can be linked to the weighted cobordant blow-ups as in \cite{Wlodarczyk22} (See Section \ref{weighted2}).
In particular,  let  $\pi: Y\to X$  be the weighted  blow-up of  a regular scheme with  the  $\QQ$-ideal center $\cJ=(u_1^{1/w_1},\ldots,u_k^{1/w_k})$. This is simply the normalized blow-up of the ideal $\cJ^{(a)}:=(u_1^{a/w_1},\ldots,u_k^{a/w_k})$, with the exceptional irreducible $\QQ$-Cartier divisor  $(1/a) E_a$ with $\cO_Y(-E_a)=\cO_X\cdot\cJ^{(a)}$, where $a$ is any positive integer such that $w_i|a$,.
The cobordant blow-up of $\cJ^{(a)}$ with respect  to the group $\Gamma=\ZZ\cdot\frac{1}{a}E_a\subset \Cl(Y/X)\otimes \QQ$ gives the formula (1) for the cobordant weighted blow-up  of $\cJ=(u_1^{1/w_1},\ldots,u_k^{1/w_k})$.  Note that the above definition does not depend upon the choice  of $a$, and the $\QQ$-Cartier divisor $(1/a) E_a$ can be interpreted as the divisor corresponding to the $\QQ$-ideal $\cO_Y\cdot \cJ=\cO_Y(-E_a)^{1/a}$.

More generally, let  $\cJ$ be a  locally monomial center, and  $\pi: Y\to X$ be the normalized blow-up of $\cJ$. Denote by $E_1,\ldots, E_k$ the exceptional divisors of $\pi$. The cobordant blow-up  of $\cJ$ with respect to the subgroup $$\Gamma=\ZZ \frac{1}{b_1}E_1\oplus\ldots\oplus\ZZ \frac{1}{b_k}E_k\subset \Cl(Y/X)\otimes \QQ,$$ generated
by $\frac{1}{b_1}E_1,\ldots,\frac{1}{b_k}E_k$, where $b_1,\ldots,b_k$ are any positive integers, leads to the {\it multiple weighted blow-up}, considered by Abramovich-Quek in \cite{AQ}.  It can be understood  as
the fantastack associated with the stack-theoretic quotient
$[ B_+\sslash T]$ (See Section \ref{AQ}). Since the stabilizers of the action are not finite, in general, one obtains an Artin stack as the stack-theoretic  quotient.

 Note that in the resolution process of hypersurfaces, one often considers locally 
 the corresponding Newton polytope. It is naturally associated with a certain coordinate system and rises to a locally monomial center. In a more general setting, the Newton polytope is replaced with the {\it dual valuation complex} of the locally monomial center.
 We show some conditions for singularities when the cobordant blow-up of such a center immediately resolves singularities.
 (see Theorems \ref{regular}, \ref{A}, \ref{B}, \ref{res}, \ref{res2}, \ref{res3}, \ref{res4}).  The particular resolution methods and theorems extend  the relevant results for the weighted cobordant blow-ups in \cite{Wlodarczyk22}. As a  Corollary \ref{AQ21}, we obtain  Abramovich-Quek's \cite[Theorem 5.1.2]{AQ}.

 


  The resolution algorithm outputs a regular scheme with a torus action which admits a good quotient having locally toric singularities and birational to the original scheme. It can be directly resolved by the canonical combinatorial methods in any characteristic as in \cite[Theorem 7.17.1]{Wlodarczyk-functorial-toroidal}. Alternatively, by  Proposition \ref{geometric}, one can always replace  in the resolution process each $B_+$ with an open stable subset $B^s$ admitting  a geometric quotient, and then apply the destackification method of  Bergh-Rydh in \cite{Bergh-Rydh}. It is also possible to use the €œcanonical reduction of stabilizers due to Edidin- Rydh \cite{ER}, and then the destackification method of  Bergh-Rydh in \cite{Bergh-Rydh}.

\subsubsection{Aknowledgements} The author would like to thank Dan Abramovich, J\"{u}rgen Hausen, Antonio Laface, Michael Temkin, Ilya Tyomkin, and Jaros\l aw Wi\'sniewski for helpful discussions and suggestions.

\subsection{Preliminaries} 
 

The definition of Cox spaces of morphisms  is similar, with some important differences, to the  notion of Cox spaces of varieties, as presented in \cite{Cox2}. We shall assume that all the schemes considered in the paper are noetherian.

\subsubsection{Construction of Cox sheaves}

 Given a proper birational morphism $\pi:Y\to X$ of normal integral  schemes, 
consider the 
the free group $\Cl(Y/X)\subset \Div(Y)$ generated by the images of the exceptional irreducible divisors $E_i$. 
It can be identified with the  kernel of  the surjective morphism $\pi_*: \Cl(Y)\to \Cl(X)$. 


\begin{definition}
By {\it the relative Cox ring} w mean the sheaf of graded $\cO_Y$-algebras
$$\cC_{Y/X}=\bigoplus_{E\in \Cl(Y/X)} \cC_{E}=\bigoplus_{E\in \Cl(Y/X)} \cO_Y({E}),$$ graded by $\Cl(Y/X)$, where $C_{E}:=\cO_Y(E)$ for
$$\cO_Y(E)(U)=\{f\in \kappa(Y) \mid (\divv_Y(f)+E)_{|U}\geq 0\}\subset \kappa(Y)=\kappa(X).$$
\end{definition}

Note the $C_0=\cO_Y$. One can  introduce the  dummy variables $t=(t_1,\ldots,t_k)$ so that  $E_i$ corresponds to $t_i^{-1}$  and $E\mapsto {\bf t}^E$. This defines the isomorphism of the gradings: 
$$\Cl(Y/X)\simeq \{{\bf t}^{-\alpha}\mid \alpha\in \ZZ^k\}\simeq \ZZ^k$$  

Using this notation, we can write
$$\cC_{Y/X}=\bigoplus_{E\in \Cl(Y/X)} C_{E}{\bf t}^{E}=\bigoplus_{\alpha\in \ZZ^k} C_{\alpha}\cdot t_1^{a_1}\cdot\ldots\cdot t_k^{a_k} \subseteq \bigoplus_{E\in \Cl(Y/X)}\kappa(Y){\bf t}^{E}$$


	

\subsubsection{Forms}



As mentioned, the Cox relative ring construction, similarly to the absolute case, is analogous to the coordinate ring $\ZZ[x_0,\ldots,x_n]$ on projective space $X=\PP^n_\ZZ$. One can choose a very ample divisor, for instance $D=V(x_0)$, and identify  the functions $f=F(x_0,\ldots,x_n)/x_0^n\in \cO_X(nD)$ with the forms $F(x_1,\ldots,x_n)$ so that the vanishing locus $V(F)$ equals to $$V(F)=V_X(F)=\divv(f)+nD.$$

Per this analogy, and as in \cite{Cox2}, the elements in $\cC_{E}$ will be called {\it forms of degree $E$} on $Y$ and can be written  formally as $F=f{\bf t}^E$  , where $f\in \cO_Y(E)$, with the natural componentwise operation of addition and multiplication. We also define the {\it divisor  of the form $F=f{\bf t}^E$} on $Y$ as $\divv_Y(F)=\divv_Y(f)+E$, and 
its {\it vanishing locus  on $Y$} to be $$V_Y(F):=\supp(\divv_Y(f)+E).$$


\subsubsection{Exceptional valuations}
By the {\it exceptional valuations}  of $\pi:Y\to X$ we shall mean the valuations $\nu_1,\ldots,\nu_k$ of $\kappa(X)=\kappa(Y)$ associated with the generic points of the exceptional divisors $E_1,\ldots,E_k$ of $\pi$.

These valuations define  ideals $\cI_{\nu,a,X}\subset \cO_X$ on $X$  for $a\in \ZZ$, generated by the functions  $f\in \cO_X$, with $\nu(f)\geq a$. In particular $\cI_{\nu,a}=\cO_X$ if $a\leq 0$.

\begin{lemma} \label{divi} Let $E= \sum n_i E_i$ correspond to $t_1^{-n_1}\cdot\ldots\cdot t_k^{-n_1}$. Then 

\begin{enumerate}
\item $\pi_*(\cO_Y(E_i))=\cO_X$. 
\item If all $n_i\geq 0$ then $\pi_*(\cO_Y(E))=\cO_X$.
\item  If there is $n_i< 0$, then $$\pi_*(\cO_Y(E))=\bigcap_{n_i<0}\, \cI_{\nu_i,-n_i,X}=\bigcap^k_{i=1}\, \cI_{\nu_i,-n_i,X}.$$

\end{enumerate}

\end{lemma}
\begin{proof} 

First, since $\pi: Y\to X$ is proper, birational and $X$ is normal, we have $\pi_*(\cO_Y)=\cO_X$.

We can reduce the situation to  the case when $X$ is affine since the problem is local on $X$.
Then    $$g\in   \cO_Y(E)(\pi^{-1}(X))\subset \kappa(X)=\kappa(Y)$$ if and only if $$\divv_Y(g)+E\geq 0.$$ This implies that $\divv_U(g)\geq 0$, where $U:=Y\setminus (\bigcup E_i)$, and $U\subset X$, where $X\setminus U$ is of codimension $\geq 2$. Thus $\divv_X(g)\geq 0$. 
So, since $X$ is normal, w get $g\in \pi_*(\cO_Y)=\cO_X$, whence $\pi_*(\cO_Y(E))\subseteq \cO_X$.

{\bf(1)} and {\bf(2)} If   $E= \sum n_i E_i$ with  $n_i\geq 0$  then $$ \cO_X=\pi_*(\cO_Y)\subseteq \pi_*(\cO_Y(E))\subseteq \cO_X.$$

{\bf(3)} In general, $g\in \pi_*(\cO_Y(E))\subseteq \cO_X$ iff $\divv_Y(g)+E\geq 0$. This translates into   $\divv_Y(g)+\sum_{n_i<0} n_iE_i\geq 0$ by part (2). Thus $\nu_i(g)\geq -n_i$ for all $n_i<0$, which yields  $$g\in \bigcap_{n_i<0} \cI_{\nu_i,-n_i,X}=\bigcap^k_{i=1} \cI_{\nu_i,-n_i,X}.$$

We use here the fact that by definition $\cI_{\nu_i,-n_i}=\cO_X$ if $n_i\geq 0$.
\end{proof}

\subsection{Cox coordinate space}
\subsubsection{Cox algebra}
As a corollary  from Lemma \ref{divi}, we obtain    
\begin{proposition} \label{valu} Let  $\pi: Y\to X$ be a  proper birational morphism of normal irreducible schemes. Assume that $E_k,\ldots,E_k$
are the irreducible exceptional divisors of $\pi$, and $\nu_i$ are the associated valuations.
Then the direct image $\pi_*(\cC_{Y/X})$  of the relative  Cox ring  is a $\Cl(Y/X)=\ZZ^k$-graded $\cO_X$-algebra:

$$\cA_{Y/X}:=\pi_*(\cC_{Y/X})=\bigoplus_{a_i\in \ZZ} \,\, \bigcap^k_{i=1}\, \cI_{\nu_i,a_i}\,\,\cdot t_1^{a_1}\cdot\ldots\cdot t_k^{a_k}\subset \cO_X[t_1,t_1^{-1},\ldots,t_k,t_k^{-1}],$$

where $E_i$ correspond to $t_i^{-1}$. \qed



 
\end{proposition}


%


\subsubsection{Cox coordinate space}
\begin{definition}\label{valu4} Given a proper birational morphism $\pi:Y\to X$ of normal integral schemes. The {\it  Cox relative coordinate space} is the  
scheme $$B=\Cox(Y/X):=\Spec_X(\cA_{Y/X}),$$
over $X$ with the natural action of $T_B=\Spec\ZZ[\Cl(Y/X)]$. The {\it Cox relative characteristic space} is the space $$B_+=\Cox(Y/X)_+:=\Spec_Y(\cC_{Y/X}).$$
over $Y$.
The {\it Cox trivial space} is given by $$B_-:=B \setminus V_B(t_1^{-1}\cdot\ldots\cdot t_k^{-1}).$$
\end{definition}

\subsubsection{Good and geometric quotient} We consider here a relatively  affine action of $$T=\Spec(\ZZ[t_1,t_1^{-1},\ldots,t_k,t_k^{-1}])$$ on a scheme $X$ over $\ZZ$. By the {\it good quotient (or GIT-quotient)} of $X$ by $T$ we mean an affine $T$-invariant morphism $$\pi: X\to Y=X\sslash T$$ such that the induced morphism of the sheaves  $\cO_Y\to \pi_*(\cO_X)$ defines the isomorphism onto the subsheaf of invariants $\cO_Y\simeq \pi_*(\cO_X)^T\subset \pi_*(\cO_X)$.

 Then $\pi: X\to Y=X/T$ will be called the {\it geometric quotient} if additionally every fiber $X_{\overline{y}}$ of $\pi$ over s geometric point $\overline{y}:\Spec(\overline{\kappa})\to Y$ defines  a single orbit of the action of $T_{\overline{\kappa}}=T\times_{\kappa} \Spec(\overline{\kappa})=\Spec(\overline{\kappa}[t_1,t_1^{-1},\ldots,t_k,t_k^{-1}])$ on $X_{\overline{y}}$.
  
  \begin{lemma}\label{good} Let  $\pi: X\to Y=X\sslash T$ be  a  good quotient of  integral schemes of a relatively affine action of the torus $T$. Then
$\pi$ is surjective.	 Moreover,  the inverse image $\pi^{-1}(Z)\subset X$of  a closed connected subscheme $Z\subset Y$ is  connected.
\end{lemma}
\begin{proof} The problem reduces to the affine situation
$\pi: X=\Spec(A)\to \Spec A^T$. Then the coaction of $T$ on $A$ determines the gradation $$A=\bigoplus_{\alpha\in \ZZ^n} A_\alpha {\bf t}^\alpha,$$
where $B=A_0$. Then for any prime ideal $p\subset A^T=A_0$, the extended ideal $p	A$ in $A$ is proper, and $p=pA\cap A_0$ is a contracted ideal. This implies that $\pi$ is surjective.

 Let  $I\subset A_0$ be  an ideal such that the scheme $\Spec(A_0/I)$ is connected. Suppose that for the ideal $$ IA=\bigoplus_{\alpha\in \ZZ^n} IA_\alpha {\bf t}^\alpha$$ of $A$ the space $\Spec(A/IA)$ is disconnected. Then there is a nontrivial  ring decomposition $A/IA=A'\oplus A''$, and $(A/IA)_0=A'_0\oplus A''_0$. Hence either the ring $A'_0=0$  or $A''_0=0$. Consequently, either $A'=0$ or $A''=0$, and the decomposition is trivial.
 
\end{proof}

\begin{lemma} \label{good2} The natural morphisms $$\pi_B: B\to B\sslash T_B\simeq X, \quad \pi_{B_+,Y}: B_+\to B_+\sslash T_B\simeq Y$$ are  good quotients.

\end{lemma}
\begin{proof} \begin{align*} &
\Spec(\cO_{B\sslash T_B})=\Spec(\cO_{B}^{T_B})=\Spec_X(\cA_{Y/X})^{T_B}=\Spec_X(\cO_X)=X  \\
&\Spec(\cO_{B_+\sslash T_B})=\Spec(\cO_{B_+}^{T_B})=\Spec_Y (\cC_{Y/X})^{T_B}=\Spec_Y(\cO_Y)=Y
\end{align*}

\end{proof}
\subsubsection{Exceptional divisors on $B=\Cox(Y/X)$} \label{Ex}
Let $\pi:X\to Y$ be a proper birational morphism of normal schemes.
Using the natural birational morphism $i_B: B_+\to B$, one can interpret the notion of the exceptional divisors of $B$.

Any exceptional divisor $E_i$ on $Y$ defines a canonical form $$F_i=t_i^{-1}={\bf t}^{E_i}\in \cO_Y(E){\bf t}^{-1}\subset \cO_{B}$$ on $Y$ of degree $E_i$ which vanishes on $V_Y(F)=E_i$. The form $t_i^{-1}$ also defines a regular homogenous function $t_i^{-1}$ on $B$ of degree $E_i$. Its divisor $D_i:=\divv_{B}(t_i^{-1})$ on $B$ determines the divisor $D_{i+}:=D_{i|B_+}=\divv_{B_+}(t_i^{-1})$ on $B_+$ which maps to $E_i$. 
 \begin{lemma} \label{Ex2}
The natural quotient morphism $\pi_{B_+,Y}: B_+\to Y$, (respectively $\pi_{B}: B\to X$) takes the exceptional divisors $D_{i+}=V_{B_+}(t_i^{-1})$ (respectively $D_i=V_{B}(t_i^{-1})$)
 surjectively onto $E_i$ (respectively surjectively to the center of $Z_{X}(\nu_i)=V_X(\cI_{\nu_i,1,X})$ of the valuation $\nu_i$). Moreover the induced morphism $D_{i+}\to E_i$ (resp. $D\to Z_{X}(\nu_i)$) is defined by the good quotient.
 \qed
 	
 \end{lemma}

\begin{proof}

\begin{align*} &D_{i+}=V_{B_+}(t_i^{-1})=\Spec_Y \cC_{Y/X}/t_i^{-1}\cC_{Y/X}=\\ &= \Spec_Y (\bigoplus (\cO_Y(E)/\cO_Y(E-E_i)){\bf t}^E)\quad \to \\\to &\quad\Spec_Y((\cC_{Y/X}/(t_i^{-1}\cC_{Y/X})_0=\Spec(\cO_Y/\cO_Y(-E_i))=E_i\end{align*} 
Thus the morphism $D_{i+}\to Y$ is defined by the good quotient and is surjective by Lemma \ref{good}.
The proof for the divisors $D_i$ is similar.
\end{proof}

\begin{definition} The divisors $D_i=V_B(t_i^{-1})$, respectively $D_{i+}=V_{B_+}(t_i^{-1})$  will be called the {\it exceptional divisors} of $B=\Cox(Y/X)\to X$, respectively of $B_+\to Y$.
 	
 \end{definition}
 
 \begin{lemma} \label{irreducible} The divisors $D_i=V_B(t_i^{-1})$   on $B$ and $D_{i+}=V_{B_+}(t_i^{-1})$ on $B_+$ are irreducible.
\end{lemma}

\begin{proof} By Lemma \ref{good}, the divisors $D_i$ are connected, so it suffices to show that they are locally irreducible.
We can assume that $X$ is affine.
It suffices to show that $$t_i^{-1}={\bf t}^{E_i}\in \cO(B)=\cA_{Y/X}(X)$$ is  a prime element. The latter can be verified for the homogenous elements. Let $\nu_i$ be the valuation on $X$ associated with $E_i\subset Y$.

Let $$t_i^{-1}=t^{E_i}|({\bf t}^E\cdot f)({\bf t}^{E'} \cdot g)={\bf t}^{E+E'}fg $$ where $f\in \cO_X(E)$, and $g\in \cO_X(E')$, and suppose $t_i^{-1}$ does not divide both $({\bf t}^E\cdot f)$, and $({\bf t}^{E'} \cdot g)$.
The first assumption implies that ${\bf t}^{E+E'-E_i}fg\in \cO(B)$. So $fg\in \pi_*(\cO_Y(E+E'-E_i))$.  

Write  the presentations $E=\sum n_jE_j$ and 
$E'=\sum n'_jE_j$. Then, by the assumption  $\nu_i(f)=n_i$ and $\nu_i(g)=n'_i$.
Thus, by Proposition \ref{valu}, and the assumptions on $f$ and $g$, we have
$$\nu_i(fg)>n_i+n'_i=\nu_i(f)+\nu_i(g),$$ which is a contradiction since $\nu_i$ is a valuation. 
The  same reasoning works for $D_{i+}$. 
\end{proof}


\subsubsection{Morphisms of Cox spaces}
The following result is analogous to \\
\cite{Cox2}[Construction 1.6.3.1] for the Cox spaces of varieties.
\begin{proposition} 
Let $\pi: Y\to X$ is a proper birational morphism of normal schemes, and ${T_B}:=\Spec(\ZZ[\Cl(Y/X)])$.  Let $E$ be the exceptional divisor with the components $E_i$.
Denote by $U_\pi:=Y\setminus E\subset Y$ the open subset of $Y$, which can be identified  with the open subset of $X$, where $Y\to X$ is an isomorphism. Let $\pi_B: B \to X$,  and  $\pi_{B_+,Y}:  B_+ \to Y$ be the natural projections.

There is a natural  ${T_B}$-equivariant birational morphism $$ i_B:B_+=\Cox(Y/X)_+\to B=\Cox(Y/X).$$
over $X$, which is an isomorphism over $U_\pi$, with $$\pi_B^{-1}(U_\pi)=\pi_{B_+,Y}^{-1}(U_\pi)=U_\pi\times {T_B},$$ and such that $\pi_B^{-1}(D_i)=D_{i+}$.

Moreover, the morphism $i_B$ induces the  morphism of the good quotients: 
$$\pi: B_+\sslash {T_B}=Y\,\, \to \,\, B\sslash {T_B}=X.$$

\end{proposition}
\begin{proof}


For any open affine $U\subset X$, we have the natural identifications 
$$\Gamma(U,\pi_*(\cC_{Y/X}))=\Gamma(\pi_B^{-1}(U),\cO_{B})$$
and
$$\Gamma(U,\pi_*(\cC_{Y/X}))=\Gamma(\pi^{-1}(U),\cC_{Y/X})=\Gamma(\pi_{B_+,Y}^{-1}(\pi^{-1}(U)),\cO_{B_+})$$Combining both equalities gives us:
$$\Gamma(\pi_B^{-1}(U),\cO_{B})=\Gamma(\pi_{B_+,Y}^{-1}(\pi^{-1}(U)),\cO_{B_+}).$$ 

Since  
$\pi_B^{-1}(U)\subset \Cox(Y/X)$ is affine we obtain a natural morphism $$\phi_U: \pi_{B_+,Y}^{-1}(\pi^{-1}(U))\to \pi_B^{-1}(U)$$ over $U$ induced by the isomorphisms on global sections. The constructed morphisms are functorial for open embeddings $U\subset V$ of affine subsets on $X$ and glue  to a global morphism $B_+\to B$. 

The morphism $B_+\to B$ is birational as it is an isomorphism over $U_\pi\subset X$. Moreover  $i_B$ is an isomorphism over $U_\pi$: $$\pi_B^{-1}(U_\pi)=\pi_{B_+,Y}^{-1}(U_\pi)=\Spec_{U_\pi}(\bigoplus_{E\in \Cl(Y/X)}\cO_{U_\pi}{\bf t}^E)=
U_\pi\times {T_B}.$$

By the construction,   $$\pi_B^{-1}(D_i)=\pi_B^{-1}(V_B(t_i^{-1}))=V_{B_+}(t_i^{-1})=D_{i+}.$$

Locally  for any open affine $V\subset \pi^{-1}(U)$ the induced homomorphisms $$\cO_B(\pi_B^{-1}(U))=\Gamma(U,\pi_*(\cC_{Y/X}))=\Gamma(\pi^{-1}(U),\cC_{Y/X})\to \Gamma(V,\cC_{Y/X})=\cO_{B_+}(\pi_{B_+,Y}^{-1}(V))$$ determine the homomorphisms 
$$(\cO_B(\pi_B^{-1}(U)))^T=\Gamma(U,\pi_*(\cC_{Y/X})^T)\to \Gamma(V,\cC_{Y/X}^T)=\cO_{B_+}((\pi_{B_+,Y}^{-1}(V)))^T,$$ and  define the global  morphism $B_+\sslash T_B=Y\to B\sslash T_B=X$.

\end{proof}
\subsubsection{Cobordization}
\begin{definition} Let $\pi: Y\to X$ be  a proper birational morphism. Then the morphism $\pi_B:B=\Cox(Y/X)\to X$, (respectively $\pi_{B_+,Y}:B_+=\Cox(Y/X)_+\to X$) will be called the {\it full cobordization of $\pi$} (respectively the {\it cobordization} of $\pi$). 

If $\cI$ is an ideal on $X$, then by the {\it full cobordant blow-up $\sigma: B\to X$ at $\cI$} (respectively  cobordant blow-up $\sigma_+: B_{+}\to X$ at $\cI$ we mean the full cobordization (respectively cobordization) of the normalized blow-up $bl_\cJ(X)\to X$.

\end{definition}

\subsubsection{The Cox trivial space}

\begin{lemma} Let $\pi: Y\to X$ be a proper birational morphism of normal schemes, and $E=\bigcup E_i$ be its exceptional divisor. Let $U_\pi=Y\setminus E\subset X$  be the maximal open subset of $X$ and of $Y$ where $\pi$ is an isomorphism exactly. Then the Cox trivial space is 
$B_-=X\times {T_B}$. Moreover we have $$i_B^{-1}(B_-)=B_+\times_BB_-=U_\pi\times {T_B}.$$
	
\end{lemma}
\begin{proof}
By Proposition \ref{valu} we have\begin{align*} & B_-:=B\setminus \bigcup_{i=1}^k D_i= B\setminus V(t_1^{-1}\cdot\ldots\cdot t_k^{-1})=\Spec(\cO_X[ t_1,t_1^{-1},
\ldots,t_k,t_k^{-1}])=X\times {T_B}
\\
& B_+\times_BB_-:=B_+\setminus \bigcup_{i=1}^k D_i= B_+\setminus V_{B_+}(t_1^{-1}\cdot\ldots\cdot t_k^{-1})=
\\& =\Spec(\bigoplus_{E\in \Cl(Y/X)} \cO_{Y}(E){\bf t}^E)[ t_1,t_1^{-1},
\ldots,t_k,t_k^{-1}])=\Spec(\cO_{U_\pi}[ t_1,t_1^{-1},
\ldots,t_k,t_k^{-1}])
\end{align*}

\end{proof}
 \subsection{Open immersion of Cox spaces}
\subsubsection{Generating forms}



\begin{lemma} \label{cover3} Let $X$ be an affine scheme  and $\pi: Y\to X$ be a  proper birational morphism of normal integral schemes.
Assume that $Y_{F}$ is affine, for a certain  form $F=f{\bf t}^{-E}$ on $Y$, with $f\in H^0(Y,\cO_Y(-E))=H^0(X,\pi_*(\cO_Y(-E))$. Then $(B_+)_F=B_F$ is affine
and 
$\pi_{B_+,Y}^{-1}(Y_{F})=(B_+)_F$. 
Moreover $$C_{Y/X}(Y_{F})=(C_{Y/X}(Y))_{F}=(H^0(B,\cO_B))_F.$$
\end{lemma}
\begin{proof}

If $y\in Y_{F}$ then $F=f{\bf t}^{-E}$ is invertible in the stalk $(\cC_{Y/X})_y$. 
 Indeed $\divv(f{\bf t}^{-E})=0$ at $y$ so $E=\divv(f)$ is principal at $y$, and thus $(f{\bf t}^{-E})^{-1}=f^{-1}{\bf t}^E$ is the inverse of $f{\bf t}^{-E}$. This shows that the form $F$ is invertible in $\cC_{Y/X}(Y_F)$, and the function $F$ is invertible on the scheme $\pi_{B_+,Y}^{-1}({Y}_F)\subset B_+$.  Thus we have an open immersion $\pi_{B_+,Y}^{-1}(Y_{F})\hookrightarrow B_{+F}$. Since $F$ is invertible on $\pi_{B_+,Y}^{-1}(Y_{F})$ the natural homorphism $\cC_{Y/X}(Y)\to \cC_{Y/X}(Y_F)$ factors through the localization $(\cC_{Y/X}(Y))_F\to \cC_{Y/X}(Y_F)$.

On the other hand if $$G=g{\bf t}^{-E'}\in C_{Y/X}(Y_F)$$ is a form on $Y_F$ then, by definition, $$\divv_Y(G\cdot F^n)=\divv_Y(G)+n\cdot \divv_Y(F)\geq 0$$ on $Y$  for sufficiently large $n$.
 Hence $G\cdot F^n\in C_{Y/X}$. This shows that $(\cC_{Y/X}(Y))_F\to \cC_{Y/X}(Y_F)$ is  surjective. But this morphism is 
defined by the restrictions of forms, so functions on open subsets of $B_+$, and thus it is also injective. Hence it is an isomorphism.

This defines an isomorphism of the global sections 
\begin{align*} (\cC_{Y/X}(Y))_F= H^0(B_+,\cO_{B_+})_F=H^0((B_+)_F,\cO_{B_+})\\\to H^0(\pi_{B_+,Y}^{-1}(Y_{F}),\cO_{B_+})=\cC_{Y/X}(Y_F) 
\end{align*}

If $Y_F$ is affine then we obtain then $\pi_{B_+,Y}^{-1}(Y_{F})$ is also affine, and the open immersion $\pi_{B_+,Y}^{-1}(Y_{F})\hookrightarrow (B_+)_F$ has the left inverse $(B_+)_F\to \pi_{B_+,Y}^{-1}(Y_{F})$ determined by the global sections. Since the schemes are separated, it is an isomorphism.


 Finally we observe that $H^0(X,\pi_*(\cO_Y(E))=H^0(Y,\cO_Y(E)$. Hence 
 $$H^0(Y,\cC_{Y/X})=H^0(X,\pi_*(\cC_{Y/X}))=H^0(X, \cA_{Y/X})=H^0(B,\cO_B),$$
and $$H^0(B_F,\cO_B)=\cA_{Y/X}(X))_F=\cC_{Y/X}(Y))_F=\cC_{Y/X}(Y_F).$$




\end{proof}



\subsubsection{Irrelevant ideal and open immersion of Cox spaces}
The notion of {\it irrelevant ideals} was used in \cite{Cox2} in the context of Cox rings. Here we consider the analogous definition and results for morphisms.

\begin{proposition}
 \label{cover2}
Let $\pi: Y\to X$ be a proper birational morphism of normal schemes. Assume that $X$ can be covered by open subsets $X_i$ such $Y_i:=\pi^{-1}(X_i)$ admits an open affine cover $(Y_i)_{F_j}$, where $F_{ij}=f_{ij}{\bf t}^{-E_{ij}}$ is a   form on $Y_i$ for $f_{ij}\in \cO_{Y_i}(-E_{ij})$. Then
there is a natural open ${T_B}$-equivariant embedding $$B_+=\Cox(Y/X)_+\hookrightarrow B=\Cox(Y/X),$$

It induces the  morphism of the good quotients: 
$$B_+\sslash {T_B}=Y\,\, \to \,\, B\sslash {T_B}=X.$$

Moreover  $
B\setminus B_+$ is of codimension $\geq 2$ in $B$.
\end{proposition}
\begin{proof}  The problem is local on $X$, so we can replace $X$ with $X_i$, and drop the subscripts $i$.
By Lemma \ref{cover3}, the open  affine cover $Y_{F_j}$ of $Y$  where $F_j\in \cI_{\irr}$ defines the open affine cover  $B_{+F_i}=\Spec_Y((\cC_{Y/X})_{F_i})=\pi_{B_+,Y}^{-1}(Y_i)$  of  $B_+$ mapping it isomorphically onto open subsets $B_{F_i}=\Spec_X((\cA_{Y/X})_{F_i})\subset X$. This induces the open immersion $$B_+
\hookrightarrow B.$$

For ``moreover part'' let $U_\pi=Y\setminus E \subset Y$ be the maximal open subset, where $\pi: Y\to X$ is an isomorphism. Then $U_\pi$ can be identified with an open subset of $X$, and 
 the complement $X\setminus U_\pi$ of the open set is of codimension $\geq 2$, and
$$B_+ \setminus D=U_\pi\times  {T_B}\subset B_-=B \setminus D=X\times {T_B}.$$
if of codimension $\geq 2$ in $B_-=B\setminus D$.

 On the other hand, by Lemma \ref{irreducible}, the divisors $D_i=V_B(t_i^{-1})$ are irreducible on $B$. 
 
 Consequently 
 the difference $D_i\setminus B_+=D_i\setminus D_{i+}$ is of codimension $\geq 2$. Thus $$B\setminus B_+=(B_-\setminus B_+)\cup (D\setminus D_+)$$
 is of codimension $\geq 2$ in $B$.
\end{proof}

The notion of {\it irrelevant ideal} on Cox coordinate spces was originally introduced in \cite{Cox2} (Definition 1.6.3.2 and Proposition 1.6.3.3(iii))
\begin{definition}
By the  the {\it irrelevant ideal} $\cI_{\irr}\subset \cA_{Y/X}$ we mean  the ideal radically generated by the  forms $F$ in $\cA_{Y/X}$, such that  $Y_F$ is open affine over $X$.
\end{definition}
\begin{corollary} Under the conditions from Proposition \ref{cover2}, $\cI_{\irr}$  is the radical coherent ideal determined by the reduced closed subscheme $B\setminus B_+$. Thus we can write 
$B_+=B\setminus V(\cI_{\irr})$.	

\end{corollary}
\begin{proof} The problem is local on $X$, and we can assume that $X$ is affine. It follows from the construction that $B_+=B\setminus V(\cI)$, where $\cI$ is generated by all $F\in \cA_{Y/X}$, such that $Y_F$ is affine. Thus  ${\rm rad}(\cI)=\cI_{\irr}$.
	
\end{proof}

\subsubsection{Cox construction for regular schemes $X$}

Recall a well-known fact:
\begin{lemma}
Let $Y$ be a normal scheme. Then the complement of any open affine subset $V\subset Y$ is the support of a Weil divisor.

Thus there is a finite open cover of ${Y}$ by open affine subsets $V_i={Y}\setminus D_i$, where $D_i$ are Weil divisors on $Y$.
\end{lemma}
\begin{proof}
By definition, $V$ is the set of points of $Y$ where all the functions $f\in \Gamma(V,\cO_Y)\subset \kappa(Y)$ are regular. Since $Y$ is normal, this means that
the supports of the divisors $\divv_-(f)$ of the negative components of  $\divv(f)$ cover ${Y}\setminus V$. Consequently, $Y\setminus V$ is the union of the Weil divisors contained in it. Thus this union is finite, and $Y\setminus V$ is the support of the Weil divisor.

This defines an open cover $V_i=Y\setminus D_i$ which can be assumed to be finite.



\end{proof}

\begin{lemma} \label{cover} Let $\pi: Y\to X$ be a proper birational morphism of normal schemes.

 Let $p\in X$ be a regular point on $X$. There is an open  affine neighborhood $U$ of $p$ in $X$, and 
 an open cover of $Y_U=\pi^{-1}(U)$ by open affine subsets ${Y}_F={Y_U}\setminus V_Y(F)$, where $F$ is a  form over $U\subset X$ and on $Y_U\subset Y$.
\end{lemma}
\begin{proof} We can assume that $X$ is affine. By the previous lemma, 
we can find an open affine cover $$V_j:=Y\setminus (D_j\cup \overline{E}_j)$$ of $Y$ defined by the divisors $D_j\cup \overline{E}_j$, where $\overline{E}_j$ are some possibly reducible exceptional divisors. Taking the images of $D_j$ in $X$, we obtain a finite collection of divisors $D'_j=\pi(D_j)$ on $X$. Consider an open affine neighborhood $$U:=X_{g}=X\setminus V(g)$$   of $p\in X$, for $g\in H^0(X,\cO_X)$, such that all the divisors $D'_j$ are principal on $U$.  Thus we can write  $D'_j=\divv_U(f_j)$, where $f_j\in \cO(X)$.
	 
	 The pullbacks of the principal divisor $D'_j=\divv_U(f_j)$ on $U$  are of the form  $\pi^*(D_j')=D_j+E^j$ on $Y_U=\pi^{-1}(U)$, where $E^j=\sum n_{ij}E_i$ is an exceptional divisor, with $n_{ij}\geq 0$. They define the  forms $$F_j:=f_j{\bf t}^{-E^j+\overline{E}_j}$$ on $Y_U$ such that $$\divv_Y(F_j)=\divv_Y(f_j)-E^j+\overline{E}_j=D_j+\overline{E}_j.$$ and thus $V_Y(F_j)=D_j\cup\overline{E}_j$ on $Y_U$. Then $$Y_U\setminus V(F_j)=Y_U\setminus (D_j\cup \overline{E}_j)=(V_j)_{g}=V_j\setminus V_Y(g)$$ is an open affine cover of $Y_U=\pi^{-1}(U)=\pi^{-1}(X_g)$.
	 
	 
\end{proof}
\begin{remark} The lemma is valid under the assumption that $p\in X$ is a $\QQ$-factorial point,  so any Weil divisor at $p$ is  $\QQ$-Cartier.
	
\end{remark}

As a corollary from  the above, we obtain the following:

\begin{proposition} 
Assume that $X$ is  regular, and $\pi: Y\to X$ is a proper birational morphism of normal schemes.
There is a natural open ${T_B}$-equivariant embedding $$B_+=B\setminus V(\cI_{\irr})\hookrightarrow B$$

It induces the  morphism of the good quotients: 

$$B_+\sslash {T_B}=Y\,\, \to \,\, B\sslash {T_B}=X.$$

\end{proposition}

\subsection{Cobordant blow-ups of ideals}
\subsubsection{The strict and the weak transform under cobordant morphism}

\begin{definition} \label{strict} Let $\cI$ be any ideal on a normal scheme $X$. Let $\pi: Y\to X$ be a proper birational morphism from a normal scheme $Y$, and $\sigma=\pi_B: B\to X$ be the full cobordization of $\pi$. Then by the {\it strict transform} of the ideal $\cI$ we mean the ideal
$$\sigma^s(\cI):=(f\in \cO_B \mid {\bf t}^{-\alpha} f\in \cO_B\cdot \cI, \,\,\mbox{for some} \,\,\alpha\in \ZZ^k_{\geq 0})\subset \cO_B.$$
The {\it weak transform} of the ideal $\cI$  is given by
$$\sigma^\circ(\cI):={\bf t}^{\alpha_0} \cI,$$
where $$\alpha_0:=\max \{\alpha\mid \cI\subset {\bf t}^{-\alpha}\cO_X\},$$ is defined for the	partial componentwise order on the set of components.
\end{definition}

 \subsubsection{Cobordant blow-ups}
\begin{lemma} \label{blow} Let $\cJ$ be an ideal  on a normal scheme $X$, such that $\codim(V(\cJ)\geq 2$. Let $\pi: Y\to X$ be the normalized blow-up of $\cJ$. Let $E=\sum a_iE_i$ be the exceptional divisor of $\pi$, such that $\cO_Y(-E)=\cO_Y\cdot\cJ$. Set $\alpha=(a_1,\ldots,a_k)$. Denote  by
$\sigma: B\to X$ be the corresponding full cobordant blow-up of $\cJ$. Then  \begin{enumerate}
	
	\item  $\sigma^{-1}(X\setminus V(\cJ))=(X\setminus V(\cJ))\times {T_B}$ is trivial.
\item  $B_+=B\setminus V_B(\sigma^\circ(\cJ))=B\setminus V_B({\bf t}^\alpha\cJ)$, where $$\sigma^\circ(\cJ)=\cO_{B}\cdot {\bf t}^{-E}\cJ=\cO_{B}\cdot {\bf t}^\alpha\cJ$$ is the weak transform of $\cJ$.

	\item $\cJ\cdot\cO_{B_+}={\bf t}^{-\alpha}\cO_{B_+}$ is a locally principal monomial ideal on $B_+$. 
 
\end{enumerate}
\end{lemma}
\begin{proof}
 Let  $U\subset X$ be an open affine subset. The ideal of sections $\cJ(U)$ is generated by  some $f_1\ldots,f_k\in \cJ(U)\subset \cO_X(U)=\cO_Y(\pi^{-1}(U))$.
The pullbacks of the functions $f_1\ldots,f_k\in \cJ(U)$ generate the ideal $$\cI_E=\cO_Y(-E)=\cO_Y\cdot\cJ$$ on $Y_U:=\pi^{-1}(U)$. Moreover on each  $Y_U\setminus V_Y(F_i)$, where $F_i:=f_i{\bf t}^{-E}$ we have exactly $\divv_{Y_U}(f_i)=E_{|Y_U}$.  

On the other hand consider the open cover of $$Y_U=\pi^{-1}(U)=\Proj \bigoplus_{i=0}^\infty \cJ^i(U)t^i,$$
where $t$ is a dummy unknown by the open subsets 
$$(Y_U)_{f_it}=\pi^{-1}(U)_{f_it}=(\Proj \bigoplus_{i=0}^\infty \cJ^i(U)t^i)_{f_it}=(\Spec (\bigoplus_{i=0}^\infty \cJ^i(U))_{f_it})_0,$$ where $f_it\in \cJ^1(U)t$. Since $f_jt$ is invertible on $(Y_U)_{f_jt}$ and $f_it/f_jt=f_i/f_j$ are regular we wee that $\cO_{(Y_U)_{f_jt}}\cdot \cJ$ is generated by $f_j$. So $E=\divv_Y(f_j)$ on $(Y_U)_{f_it}$,  and consequently the form $F_j=f_j{\bf t}^{-E}$ is invertible on $(Y_U)_{f_jt}$.

Computing $\divv(f_it)$ on the cover $(Y_U)_{f_jt}$ of $Y_U$ gives us $$\divv(f_it)=\divv(f_it/f_jt)=\divv(f_i/f_j)=\divv(f_i)-E=\divv(F_i)=\divv(f_i{\bf t}^{-E}).$$
Consequently we conclude that $(Y_U)_{f_it}=(Y_U)_{F_i}$
  is affine and cover $Y_U$.	 Thus, by Proposition \ref{cover2}, there is an open immersion $B_+\subset B$, where $B_+$ is covered by $B_{+F_i}$. Moreover, by the above,   the ideal $\cJ {\bf t}^\alpha$ on $B_{+F_i}$ is generated by $f_i{\bf t}^{-E}$, and thus  equal to $$\cJ {\bf t}^\alpha_{|B_{+F_i}}=\cO_{B_{+F_i}}\cdot\cO_{Y_{F_i}}(-E){\bf t}^{-E}=\cO_{B_{+F_i}}\cdot f_i{\bf t}^{-E}=\cO_{B_{+F_i}}\cdot F_i.$$ But $F_i=f_i{\bf t}^{-E}$  is invertible on $B_{+F_i}$ of degree $-E$, whence 
  $$\cO_{B_{+F_i}}\cdot\cJ=\cO_{B_{+F_i}}\cdot\cJ {\bf t}^\alpha\cdot {\bf t}^{-\alpha}=\cO_{B_{+F_i}} {\bf t}^{-\alpha},$$
which implies that
\begin{align}\cO_{B_+}\cdot\cJ=\cO_{B_+} {\bf t}^{-\alpha}.\end{align}

On the other hand $\cJ {\bf t}^\alpha\subset \cO_B$, since any element $f{\bf t}^\alpha\in \cJ(U){\bf t}^\alpha$ is in \\ $\cO_{\pi^{-1}(U)}(-E){\bf t}^\alpha$ which is the $-E$ gradation of $\cO_B$ over $U$. This also shows that $\cJ {\bf t}^\alpha=\sigma^\circ(\cJ)$, as by equality (4) for $B_+$, the form ${\bf t}^{-\alpha}=t^{E}$ is the maximal factor which divides $\cO_B\cdot\cJ$.

Finally, by the above $$V(\sigma^\circ(\cJ))=V(f_1{\bf t}^{-E},\ldots,f_k{\bf t}^{-E})=V(F_1,\ldots,F_k)=B\setminus B_+=V(\cI_{\irr}).$$
\end{proof}

\begin{remark} It follows from the above that the inverse image $\cO_{B_+}\cdot \cJ$ of ideal $\cJ$ under the cobordant blow-up is the ideal of the exceptional divisor ${\bf t}^E$, analogously to the standard blow-up of $\cJ$. However, this is no longer true for the full cobordant blow-up $\cJ$.
	
\end{remark}

	
	

\subsection{Cobordant flips}
\begin{lemma} Let $\phi_1: X_1\to Z$, and $\phi_2:X_2\to Z$ be proper birational morphisms from normal schemes $X_1,X_2$ to   $Z$.
Assume that the induced proper birational map $X_1\dashrightarrow X_2$ over  $Z$ is an isomorphism in codimension one. Then  $$B:=B(X_1/Z)=B(X_2/Z),$$  is equipped with the action of torus ${T_B}=\Cl(X_1/Z)=\Cl(X_2/Z)$,
and there is a natural birational map $B(X_1/Z)_+\dashrightarrow B(X_2/Z)_+$ over $B$.
 Moreover if $\phi_1,\phi_2$ satisfy the condition of Proposition \ref{cover2},  then $B(X_1/Z)_+$ and $B(X_2/Z)_+$ are open subschemes of $B$ which coincide in codimension $1$.
\end{lemma}

\subsection{Functoriality of Cox spaces  for open immersions}
The construction of the full cobordization  is functorial for open immersions up to torus factors:
\begin{lemma} \label{open} Let $\pi: Y\to X$ be a proper birational morphism of normal integral schemes.
Let $U\subset Y$ be an open subset, and $Y_U:=\pi^{-1}(U)$.
 Let $E_1,\ldots,E_k$ be the irreducible exceptional divisors of $\pi: Y\to X$.
Let $\pi_B:B=\Cox(Y/X)\to X$ be the full cobordization of
a proper birational morphism $\pi: Y\to X$, and $\pi_{B+}:B_+\to X$ is its cobordization.   Let $$T_{B\setminus B_U}:=\Spec(\,\ZZ[t_i, t_i^{-1}\mid E_i\subset Y\setminus Y_U \,\,]\,),$$

Then $$B_U:=\pi_B^{-1}(U)=B(Y_U/U)\times T_{B\setminus B_U},\quad B_{U+}:=\pi_{B+}^{-1}(U)=B(Y_U/U)_+\times T_{B\setminus B_U}.$$  \end{lemma}

\begin{proof} For any open subset $U\subset X$, and $Y_U=\pi^{-1}(U)$, we can construct a subgroup $\Cl(Y_U/U)\subseteq \Cl(Y/X)$, with the canonical splitting	 $ \Cl(Y/X)\to \Cl(Y_U/U)$. Write $\Cl(Y/X)=\Cl(Y_U/U)\oplus \Cl^0(Y_U/U)$, where $\Cl^0(Y_U/U)$ is generated by $E_i\subset Y\setminus Y_U$.
\begin{align*} &\pi_B^{-1}(U)=\Spec_U(\bigoplus_{E\in \Cl(Y/X)} \pi_*(\cO_Y({E})_{|U}){\bf t}^E=\\ & \Spec_U(\bigoplus_{E\in \Cl(Y_U/U)} \pi_*(\cO_Y({E}))_{|U} {\bf t}^E)\otimes_{\cO_U}(\bigoplus_{E\in \Cl^0(Y_U/U)} \pi_*(\cO_Y({E}))_{|U}{\bf t}^E)\\ & \Spec_U(\bigoplus_{E\in \Cl(Y_U/U)} \pi_*(\cO_Y({E}))_{|U} {\bf t}^E)\otimes_{\cO_U}(\bigoplus_{E\in \Cl^0(Y_U/U)}\cO_U{\bf t}^E)=B(Y_U/U)\times T_{B\setminus B_U}
	\end{align*}
Similarly
 \begin{align*} &\pi_{B+}^{-1}(U)=\Spec_{Y_U}(\bigoplus_{E\in \Cl(Y/X)} \cO_{Y_U}({E}){\bf t}^E)=\\ & \Spec_{Y_U}(\bigoplus_{E\in \Cl(Y_U/U)} \cO_{Y_U}({E}){\bf t}^E) \otimes_{\cO_{Y_U}}(\bigoplus_{E\in \Cl^0(Y_U/U)} \cO_{Y_U}({E}){\bf t}^E)\\ & \Spec_{Y_U}(\bigoplus_{E\in \Cl(Y_U/U)} \cO_{Y_U}({E}){\bf t}^E) \otimes_{\cO_{Y_U}}(\bigoplus_{E\in \Cl^0(Y_U/U)} \cO_{Y_U}{\bf t}^E)\\
	&=B(Y_U/U)_+\times_{Y_U}(Y_U\times T_{B\setminus B_U})=B(Y_U/U)_+\times T_{B\setminus B_U}
	\end{align*}

\end{proof}

\section{Relative Cox construction for  toric morphisms }
\subsection{Toric varieties}
Recall some basic properties of toric varieties over a field. (See \cite{KKMS}, \cite{Oda},
\cite{Dan}, \cite{Fulton}).
\subsubsection{Fans} \label{fans} Let $\kappa$ be a field, and let $$T=\Spec(\kappa[x_1,x_1^{-1},\ldots,x_k,x_k^{-1}]=\Spec(\kappa[\MM])$$ be the torus, where
$\MM=\Hom(T,G_m)\simeq \ZZ^k$. The elements of $\MM$ can be described by the Laurent monomials $x^\alpha\in \MM$, where $\alpha\in \ZZ^k$.

Denote by  $\NN:=\Hom(G_m,T)$ the group of algebraic homomorphisms $t\to {\bf t}^\beta=({\bf t}^{b_1},\ldots,{\bf t}^{b_k}).$

This determines a nondegenerate pairing $(\cdot,\cdot)$  $\NN\times \MM\to \ZZ$ defined by the composition:
$$\Hom(G_m,T)\times \Hom(T,G_m)\to \Hom(G_m,G_m)
\quad,  x^\alpha \circ {\bf t}^\beta= {\bf t}^{(\beta,\alpha)}.$$
Thus $N=M^*\simeq Hom(M,\ZZ)$ is dual to $M$.

By a {\it fan} $\Delta$ in $\NN_\QQ$, we mean a collection of strictly convex cones, which is closed under the face relation, and such that two cones intersect along the common face. If $\tau$ is a face of $\sigma$, written as $\tau\leq \sigma$ then $X_\tau\subset X_\sigma$ is an open immersion.

\subsubsection{Toric varieties from fans} \label{fans2}
With any rational strictly convex cone $\sigma$ in $\NN_\QQ=\NN\otimes \QQ$  we associate its dual
$$\sigma^\vee:=\{y\in \MM_\QQ \mid (x,y)\geq 0 \quad\mbox{for all}\quad x\in \sigma\}.$$
The cone $\sigma^\vee$ determnies the  monoid $P_\sigma := \sigma^\vee\cap M$, and the relevant  affine toric  variety $X_\sigma=\Spec(\kappa[P_\sigma])$.

We say that a cone
$\sigma$ in $\NN^{\QQ}$ is {\it regular} or {\it nonsingular} if it is generated by a part of a basis of the lattice
$e_1,\ldots,e_k\in \NN$, written $$\sigma=\langle e_1,\ldots,e_k\rangle:=\QQ_{\geq 0}e_1+\ldots+\QQ_{\geq 0}e_k.$$ Similarly a  cone $\sigma=\langle v_1,\ldots,v_k\rangle$ in $\NN^{\QQ}$ is {\it simplicial}  it  if it generated by a linearly independent set
$\{v_1,\ldots,v_k\}\in \NN$.

With a fan $\Sigma$ we associate the {\it toric variety $X_\Sigma$} obtained by glueing $X_\sigma$, where $\sigma \in \Sigma$, along $X_\tau$, where $\tau\leq \sigma$. 
The torus $T=\Spec(\kappa(M))$ acts on toric variety $X_\Sigma$ with an open dense  orbit $T=\Spec(\kappa(M))$ corresponding to $\{0\}\in \Sigma$.

The fan $\Sigma$ will be called {\it regular} (respectively 
{\it simplicial}) if all its cones are { regular} (respectively 
{simplicial}). 

The regular (resp. simplicial) fans $\Sigma$ are in the bijective correspondence with the smooth (resp. $\QQ$-factorial) toric varieties $X_\Sigma$.

For any $r\in \ZZ_{\geq 0}$ by $\Sigma(r)$ denote the set of cones $\sigma$ of dimension $r$ in $\Sigma$.  The cones in $\Sigma(r)$ correspond to the orbits $\cO_\sigma$ and thus to the irreducible $T$-stable  closed subvarieties $\cO_\sigma$. In particular, the irreducible $T$-stable 
 divisors correspond to the one-dimensional faces in $\Sigma(1)$.
 \subsubsection{Toric valuations}
 Any integral vector $v\in \NN$ determines a monomial valuation $\val(v)$, which can be defined for $f=\sum c_m\cdot m \in \kappa[\MM]$, as $$\val(v)(f)=\val(v)(\sum c_m\cdot m)= \min_{c_m\neq 0} (v,m).$$
 
 The center $Z_{\val(v)}$ of the valuation $\val(v)$ is the union of orbits $\cO_\tau$, which correspond to the cones $\tau$ in $$ \Star(\tau,\Sigma)=\{\tau\mid \sigma\leq \tau\}.$$ 
 
 The associated ideals on $X_\Sigma$ are given locally on $X_\sigma$ as 
$$\cI_{\val(v),a,X_\sigma}=(m\in P_\sigma\mid (v,m)\geq a).$$

 By a {\it vertex} of $\Sigma$, we mean the {\it primitive vector}, so the integral vector with relatively coprime coordinates, which lies in a  one-dimensional face of $\Sigma$. The set of vertices of $\Sigma$ will be denoted by $\Ver(\Sigma)$. Each vector $v\in \Ver(\Sigma)$ defines the one-dimensional face $\langle v \rangle$, and the valuation $\val(v)$, which is precisely the valuation of the associated $T$-stable irreducible divisor $D$.
 
 

\subsubsection{Decomposition of fans}
 
By the {\it support of a fan} $\Sigma$ we mean the union   of its cones $|\Sigma|=\bigcup_{\sigma\in \Sigma} \sigma$. 

The {\it decomposition} of the fan $\Sigma$ is a fan $\Sigma'$ such that any cone $\sigma'\in \Sigma'$ is contained in $\sigma\in \Sigma$, and $|\Sigma'|=|\Sigma|$.

 For any subset  $\Sigma_0$ of the fan $\Sigma$, denote by $\overline{\Sigma_0}$ the set of all faces of the cones in $\Sigma_0$.
  The typical examples of the decompositions are given by the {\it star subdivisions}.
\begin{definition}\label{de: star subdivision} Let $\Sigma$ be a fan and $v$ be a primitive vector
 in the
relative interior of  $\tau\in\Sigma$. Then the {\it star
subdivision} 
 $v\cdot\Sigma$ of $\Sigma$ at
$v$ is defined to be
$$v\cdot\Sigma=(\Sigma\setminus {\rm Star}(\tau ,\Sigma) )\cup
\{ \langle v\rangle +\sigma \mid   \sigma\in \overline{\Star(\tau
,\Sigma)}\setminus \Star(\tau
,\Sigma)\}.$$ 
The vector $v$ will be called the {\it center} of the star subdivision.
\end{definition} 
\begin{lemma} The decompositions  $\Delta$ of a fan $\Sigma$ are in bijective correspondence with the  proper birational $T$-equivariant morphisms $X_\Delta\to X_\Sigma$.

The star subdivision $v\cdot\Sigma$ corresponds to  the blow-up of the valuation, which is the normalized blow-up of $\cI_{\val(v),a,X_\Sigma}$ for a  sufficiently divisible $a$.
	
\end{lemma}

\subsubsection{Maps of fans}
By {\it a map of  fans} $(\Sigma',\NN')\to (\Sigma,\NN)$ we mean a linear map $\phi: \NN'\otimes\QQ\to \NN\otimes\QQ$ of vector spaces, such that
\begin{enumerate}
\item $\phi(\NN')\subset \NN$.
\item For any $\sigma'\in \Sigma'$ there is is $\sigma\in \Sigma$ such that $\phi(\sigma')\subset \sigma$.	
\end{enumerate}
  
 The map of fans corresponds to a $T_{\NN'}$-equivariant morphism of toric varieties $(X_{\Sigma'},T_{\NN'})\to (X_\Sigma,T_{\NN})$, where the action of $T_{\NN'}=\Spec \kappa[\MM']$ on $X_\Sigma$ is defined by the  homomorphism of tori $$T_{\NN'}=\Spec \kappa[\MM']\to T_{\NN}=\Spec \kappa[\MM],$$ induced   by  $\NN'\to \NN$. The decomposition $\Sigma'$ of a  fan $\Sigma$ corresponds to the proper birational morphism. 

\subsubsection{Good quotients}
 Let $\phi: (\sigma',\NN')\to (\sigma,\NN)$ be a surjective map of cones, such that $\phi(\sigma')=\sigma$, and $\phi(\NN')=\NN$. Let $\NN'':=\ker(\NN'\to \NN)$. Then 
the exact sequence 
$$0\to \NN''\to \NN'\to \NN\to 0,$$
has its dual
$$0\to \MM\to \MM'\to \MM''\to 0.$$

Thus $\MM$ can be identified with the sublattice of $M''$ defined as $$\MM=\{m\in M'' \mid (n,m)=0\quad \mbox{for \,\, all}\,\, n\in \NN'' \}$$ 
Consequently, $\kappa[\MM]=\kappa[\MM'']^{T_{\NN''}}$.
Moreover the dual map determine the inclusion $\sigma^\vee\hookrightarrow (\sigma')^\vee$ for which 
$(\sigma')^\vee\cap \MM_\QQ=\sigma^\vee $, and 
$$(P_{\sigma'})^{T_{\NN''}}=P_{\sigma'}\cap \MM=P_{\sigma}.$$
 Hence $$\cO(X_{\sigma'})^{T_{\NN''}}=\kappa[P_{\sigma'}]^{T_{\NN''}}=\kappa[P_{\sigma}]=\cO(X_{\sigma}).$$
Thus $$X_{\sigma'}\to X_\sigma\simeq X_{\sigma'}\sslash T_{\NN''}$$ is  {\it a good quotient}.

If additionally $\phi: \sigma'\to \sigma$ is injective, so it is an isomorphism of cones, then the inverse image of any  orbit is a single orbit, and thus the corresponding morphism $X_{\sigma'}\to X_\sigma\simeq X_{\sigma'}\slash T_{\NN''}$ is {\it a geometric quotient}.

If the map of fans $\phi: (\Sigma',\NN')\to (\Sigma,\NN)$  is surjective, i.e. $\phi(|\Sigma'|)=|\Sigma|$ and $\phi(\NN')=\NN$, and  for any cone $\delta\in \Sigma$, the inverse image $\phi^{-1}(\delta)\cap |\Sigma'|$ is a unique cone $\delta'\in \Sigma'$, then the corresponding morphism $X_{\Sigma'}\to X_\Sigma$ is affine.
Consequently, by the previous argument,  it 
is a good quotient with respect to $T_{N''}=\ker(T_{\NN'}\to T_{\NN})$, where  $\NN''=\ker(\NN'\to \NN)$.

If additionally, the map $|\Sigma'|\to |\Sigma|$ is bijective then $X_{\Sigma'}\to X_\Sigma$ is a geometric quotient.

\subsection{Cox construction for toric varieties}

We recall here the standard Cox construction  for toric varieties from the convex geometry point of view. This presentation relies greatly on \cite{Cox}, \cite{Cox2}, and will be then adapted to the relative situation.
\subsubsection{Cox construction}
Given a  toric variety $X$ with associated  fan $\Sigma$ in the space $\NN^\QQ\simeq \QQ^n$ containing the standard lattice $\NN\simeq \ZZ^n$. 
 We shall assume that the fan $\Sigma$ is {\it nondegenerate} that is the set {\it  $\Ver(\Sigma)$ generate the vector space $\NN_\QQ$.}

 Let $\Ver(\Sigma)=\{v_1,\ldots,v_k\}$ denote the set of  vertices of $\Sigma$. Let $e_1,\ldots,e_k$ denote the standard basis of $\ZZ^k\subset \QQ^k$, and let $$\sigma_{B}:=\langle e_1,\ldots,e_k\rangle=\{\sum_{i=1}^k a_iv_i\mid a_i\in \QQ_{\geq 0}\}.$$
The cone $\sigma_{B}$ defines a regular fan $\Sigma_{B}$ in $\NN_B^\QQ=\QQ^k$, consisting of all the faces of $\sigma_B$. It corresponds to the affine space $$X_{\sigma_{B}}=\Spec (\kappa[x_1,\ldots,x_k])=\AA^k_\kappa$$
Consider the  linear map $\pi_B:\NN_B^\QQ=\QQ^k\to \NN^\QQ=\QQ^n$ defined on the basis $e_1,\ldots,e_k$, such that $\pi_B(e_i)=v_i$. 
We construct the subfan $\Sigma_{B_+}$  of $\sigma_{B}$ to be the set of all the faces $\sigma$ of $\sigma_{B}$ such that $\pi_B(\sigma)$ is contained in a face of $\Sigma$ (\cite{Cox2}).  This determines a morphism
$\pi_B: \Sigma_{B_+}\to \Sigma$. Note that it follows from the definition that for any face $\delta=\langle v_{i_1},\ldots v_{i_k}\rangle$ of $\Sigma$, there is a unique face $$\delta_0=\pi_B^{-1}(\delta)=\langle e_{i_1},\ldots e_{i_k}\rangle\in \Cox(\Sigma).$$
\subsubsection{Cox coordinate ring}
Let $\Div(X)$ be the group of Weil divisors on  $X=X_\Sigma$, and $\Div(X)_+$ be the monoid of the effective Weil divisors and zero on  $X$.  Let $\Ver(\Sigma)=\{v_1,\ldots,v_k\}$ denote the set of  vertices of $\Sigma$. The corresponding Weil divisors $D_1,\ldots,D_k\in \Div(X)$  freely generate $\Div(X)$. 

\begin{definition} \cite{Cox}
The {\it Cox coordinate ring} is defined to be $$\cC(X):={\kappa}[x_1,\ldots,x_k]={\kappa}[\Div(X)_+]=\bigoplus_{D\in \Div(X)_+} {\kappa} x^D,$$
with the natural identification  $x_i=x^{D_i}$, and $x^D=x^\alpha$ and the induced multiplication $x^{D_1}\cdot x^{D_2}=x^{D_1+D_2}$.
\end{definition}

Denote by  $\Prin(X)$ the subgroup  of $\Div(X)$ of the principal divisors on $X$, which is generated by $\divv(m)$, where $m\in M$, giving an isomorphism $$M \simeq \Prin(X),\quad m\mapsto \sum (v_i,m) D_i.$$ 
We use here the assumption that $\Sigma$ is nondegenerate.

 Let $\Cl(X)=\Div(X)/\Prin(X)$ be the 
 Weil divisor class group.
Although the  Cox coordinate ring, as defined, comes with the natural $\Div(X)$-gradation, one can also consider its 
$\Cl(X)=\Div(X)/\Prin(X)$-gradation. Then for any class $[E]\in \Cl(X)$ of the divisor $E\in  \Div(X)$  the space of effective Weil divisors in $[E]$ on $X$ is $T$- stable and thus generated by all $T$- invariant effective divisors $$E+\divv(m)\geq 0.$$  
Thus one can describe the $[E]$- gradation to be $$\cC(X)_{[E]}=\bigoplus_{D\in [E]} {\kappa}\cdot x^D= \bigoplus_{m\in M, \divv(m)+E\geq 0} {\kappa} x^E\cdot x^{\divv(m)} \simeq H^0(X,\cO_X(E))\cdot x^E,$$

Thus choosing any set $E_1,\ldots E_k$ of $\Div(X)$ which determines a basis of the lattice $\Cl(X)$, one identifies $\Cl(X)$, with the subgroup of $\Div(X)$. Under this noncanonical identification we can write as in  \cite{Cox} and \cite{Cox2}:
$$\cC(X)=\bigoplus_{E\in \Cl(X)} H^0(X,\cO_X(E))\cdot x^E$$

On the other hand the canonical $\Cl(X)$-gradation on $\cC(X)$ determines the natural action of the torus $$T_X:=\Spec({\kappa}[\Cl(X)])\simeq \Spec({\kappa}[t_1,t_1^{-1},\ldots,t_r,t_r^{-1}],$$
where $\Cl(X)\simeq \ZZ^r$.
\subsubsection{Cox coordinate space}
The Cox coordinate ring defines {\it the Cox coordinate space} (as in \cite{Cox} and \cite{Cox2}) to be
$$B=\Cox(X):=\Spec(\cC(X))=\Spec(\bigoplus_{E\in \Cl(X)} H^0(X,\cO_X(E))\cdot x^E)\simeq \AA^k,$$

It is the toric variety associated with the fan ${\Sigma_{B}}$ of all the faces of $\sigma_B$.
\subsubsection{Good and geometric quotients}
Let $$B_+=\Cox(X)_+:= X_{\Sigma_{B_+}}\subset B$$ be the open toric subscheme of $B$ associated with $\Sigma_{B_+}$.
The subscheme $B_+$  is called the {\it Cox characteristic space}. 
The morphism $B_+\to X$ corresponding to $\Sigma_{B_+}\to \Sigma$ is  toric and affine. It defines  the homomorphism of the relevant tori $$\phi: T_B:=\Spec({\kappa}[\Div(X)])\to T:=\Spec({\kappa}[M]),$$  
corresponding to the inclusion $M\hookrightarrow \Div(X)$ and  defining
the exact sequence 
 $$0\to M \to \Div(X)\to \Cl(X)\to 0.$$
Consequently, the kernel of $\phi$ can be identified canonically with $T_X:=\Spec({\kappa}[\Cl(X)])$. Since $T_X$ acts trivially on $T\subset X$, the morphism $B_+\to X$ is $T_X$-invariant and affine.  Moreover for any $\delta\in \Sigma$, and $\delta_0=\pi^{-1}(\delta)$, we have that $\pi(\delta_0)=\delta$, and $X_{\delta_0}\sslash T_X=X_\delta$. Thus, the  affine $T_X$-invariant morphism   $B_+\to X$ is a good quotient.

\subsubsection{Forms }
 By the {\it form} $F$ on $X$ we mean a $\Cl(X)$-homogenous function of gradation $[E]$ in $$H^0(X,\cO_X(E))x^E=H^0(B,\cO_B)_{[E]}=\cC(X)_{[E]}=(\cO_B)_{[E]}.$$ 
Each such $T_B$-semiinvariant form can be described as  $$F=x^D=x^\alpha= x^m \cdot x^E,$$ where $D\in \Div(X)$, $D=E+\divv(x^m)$, and $x^m\in H^0(X,\cO_X(E))$, for $E$ being a linear combination of $E_i$. 

With any form $F=fx^E\in H^0(X,\cO_X(E))x^E$ we can associate its divisor $\divv_X(F)=E+\divv(f)$, and its vanishing locus $V(F)=\supp(\divv(F))$. 
This extends to a homomorphism $$\Div(X)\to \Div(X),\quad D\to \divv(x^D),$$ which is identical on generators $E_i$ of $\Cl(X)\subset \Div(X)$, and thus on their  linear combinations. 
On the other hand, any class $[D]\in \Cl(X)$ can be written as the difference  $$[D]=[E']\setminus  [E''] $$ of  effective linear combinations  $E'$ and $E''$ of the generators $E_i$.
 
 Then $E'+\divv(x^{m})=D+E''$ is effective,  for a certain $m\in M$, and  we have the equality for the form $F:=x^{m} x^{E'}$:
 $$\divv(F)=\divv(x^{m}x^{E'})=\divv(x^{D})+\divv(E''),$$
 whence
 $$D+E''=E'+\divv(x^{m})=\divv(x^{D})+E''$$
 and thus  $$\divv_X(x^D)=D,$$ 
for any $D\in \Div(X)$.
Consequently $V(x^D)=\supp(D)$ for any form $x^D$, where $D\in \Div(X)_+$.

In particular, the vanishing locus $\divv(x_i)=\divv(x^{D_i})=D_i$ corresponds to the vertex $v_i\in \Ver(\Sigma)$. 
 
\subsubsection{Cox characteristic space}
The subscheme $B_+$ can be described using the $T_B$-semiinvariant forms on $X$ as in \cite{Cox2}.
By the construction, $B_+$ can be covered by the open affine subsets $B_{\delta}:=\pi^{-1}(X_{\delta})$, where $\delta\in \Sigma_{B_+}$. For each $\delta\in \Sigma_{B_+}$ consider the form $\check{x}_\delta:=\prod_{v_i\not \in \delta} x_i$ on $X$.
Its vanishing locus is equal to the complement $$X\setminus X_{\delta}=\bigcup_{v_i\not \in \delta} D_i.$$
So we can write $X_\delta=X\setminus V_X({\check{x}_\delta})$. Similarly $B_{\delta}=B\setminus V_B(\check{x}_\delta)=B_{\check{x}_\delta}$, where $\check{x}_\delta$ is considered as a function on $B$. 

Consequently $$B_+=B\setminus V(\cI_\irr),$$ where
$$\cI_{\irr}:=(\check{x}_\delta \mid \delta\in \Sigma)\subset \cO(B)=\cC(X)$$ is the {\it irrelevant ideal} (see \cite{Cox},\cite{Cox2}). 
 
 Moreover the morphism $B_{\delta}\to X_\delta$, can be described as $$B_{\delta}=B_{\check{x}_\delta}=\Spec(\kappa[x_1,\ldots,x_k]_{\check{x}_\delta})\to X_\delta=X_{\check{x}_\delta}=X\setminus V_X(\check{x}_\delta)$$
Note however then that the condition $f\in H^0(X_{\check{x}_\delta},\cO_X(E))$ is equivalent to
 $$\divv(f)+E+\divv(\check{x}_\delta^n)\geq 0$$ for  $n\gg 0$. The latter condition can be written as $$fx^E\cdot \check{x}_\delta^n\in H^0(X, \cO_X(E+n[\divv(\check{x}_\delta)])\cdot x^{E+n[\divv(\check{x}_\delta)]}$$ 
 
 Consequently
  $$\kappa[x_1,\ldots,x_k]_{\check{x}_\delta}=(\bigoplus_{E\in \Cl(X)} H^0(X,\cO_X(E))\cdot x^E)_{\check{x}_\delta}=\bigoplus_{E\in \Cl(X)} H^0(X_{\delta},\cO_X(E))\cdot x^E$$
 
 The latter leads to the formula for the Cox characteristic space to be
 $$B_+=\Cox(X)=\Spec_X (\bigoplus_{E\in \Cl(X)} \cO_X(E)\cdot x^E)$$
as in \cite{Cox2}.

\subsection{Cox relative spaces over affine toric schemes}
In this section, we shall study the general relative Cox construction developed in Chapter 1 in the context of birational toric morphisms. To a great extent, it is analogous to the original Cox construction for  toric varieties (as in \cite{Cox}) presented in the previous section. On the other hand, one can  link it to the original construction of Satriano, who developed a similar notion in the context of the toric Artin stacks in \cite{Satriano}.

The following result shows the relation between the toric Cox construction for toric varieties and the  general Cox construction for proper morphisms. 

\begin{lemma}
Let $\sigma$ be a regular cone, and $\Delta$ be its subdivision. Let $\pi: X_\Delta\to X_\sigma$ be the induced proper birational morphism.  Then the toric Cox coordinate space $\Cox(X_\Delta)$ and the toric Cox characteristic space $\Cox(X_\Delta)_+$ for toric variety $X_\Delta$ coincide with the relative Cox coordinate space $B=\Cox(X_\Delta/X_\sigma)$ and relative Cox characteristic space $B_+=\Cox(X_\Delta/X_\sigma)_+$ for the proper birational morphism $X_\Delta\to X_\sigma$. 

\end{lemma}

\begin{proof}
The construction of the spaces is formally identical. The reason is that the gradation in both cases is the group $\Cl(X_\Delta)=\Cl(X_\Delta/X_\sigma)$, which is freely generated by the exceptional toric divisors $E_i$ with no relations.
\end{proof}

\subsubsection{System of local parameters on  affine toric schemes} 

Let $P_\sigma=\sigma^\vee\cap M$ be the  monoid associated with the affine toric variety $X_\sigma=\Spec(\kappa[P_\sigma])$.  Denote by ${P}_\sigma^*\simeq \ZZ^r$  the subgroup of the invertible elements in ${P}_\sigma$, and let  $\overline{P}_\sigma:={P}_\sigma/{P}_\sigma^*$. The natural homomorphism of monoids $ {P}_\sigma\to \overline{P}_\sigma={P}_\sigma/{P}_\sigma^*$ splits, and  one can write noncanonically $$P_\sigma=\overline{P}_\sigma\times P_\sigma^*,$$

 Let $u_1=m_1,\ldots,u_s=m_s\in P_\sigma$ be the
minimal set of generators of the monoid $\overline{P}_\sigma$. This set is determined uniquely and consists of the elements $m\in \overline{P}_\sigma$, which cannot be written as $m=m'\cdot m''$ for the nontrivial elements $m',m''\in \overline{P}_\sigma$.  
\begin{definition}
The set of generators of $u_1,\ldots,u_s\in \overline{P}_\sigma\subset {P}_\sigma$ will be called a {\it system of local toric parameters} on $X_\sigma$.
\end{definition}


\subsubsection{Cox relative spaces over affine toric schemes}
\begin{lemma} \label{cover4}
Let $\sigma_0$ be any   cone in $\NN^\QQ$, and $\Delta$ be its subdivision. Consider the induced toric morphism $\pi: X_\Delta\to X_{\sigma_0}=\Spec(\kappa[P_{\sigma_0}])$.  Let $E_1,\ldots,E_k$ be the toric exceptional divisors  of $\pi$ corresponding to the vertices
$v_1,\ldots,v_k\in\Ver(\Delta)\setminus \Ver(\sigma_0)$, and the exceptional valuations $\nu_i=\val(v_i)$.

Let $B$ and $B_+$ denote the full cobordization and, respectively, the cobordization of the morphism $X_\Delta\to X_{\sigma_0}$

Then 
\begin{enumerate}
\item $B=\Spec\cO_{X_\sigma}[t_1^{-1},\ldots,t_k^{-1},u_1{\bf t}^{\alpha_1},\ldots,u_s{\bf t}^{\alpha_k}]$, where and  $u_1,\ldots,u_k\in P_\sigma$ is a system of local toric parameters
and $\alpha_i=(a_{i1},\ldots,a_{ik})$, with $a_{ij}:=\nu_j(u_i)$.
\item $B$ is a toric variety $B\simeq X_{\sigma_0}\times \AA^k$, and the corresponding cone is $$\sigma_{B}=\sigma_0\times \langle e_1,\ldots,e_k\rangle.$$
\item The natural morphism	$$B=\Spec\cO_{X_{\sigma_0}}[t_1^{-1},\ldots,t_k^{-1},u_1{\bf t}^{\alpha_1},\ldots,u_k{\bf t}^{\alpha_k}]\to X_{\sigma_0}$$ is given by the projection $\pi_\Sigma:\sigma_{B}\to \sigma_0$, mapping $e_i\mapsto v_i$.\item $B_+\subset X_{\sigma_{B}}$ can be described as the set $\Sigma_{B+}$ of the faces  $\sigma$ of $\sigma_{B}$ such that $\pi_\Sigma(\sigma)\subseteq \delta$, where $\delta \in \Delta$. In particular, $B_+\subset B$ is an open inclusion.

\end{enumerate}

\end{lemma}
\begin{proof} First we will show that $X_\Delta$ can be covered by the open affine subsets $(X_\Delta)_F$, where $F$ is a form on $X_\Delta$. The problem translates into a toric situation. For any cone $\delta\in \Delta$ let $\omega$ be a maximal common face of $\delta$ and $\sigma_0$. Consider a character $\chi_\delta\in \sigma_0^\vee$ which is zero on $\omega$ and strictly positive on $\sigma_0\setminus \omega$. The character $\chi_\delta$ defines a regular function on $X_{\sigma_0}$, for which $$n_i:=\chi_\delta(v_i)=\nu_i(\chi_\delta)>0,$$ for any vertex $$v_i\in \Ver(\delta)\setminus \Ver(\omega)=\Ver(\delta)\setminus \Ver(\sigma_0).$$ In particular $\divv(\chi_\delta)-E_\delta\geq 0$, where $E_\delta:=\sum_{E_i\cap X_\delta\neq \emptyset} n_iE_i$. Then, for the form  $F_\delta:=\chi_\delta x^{-E_\delta}$, its support $$\supp(\divv(F_\delta))=\supp(\divv(\chi_\delta))-E_\delta)$$ on $X_\Delta$ is the union of all the toric divisors  which are  in $X_\Delta\setminus X_\delta$ and  which
correspond to  the vertices in $\Ver(\Delta)\setminus \Ver(\delta)$.
Consequently $\supp(\divv(F_\delta))=X_\Delta\setminus X_\delta$, and $X_{F_\delta}=X_\delta$. 
This implies, by Proposition \ref{cover2},  that the natural morphism $$B_+=\bigcup_{\delta\in \Delta} B_\delta \hookrightarrow B$$ is an  open immersion, where $B_\delta:=B_{F_\delta}$ is open affine.  

(1) For any $i=1,\ldots,k$ let $t_i$ be the coordinate corresponding to $-E_i$. Set 

$$\check{\bf t}_i:=(t_1,\ldots,\check{t}_i,\ldots,t_k)\quad \check{\bf t}_i^{-1}:=(t^{-1}_1,\ldots,\check{t}^{-1}_i,\ldots,t^{-1}_k)$$  By Proposition \ref{divi} one can write:
$$\cA_{Y/X}=\bigoplus_{a_i\in \ZZ} \,\, \bigcap^k_{i=1}\, \cI_{\nu_i,a_i}\,\,\cdot t_1^{a_1}\cdot\ldots\cdot t_k^{a_k}=\bigcap^k_{i=1}\bigoplus_{a_i\in \ZZ} \cI_{\nu_i,a_i}\,\,\cdot t_i^{a_i}[\check{\bf t}_i, \check{\bf t}_i^{-1}],$$
Let $u_1,\ldots,u_k\in \overline{P}_\sigma=P_\sigma/P_\sigma^*\subset  P_\sigma$ be the generators of $\overline{P}_\sigma$, and let
 $\nu_i(u_j)=a_{ij}\in \ZZ_{\geq 0}$ for $i=1,\ldots,k$, and $j=1,\ldots,n$. Then
 $$\cI_{\nu_i,a}=(u_1^{b_1}\cdot\ldots \cdot u_n^{b_n}) \mid \sum_{j=1}^k {b_{j}}a_{ij}\geq a).$$ Comparing gradations we easily see that for each $i$, $$\bigoplus_{a_i\in \ZZ}  \cI_{\nu_i,a}t_i^{a_i}=\cO_X[t_i^{-1},u_jt_i^{a_{ij}}].$$
So $$\cA_{Y/X}=\bigcap^k_{i=1} \cO_X[t_i^{-1},u_jt_i^{a_{ij}}][\check{\bf t}_i, \check{\bf t}_i^{-1}]=\cO_X[t_1^{-1},\ldots,t_k^{-1},u_1{\bf t}^{\alpha_1},\ldots,u_k{\bf t}^{\alpha_k}],$$
where $\alpha_i=(a_{i1},\ldots,a_{ik})$.

(2) \begin{align*}B&=\Spec\cO_{X_\sigma}[t_1^{-1},\ldots,t_k^{-1},u_1{\bf t}^{\alpha_1},\ldots,u_k{\bf t}^{\alpha_k}]=\\
 	&=\Spec(\kappa[u_1,\ldots,u_k, v_1,\ldots,v_r][t_1^{-1},\ldots,t_k^{-1},u_1{\bf t}^{\alpha_1},\ldots,u_k{\bf t}^{\alpha_k}]=\\
 	&=\Spec(\kappa[t_1^{-1},\ldots,t_k^{-1},u_1{\bf t}^{\alpha_1},\ldots,u_k{\bf t}^{\alpha_k},v_1,\ldots,v_r]\simeq \\&\simeq\Spec(\kappa[t_1^{-1},\ldots,t_k^{-1},u_1,\ldots,u_k,v_1,\ldots,v_r])\simeq X_{\sigma_0}\times \AA^k.
 \end{align*}

(3) The toric map \begin{align*}B&=\Spec(\kappa[t_1^{-1},\ldots,t_k^{-1},u_1{\bf t}^{\alpha_1},\ldots,u_k{\bf t}^{\alpha_k},v_1,\ldots,v_r]\to\\& X_{\sigma_0}=\Spec(\kappa[u_1,\ldots,u_s,v_1,\ldots,v_r]),\end{align*}
corresponds to the map of cones $$\pi_B:\sigma_{B}\simeq\sigma_0\times \langle e_1,\ldots,e_k\rangle \to {\sigma_0}.$$ 
Under this correspondence $$\val(e_i)(t_j^{-1})=\delta_{ij},\quad \val(e_i)(u_j{\bf t}^{\alpha_j})=0.$$ On the other hand $\val(v)(t_i^{-1})=0$ for any 
integral vector $v\in {\sigma_0}$.  

By the above we can write $$B=\Spec(\kappa[t_1^{-1},\ldots,t_k^{-1}]\times \Spec(\kappa[u_1{\bf t}^{\alpha_1},\ldots,u_k{\bf t}^{\alpha_k},v_1,\ldots,v_r])$$

The toric valuation $\mu_i$ on $B$ associated to the divisor $D_i=V_B(t_i^{-1})$ satisfies $\mu_i(u_j{\bf t}^{\alpha_j})=0$, and $\mu_i({\bf t}^{-1}_{i'})=0$ for $i\neq i'$.
It corresponds to the vector $e_i$, as $\val(e_i)$ fulfills precisely the same relations.

The quotient morphism $\pi_B: B\to X_{\sigma_0}$ takes a toric valuation $\val(v)$ on $B$, for any integral $v\in\sigma_{B}$ to the restriction to $\cO(X_{\sigma_0})$ corresponding to $\val(\pi_\Sigma(v))$.
It maps the vertices of the face ${\sigma_0} \subset\sigma_{B}$ to the very same vertices of ${\sigma_0}$. The image of the vector $e_i$ is the vertex $\pi_\Sigma(e_i)=v_i\in \Ver(\Delta)\setminus \Ver({\sigma_0})$. This follows  from  Lemma \ref{Ex2} or can be seen by direct computation. By the previous considerations, $v_i$ corresponds to $\nu_i$ on $X_\Delta$, and $e_i$  to the valuation $\mu_i$ of $t_i^{-1}$ on $B_+$. The restriction of the toric valuation $\mu_i$ to $\kappa[u_1,\ldots,u_k]$, gives $$\mu_i(u_j)=\mu_i(u_j{\bf t}^{\alpha_j}\cdot {\bf t}^{-\alpha_j})=\mu_i({\bf t}^{\alpha_j})={a_{ij}}=\nu_i(u_j).$$

(4) By the considerations at the beginning of the proof, and  Lemma \ref{cover3}, we can write $B_+$ as the union of open affine subsets $B_\delta=B_{F_\delta}=\pi_{B_+}^{-1}(X_\delta)$:$$B_+=\bigcup_{\delta\in \Delta} B_\delta \subset B.$$ 
The induced map  of fans $\Sigma_{B_+}\to \Delta$ corresponds to the good quotient $B_+\to B_+\sslash T$, and  is defined by the linear map: $$\pi_\Sigma: \NN^{\QQ}_{B}=\NN^{\QQ}_{B_+}\to \NN^\QQ=\NN^\QQ_Y= \NN^\QQ_Y.$$
 
 Thus any cone $\delta\in \Delta$ can be written as the image $\delta=\pi_\Sigma(\delta')$, where $\delta'\in \Sigma_{B_+}$. 
  In particular any vertex $v_i\in \Ver(\Delta)\setminus \Ver(\sigma_0)$  is the image $v_i=\pi(e_i)$ of $e_i\in \Ver(\Sigma_{B})=\Ver(\Sigma_{B_+})$.
Consequently,
the  fan $\Sigma_{B_+}$ is determined by the faces $\tau$ of $\Sigma_{B}$ such that $\pi_\Sigma(\tau)\in\Delta$.
\end{proof}

\subsection{Cox relative spaces for  toric morphisms. General case}

\subsubsection{Coborization of proper toric morphisms}
Let $\Delta$ be a subdivision of a fan $\Sigma$.
We can further generalize the characterization of the cobordization of any proper birational toric morphism $\pi: X_\Delta\to X_\Sigma$.

\begin{proposition} \label{toricg} Let $\Delta$ be a fan subdivision of a  fan $\Sigma$. Let $\pi:Y=X_\Delta\to X=X_\Sigma$ be the associated proper toric morphism of toric varieties. Let $v_1,\ldots,v_k$ be the vertices of $\Ver(\Delta)\setminus \Ver(\Sigma)$ corresponding to the toric valuations $\nu_i=\val(v_i)$, associated with the exceptional divisors $E_1,\ldots,E_k$. Let ${\sigma_0}=\langle e_1,\ldots,e_k\rangle$ be the regular cone defined by the free basis  $e_1,\ldots,e_k$. 

Let $\pi_\Sigma: |\Sigma|\times{\sigma_0}\to |\Sigma|$ be the linear map of  the supports of fans  such that $\pi_\Sigma(e_i)=v_i$, and identical on $|\Sigma|$. Consider the subfan $\Sigma_{B}$ of $\Sigma\times{\sigma_0}$ consisting of the faces of $\Sigma\times{\sigma_0}$ mapping to faces of $\Sigma$, under the projection $\pi_\Sigma$.
Then the full cobordization  $B\to X$ of $\pi$ can be described as the toric morphism associated with the projection $\pi_{\Sigma|\Sigma_B}:\Sigma_{B}\to \Sigma$.

The  morphism $B_+\subset B$ is an open inclusion which corresponds to the subfan $\Sigma_{B_+}$ of $\Sigma_{B}$ of all the faces of $\Sigma\times {\sigma_0}$ mapping to the faces of $\Delta$.

\end{proposition}
\begin{proof} 
By  Lemma \ref{cover4}, and reducing to the affine case, we see that $B_+\subset B$ is an open immersion.


Let $T_0:=\Spec (\kappa[M])\subset X_\Sigma$ be the torus acting on $X_\Sigma$, and on $X_\Delta$. Let ${T_B}:=\Spec(\kappa[t_1,t_1^{-1},\ldots,t_k,t_k^{-1}])$, where the coordinates $t_1^{-1},\ldots,t_k^{-1}$,  correspond to $e_1,\ldots,e_k$ on $X_{{\sigma_0}}=\kappa[t_1^{-1},\ldots,t_k^{-1}])$
.
By Proposition \ref{valu},  we can write $B$ as $B=\Spec(\cA_{Y/X})$,
 where $$\cA_{Y/X}= \bigoplus_{a_i\in \ZZ} \,\, \bigcap^k_{i=1}\, \cI_{\nu_i,a_i}\,\,\cdot t_1^{a_1}\cdot\ldots\cdot t_k^{a_k}\subset \cO_X[t_1,t_1^{-1},\ldots,t_k,t_k^{-1}]$$
Consequently, $B$ contains  a toric variety 
 $$B=\Spec_X(\cA_{Y/X}) \supset B_-=\Spec_X(\cO_X[t_1,t_1^{-1},\ldots,t_k,t_k^{-1}])= X_\Sigma\times {T_B},$$
and hence  contains a torus $T_0\times T_B$. Moreover the torus
$T_0\times T_B$  acts on $B$. 

On the other hand, by Lemmas \ref{open}, and \ref{cover4} the scheme $B$ is the union of toric varieties $B_\sigma$ containing $T_0\times T_B$, associated with $\sigma\in \Sigma$, such that 
\begin{align*} & B_\sigma:=\pi_B^{-1}(X_\sigma)=B(X_{\Delta_{|\sigma}}/X_\sigma)\times T(X_\sigma)=\\ &\Spec_{X_\sigma}(\cO_{X_\sigma}[t_1^{-1},\ldots,t_k^{-1},u_1{\bf t}^{\alpha_1},\ldots,u_s{\bf t}^{\alpha_k}].\end{align*}
 
 Thus $B$ is a toric variety, let $\Sigma_B$ its corresponding fan.  The affine toric morphism $B\to X$ determines the homomorphism of tori $T_0\times T_B\to T_0$. It  corresponds to the map of fans $(\Sigma_{B}, N_0\times \NN_B)\to (\Sigma,\NN_0)$, defined by the natural projection $N_0\times \NN_B$
 
 Consider the toric variety $X_\Sigma\times \Spec(\kappa
 [t_1^{-1},\ldots,t_k^{-1}])$, associated with the fan $\Sigma\times {\sigma_0}$, with the lattice $N_0\times \NN_B$.
The linear map  $$\pi_\Sigma: (|\Sigma|\times {\sigma_0},N_0\times \NN_B)\to (|\Sigma|,\NN_0),$$ satisfies  $\pi_\Sigma(e_i)=v_i$, and  $\pi_{\Sigma_|\NN_0}=\id_{\NN_0}$.

 
The full cobordization morphism $B\to X_\Sigma$  takes the divisor $D_i=V_B(t_i^{-1})$ to $E_i$ by Lemma \ref{Ex2}. Thus it defines the same map on the lattices $\pi_\Sigma: \NN_B\to \NN$.
 Moreover, by Lemmas \ref{open}, and \ref{cover4}, each  toric variety $$B_\sigma=B(X_{\Delta_{|\sigma}}/X_\sigma)\times T_{B\setminus B_\sigma}\subset B$$ corresponds to the subfan determined by the cone $\sigma\times \tau(\sigma)$  of $\Sigma\times {\sigma_0}$, where $\tau(\sigma)\leq {\sigma_0}$ is generated by all $e_i$ with $v_i\in \sigma$. 
 
 Thus, by Lemma \ref{cover4}(4), $\Sigma_{B}$ consists exactly of the faces of $\Sigma\times\sigma_0$ mapping into faces of $\Sigma$, under the projection $\pi_\Sigma: e_i\mapsto v_i$.
Consequently, $B_+$ corresponds to the subfan $\Sigma_{B_+}$ of $\Sigma_{B}$ of all the faces mapping to the faces of $\Delta$.

  \end{proof}
\subsection{The dual complex of the exceptional divisor}
\subsubsection{The dual complex of toric morphisms} \label{dual} Let $\pi: Y=X_\Delta\to X_\Sigma$ be a proper toric morphism, where $\Delta$ is a subdivision of $\Sigma$. Assume that $X=X_\Sigma$ is smooth. Then the full cobordization  $B$ of $\pi$ is a smooth toric variety with the toric morphism $\pi_B: B\to X$. Consequently the exceptional divisors $D=V_B(t_1^{-1},\ldots,t_k^{-1})$ of $\pi_B$ and $D_+=D\cap B_+$  of $\pi_{B+}$ are SNC. On the other hand the components $D_{i+}$ map to the components $E_i$ of 
the exceptional toric divisor $E$ of $\pi: Y\to X$.   

 One can define the divisorial stratifications
$S_D$, and $S_{D_+}$ on $B$, and $B_+$ with the strata determined by the nonempty sets $$s_I:=\bigcap_{i\in I} D_i\setminus \bigcup_{j\in J} D_j,$$ where $I\cup J=\{1,\ldots,k\}$, $I\cap J=\emptyset$.
Note that the closure $\overline{s_I}$ can be written in the form
$$\overline{s_I}=\bigcap_{i\in I} D_i.$$

Likewise the  stratification $S_E$ od $E$ on $Y$ is determined by the nonempty closed sets
$\overline{s_I}^E:=\bigcap_{i\in I} E_i$, which determine the strata $s_I^E$ obtained by removing from $\overline{s_I}^E$ all the proper subsets $\overline{s_J}^E\subset \overline{s_I}^E$, with $\cJ\supset \cI$.

These three stratifications are coarser than the orbit stratifications; thus, each stratum is the union of orbits.
The divisorial stratifications $S_D$ and $S_{D_+} $ define the dual complexes  $\Delta_D$ and $\Delta_{D_+}\subset \Delta_D$. The vertices $e_i$ of $\Delta_D$ and $\Delta_{D_+}$ correspond to the divisors $D_i$ or, respectively $D_{i+}$. We associate with a stratum $s=\bigcap_{i\in I} D_i\setminus \bigcup_{j\in J} D_j,$ the simplex $\sigma_s:=\Delta(e_i\mid i\in I)$. 

 Similarly, we can  define the dual complex $\Delta_E$ associated with the toric exceptional divisor $E$ on $Y$ (which is usually not SNC). Again the vertices $e_i$ of $\Delta_E$  correspond to the divisors $E_i$. We associate with any set of divisors $\{E_i\mid i\in I\}$ such that  
  $\bigcap_{i\in I} E_i\neq \emptyset$ the simplex $\sigma_I:=\Delta(e_i\mid i\in I)$. 
  Summarizing we obtain the following characterization of the complexes: 
  \begin{lemma} A simplex $\sigma$ in $\Delta_D$( respectively in $\Delta_{D_+}$ or $\Delta_E$) corresponds bijectively to a set of divisors $D_i$ ( respectively $D_{i+}$ or $E_i$) having a nonempty intersection. \qed

  \end{lemma}

\begin{lemma}\label{closed} Let $B\to X$ be the full cobordization of $\pi:Y=X_\Delta\to X=X_\Sigma$. Let $D$ be the exceptional divisor on $B$, and $S_D$ be the induced stratification. Then for any stratum $s\in S_D$, the image $\pi_B(s)$ is closed in $X$. 
 \end{lemma}
\begin{proof}   The problem is local on $X$ so we can assume that $X=X_\sigma$.
 Let $B=X_\sigma\times X_\delta$, where $\delta=\langle e_1,\ldots,e_k\rangle$ the regular cone generated by a free basis $\{e_1,\ldots,e_k\}$.
 
 The morphism $B\to X$ corresponds to the projection $\pi_\Sigma: \sigma_B=\sigma\times \delta\to \sigma$, mapping $e_i$ to $v_i\in \Ver(\Delta)\setminus \Ver(\sigma)$.

Any stratum  $s=\bigcap_{i\in I} D_i\setminus \bigcup_{j\in J} D_j$ in $S_D$
 is closed on the open affine subset $B'=B\setminus \bigcup_{j\in J} D_j$ of $B$. By replacing $B$ with the open affine subset $B'=B_{t_J^{-1}}$, where $t_J=\prod_{j\in J} t_j$,  we assume  that $s=\bigcap_{i\in I} D_i$ and all the exceptional vertices $e_i$ of $\sigma\times \delta$ , where $i \in I$,  span the cone $\delta$. Then $s$ is the union of orbits corresponding to the cones  in $ \Star(\delta,\sigma_B)$.  

Let   $\delta_0:=\pi_\Sigma(\delta)$ be the image of $\delta$, which is a subcone in $\sigma$ generated  by $\pi_\Sigma(e_i)$. Denote  by $ \sigma_0\leq \sigma$  the unique face such that $\inte(\delta_0)\subset \sigma_0\leq \sigma$. 

Thus, since $\sigma$ is regular, and the map $\pi_\Sigma$ is surjective, the image $\pi_B(s)$ of $s$ is defined by the orbits corresponding to the cones in $\Star(\sigma_0,\sigma)$, and thus it is closed.





\end{proof}

\subsubsection{The center of valuation}\label{zval}
Recall that for any valuation $\nu$ of the quotient field $\kappa(X)$, we denote its center on $X$ by $Z_X(\nu)$.
\begin{corollary}\label{special15} Consider any stratum  $s\in S_D$ such that $s=\bigcap_{i\in I} D_i\setminus \bigcup_{j\in J} D_j.$
Then $$\pi_B(s)=\bigcap_{i\in I} Z_X(\nu_{i}),$$
 \end{corollary}
\begin{proof} Note that, by Lemma \ref{Ex2}, the image $\pi_{B}(D_{i})=Z_X(\nu_{i})$. Then $\pi_B(s)\subseteq \bigcap_{i\in I} Z_X(\nu_{i}).$
The problem is local on $X$, and we use the notation and the description from the  proof of the previous Lemma.
The stratum $s$ contains a generic toric orbit corresponding to the cone $\delta=\langle e_i \mid i\in I \rangle$.   Its image $\pi_B(s)$ is closed and corresponds to the $\Star(\sigma_0,\sigma)$, where $\sigma_0$ is the smallest face containing the images $\{v_i \mid i\in I\}$, where $\val(v_i)=\nu_i$. 

On the other hand, $\bigcap_{i\in I} Z_X(\nu_{i})$	corresponds to  the faces of $\sigma$  containing all $v_i$. Both sets are identical and $\pi_B(s)=\bigcap_{i\in I} Z_X(\nu_{i}).$
\end{proof}
\begin{lemma} \label{special5} The morphism $\pi_{B_+,Y}: B_+\to Y$ determines a bijective correspondence between the sets of divisors $\{D_{i+} \mid i\in I\}$ such that $\bigcap D_{i+}\neq 0$, and the sets $\{E_{i} \mid i\in I\}$ for which $\bigcap E_{i}\neq 0$.




\end{lemma}
\begin{proof} We need  to show that $\bigcap_{i\in I} D_{i+}\neq \emptyset$ iff $\bigcap_{i\in I} E_i\neq \emptyset$. By Lemma \ref{special5}, $\pi_{B_+,Y}(D_{i+})=E_i$.
 Thus if $\bigcap_{i\in I} D_{i+}$  is nonempty then 
 $$\bigcap_{i\in I} \pi_{B_+,Y}(D_{i+})=\bigcap_{i\in I} E_i\supseteq \pi_{B_+,Y}(\bigcap_{i\in I} D_{i+})\neq \emptyset.$$ 
 Conversely, if $\bigcap_{i\in I} E_i$ is nonempty then the vertices $v_i$ corresponding to $E_i$ form   the subcone $\tau=\langle v_i \mid i\in I\rangle$ of a face $\delta\in \Delta$, with $\Ver(\tau)\subseteq \Ver(\delta)\setminus \Ver(\sigma)$. 
  Consequently, by Lemma \ref{cover4}(4), $\tau$  is the image of  the face $\delta'=\langle e_i \mid i\in I\rangle\in \Delta_{B_+}$, whence $\bigcap_{i\in I} D_{i+}$  is nonempty.
\end{proof}



\begin{corollary} \label{special6} The  natural surjective map $S_{D+}\to S_E$ determines an isomorphism of the dual complexes $\Delta_{D_+}\simeq\Delta_E$.\qed
	\end{corollary}

Also, we have
\begin{corollary} \label{special7}  The  natural morphism $S_{D+}\to S_D$ determines the inclusion of the dual complexes $\Delta_{D_+}\simeq\Delta_{D}$, so that $\Delta_{D_+}$ is a subcomplex of $\Delta_{D}$.
	\end{corollary}
\begin{proof} By the construction, the faces of $\Delta_{D}$ (and $\Delta_{D_+}$) correspond to the sets of divisors $\{D_i\mid i\in I\}$ such that $\bigcap_{i\in I} D_{i}\neq \emptyset$.
Now, if $\bigcap_{i\in I} D_{i+}\neq \emptyset$  then obviously $\bigcap_{i\in I} D_{i}\neq \emptyset$.

	
\end{proof}

\subsection{Newton polytopes of monomial ideals}

\subsubsection{Newton polytopes} 
\label{Newt}
\begin{definition} \label{Newt2}

Consider the lattice of monomials $$\MM=\{{\bf x}^\alpha\mid \alpha \in  \ZZ^k\}\simeq \ZZ^k, $$ and let $N=\Hom(\MM,\ZZ)$ be its dual.
Let $I=({\bf x}^{\alpha_1},\ldots, {\bf x}^{\alpha_k})\subset \kappa[x_1,\ldots,x_n]$  be  a toric ideal generated by the monomials corresponding to the elements of  $$\alpha_i\in \ZZ_{\geq 0}^n\subset \sigma^\vee_0=\langle e_1^*,\ldots,e_n^*\rangle=\QQ_{\geq 0}^n.$$ By the associated {\it  Newton polytope} of $\cI$ we mean
 $${\rm P}={\rm P}_I:=\con(\alpha_1+\QQ_{\geq 0}^n,\ldots,\alpha_k+\QQ_{\geq 0}^n)\subseteq \QQ_{\geq 0}^n\subseteq M_{\QQ}=M\otimes \QQ=\QQ^n$$ 
 \end{definition}

Conversely any polytope (or polyhedron) $P=P+\QQ_{\geq 0}^n$ determines the ideal $$\cI=\cI_P:=({\bf x}^\alpha\mid \alpha \in P).$$

\begin{lemma}
 There is   a bijective correspondence $I\mapsto P_I$, $P\mapsto I_P$, between integrally closed toric ideals
$I\subset \kappa[x_1,\ldots,x_n]$, and polytopes $P=P+\QQ_{\geq 0}^n$ with integral vertices. \qed
\end{lemma}




\subsubsection{The orbit stratification}
One can identify $e_i^*$ with $x_i$, so we can write $\sigma_0^\vee=\langle  x_1,\ldots,x_n\rangle$. Denote by $N_{\QQ}$ the dual space of $M_{\QQ}$, and $\sigma_0\subset N_{\QQ} $ the dual of $\sigma_0^\vee$ as in Sections \ref{fans}, \ref{fans2}.
For any $\tau\subset N_{\QQ}$ set $$\tau^{\perp}:=\{y\in M_{\QQ}\mid (x,y)=0 \quad \mbox{for\,\, all}\,\, x\in \tau \},$$ 
\begin{lemma}\label{strata} There is a natural bijective correspondence between 

\begin{itemize}
\item the faces $\tau$ of $\sigma_0$
\item the faces $\tau^*:=\tau^\perp\cap \sigma_0^\vee$ of $\sigma_0^\vee$.
\item the open affine subsets $X_\tau\subset X_{\sigma_0}$	
 \item the minimal closed orbits $O_\tau \subset X_\tau$  which are in $X_{\sigma_0}$.
\end{itemize}

Moreover under the above  identification the closure $\overline{O_\tau}$ of the orbit $O_\tau$ is defined by the ideal $(x_i\mid x_i \not \in \tau^*)$.

\end{lemma}

\begin{proof} The face $\tau$ of $\sigma_0$ determines 
the open subset $$X_\tau=\Spec(\kappa[\tau^\vee\cap M])=\Spec(\kappa[P_\tau])=\Spec(\kappa[P^*_\tau])\times \Spec(\kappa[\overline{P}_\tau])$$ of $X_{\sigma_0},$ where 
$$P_\tau=\tau^\vee\cap M=P_{\sigma_0}+P^*_\tau= P_{\sigma_0} -( \tau^*\cap M).$$ 
Thus $ \tau^*\cap M=P_\tau^*\cap P_{\sigma_0}$ consists of the elements of $P_{\sigma_0}=\sigma_0^\vee \cap M$, which are invertible in $P_\tau$.  
The closed orbit $O_\tau\subset X_\tau$ is described by the ideal generated by the set of noninvertible elements 
$$P_\tau\setminus P_\tau^*:=(\tau^\vee\setminus \tau^*\setminus (-\tau^*))\cap M\subset P_\tau.$$ Thus its closure $\overline{O_\tau}$ in $X_{\sigma_0}$ is defined by the ideal $(x_i\mid x_i \not \in \tau^*)$ corresponding to the monoid ideal 	$P_{\sigma_0}\setminus P^*_\tau=(\sigma_0^\vee\setminus \tau^*)\cap M$. Conversely,
any face $\tau^*$ of $\sigma_0^\vee$ determines the closure $\overline{O_\tau}\subset X_{\sigma_0}$ of the orbit ${O_\tau}$ with the monoid ideal $(\sigma_0^\vee\setminus \tau^*)\cap M$, and the face $\tau=(\tau^*)^\perp \cap \sigma_0$ of $\sigma_0\subset N_{\QQ}$. \end{proof}

By the construction, $\cO_\tau$  is the smallest $T$-stable closed subset of $X_\tau$. If $\tau\subset \tau'$ is the inclusion of the faces then $X_{\tau}\subset X_{\tau'}$ is an open immersion, and $\overline{\cO_{\tau}} $ contains $\cO_{\tau'}$. Consequently the orbits $\cO_\tau$ form the stratification of $X_{\sigma_0}=\Spec(\kappa[x_1,\ldots,x_n])$.

\begin{corollary} \label{orbit} Let $\cI\subset \kappa[x_1,\ldots,x_k]$ be a monomial ideal and $P_\cI\subset \sigma_0^\vee$ be its Newton polytope. Then the toric subset $V(\cI)$ is exactly the union of the orbits $O_\tau$ such that $\tau^*$ is disjoint from $P_\cI$. 
\end{corollary}

\begin{proof} The orbit  $O_\tau$ is contained  in $V(\cI)$ if and only if the ideal of $\overline{O_\tau}$ contains $\cI$. Thus the corresponding  monoid ideal $(\sigma_0^\vee\setminus \tau_*) \cap M$ contains $P\cap M$. The latter is equivalent to the condition $\tau_*\cap P=\emptyset$.

\end{proof}

\subsubsection{Supporting faces}

The monomial ideal $I=({\bf x}^{\alpha_1},\ldots,{\bf x}^{\alpha_k})$ 
  defines a piecewise linear convex function $F_I:=\min (\alpha_i,v)$ on $\sigma_0:=\langle e_1,\ldots,e_n\rangle$ which is dual to $\sigma_0^\vee$.

Likewise any polytope $P\subset \sigma_0^\vee$ determines a piecewise linear convex function $$F_P:=\min((w,v)\mid  w\in P)$$ on $\sigma_0$.

If $P=P(I)$ then both functions coincide:
$$F_P=F_I=\min (v,\alpha_i)=\min((v,w)\mid  w\in P))$$

By the {\it dual fan} or {\it normal fan} of $P$, we mean the fan $\Delta_P=\Delta_\cI$  is determined by the maximal cones $\tau\subset \sigma_0$, where $F_P$ is linear.  By definition, $\Delta_P$  is a decomposition of $\sigma_0$.

Conversely, the function $F_P$ on $\sigma_0$, determines the
polytope $$P=\{w\in \sigma^\vee_0\mid (\cdot,w)_{|\sigma_0}\geq F_{P|\sigma_0}\}.$$


Recall the standard fact from the convex geometry:
\begin{lemma} \label{comb} There is a bijective correspondence between the faces $P$ of the polytope $P_0$, and the faces $\tau_{P}$ of the fan $\Delta_{P_0}$.
 $$P\mapsto \tau_{P}:=(P_0-P)^\vee=\{v\in \sigma_0 \mid (v,w)\geq 0,\,\, w\in P_0-P\}\in \Delta_{P_0}$$
	$$ \tau\mapsto P_\tau=\{w\in P \mid F_{P|\sigma}=(\cdot,w)_{|\sigma}\}.$$
Moreover $\dim(P)=n-\dim(\sigma_P)$.  \qed
	
\end{lemma}

	

\begin{remark} For any $i=1,\ldots,n$,  let $$a_i:=\min \{x_i(p)\mid p\in P\}.$$ Then $$P_i:=\{p\in P\mid (x_i-a_i)(p)=0\}$$ is the face of $P$ corresponding to the one-dimensional face $\langle e_i\rangle$  determine by the vertex $e_i$ of $\sigma$.

	
\end{remark}

\begin{definition}\label{support} By the {\it supporting facets} of $P_0$ we mean the faces corresponding to the vertices of $\Ver(\Delta_{P_0})\setminus \Ver(\sigma)$. The affine hull of a supporting face will be called a {\it supporting hyperplane}.  By the {\it supporting faces}, we mean the faces, which are the intersections of some supporting facets.
	
\end{definition}


As a corollary from Lemma \ref{comb}, we obtain 
\begin{lemma} \label{delta} Let $\Delta
$  be the subdivision of regular cone $\sigma_0$ associated with the normalized blow-up $\pi:Y=X_\Delta\to X=X_{\sigma_0}$ of the monomial ideal $I\subset \kappa[P_\sigma]=\kappa[x_1,\ldots,x_n]$. 
 Let $ B\to X=X_{\sigma_0}=\AA^n$ be the  full cobordant blow-up of $I$. 
Then the following sets are in the bijective correspondence
\begin{enumerate}
\item The supporting hyperplanes $H_i$ of $P(I)$.
\item  The vertices $v_i$ of $\Ver(\Delta)\setminus \Ver(\sigma_0)$.
\item The exceptional divisors $D_i$ of $B\to X$.
\item The exceptional divisors $E_i$ of $Y\to X$.
\item The toric exceptional valuations $\nu_i=\val(v_i)$ on $X$ associated with $E_i$ on $Y$.
\item	The vertices  of the dual complexes $\Delta_E\simeq \Delta_{D_+}$ and $\Delta_D$.

\qed
\end{enumerate} 

\begin{remark} The supporting faces exist if $\codim(V(\cI))\geq 2$. On the other hand, if $\cI$ is principal, then  $\Delta=\sigma_0$, and thus $P$ admits no supporting faces.

\end{remark}

 \end{lemma}
 \begin{corollary} \label{use} With the above notation and assumptions:
 \begin{enumerate}
 \item Any exceptional valuation $\nu$ determines the supporting hyperplane $H_\nu$.

 \item With a face $\sigma$ of $\Delta_E$ one can associate
 the set $\omega_\sigma$ of the exceptional valuations corresponding to the vertices $\Ver(\sigma)$.
\item Any  face $\sigma$ of the dual complex $\Delta_E$ determines
 the supporting  face $P_\sigma$ of $P$, where $$P_\sigma= \bigcap_{\nu\in \omega_\sigma} H_\nu \cap P.$$	
 \item  \begin{align*} & \inv^\circ_{\omega_\sigma}(\cI):= (u^\alpha \in I\mid \nu(\cI)=\nu({\bf x}^\alpha), \nu \in \omega_\sigma)=\\& =\inv_{P_\sigma}(\cI):=({\bf x}^\alpha \in I\mid \alpha \in P_{\sigma}).
  \end{align*}
\end{enumerate} \qed

 \end{corollary}


 

Here $$\nu(\cI):=\min\{\nu(f)\mid f\in \cI\}.$$
One can see the above relations in the following example:
\begin{example} \label{111} Let $I=(x^k,xy,y^{l})\subset \kappa[x,y]$. 

The Newton polytope $P$  of $I$ is generated by the vertices $P_1=(k,0), P_2=(1,1), P_3=(0,l)$ of $P$.
The supporting planes $H_1$, and $H_2$  are determined, respectively, by the supporting facets $P_{12}=\con(\{(k,0), (1,1)\})$, and $P_{23}=\con(\{(1,1), (0,l)\})$.  They correspond to the vectors $v_1=(v_{11},v_{12})$, $v_2=(v_{21},v_{22})$  such that $$kv_{11}=v_{11}+v_{12},\quad v_{21}+v_{22}=lv_{21}$$
Thus $v_1=(1,k-1)$, $v_2=(l-1,1)$. The decomposition $\Delta_P$ consists of three $2$-dimesional cones $\sigma_1=\langle e_1,v_1 \rangle$, $\sigma_2=\langle v_1,v_2\rangle$, $\sigma_3=\langle v_2,e_2\rangle$, and their $1$-dimesional faces. These $2$-dimesional faces in $\Delta_P$  correspond to the vertices $(k,0), (1,1), (0,l)$ of $P$, and the associated monomials $x^k,xy,y^{l}$. 
The vectors $v_1,v_2\in \Ver(\Delta_p)$ correspond to the exceptional valuations $\nu_1=\val(v_1), \nu_2=\val(v_2)$.
In particular \begin{align*} & \inv^\circ_{\nu_1}(\cI)=\inv_{P_{12}}(\cI)=(x^k,xy),\quad\inv^\circ_{\nu_2}(\cI)=\inv_{P_{23}}(\cI)=(xy,y^{l}),\\&
\inv^\circ_{\nu_1,\nu_2}(\cI)=\inv_{P_{2}}(\cI)=(xy).\end{align*}

Note that the vertices in  $\Ver(P)$, which are, in our case, defined by $x^k,xy,y^l$, label the maximal faces in $\Delta_P$.

 We see that  the monomials in $\inv_\omega^\circ(\cI)=\inv_P(\cI)$ correspond to the maximal cones in the star of the relevant face  in $\Delta_P$. This face is described as the dual to
$P$. Equivalently, it is defined as the smallest face containing the set of the vertices of $\Delta_P$ determined by $\omega$.  In particular, the generators $x^k,xy$ occurring in $\inv^\circ_{\nu_1}(\cI)=\inv^\circ_{P_{12}}(\cI)=(x^k,xy)$ correspond to the maximal cones in the star of the face $\langle v_1\rangle\in \Delta_P$.

 The full cobordant blow-up of $I=(x^k,xy,y^l)$ is given by 
$$B=\Spec(\kappa[t_1^{-1},t_2^{-1}, xt_1t_2^{l-1}, yt_1^{k-1}t_2]),$$
and using Lemma \ref{blow},
$$B_+=\Spec(\kappa[t_1^{-1},t_2^{-1}, xt_1t_2^{l-1}, yt_1^{k-1}t_2])\setminus V(t_1^kt_2^l(x^k,xy,y^{l})).$$	
\end{example}





\subsection{Geometric quotients for  toric morphisms}
\begin{lemma} Let $\pi: Y=X_\Delta\to X_\Sigma$ be a  toric morphism, associated with the decomposition $\Delta$ of $\Sigma$. Assume that $\Sigma$ is simplicial.

Then its cobordization  $B_+ \to  Y=B_+/T_B$ is a geometric quotient iff $ \Delta$ is simplicial.

\end{lemma}
\begin{proof} The problem is local on $X$, and can be reduced to the affine toric morphism $X_\Delta\to X_\sigma$ corresponding to the subdivision $\Delta$ of a simplicial cone $\sigma$. Then, by Proposition \ref{toricg}, $\Sigma_B$ is simplicial, and so is $\Sigma_{B_+}$.

 The natural projection $\Sigma_{B_+} \to \Delta$ is defined bijectively on the vertices. Moreover, the faces of $\Delta$ are the images of cones in $\Sigma_{B_+}$.
  Thus $\Sigma_{B_+} \to \Delta$ is bijective on faces if and only if $\Delta$ is simplicial.   On the other hand, the condition that $\Sigma_{B_+} \to \Delta$ is bijective on faces is equivalent to $B_+\to Y$ being a geometric quotient.

\end{proof}

\begin{lemma} \label{toric2} Let $\pi: Y=X_\Delta\to X=X_\Sigma$ be a proper birational toric morphism of  toric varieties, with $X$ regular. Then $B_+\subset B$
contains open maximal subsets $B^s\subset B_+$ admitting geometric quotient $B^s\slash T_B$ which is projective birational over $Y$.
\end{lemma}
\begin{proof} The morphism $Y\to X$ corresponds to the subdivision $\Delta$ of $\Sigma$. Consider the sequence of the star subdivisions centered at  $\Ver(\Delta)$ of $\Delta$. By the definition of the star subdivision, the process transforms $\Delta$ into a simplicial fan $\Delta'$ with $\Ver(\Delta')=\Ver(\Delta)$, as all the vertices in the faces form linearly independent sets being the centers of the star subdivisions. So the valuations of the exceptional divisors corresponding to $\Ver(\Delta)=\Ver(\Delta')$ remain unchanged. Then, by Proposition \ref{toricg} we obtain that $B(Y/X)=B(Y'/X)$. On the other hand, by the second part of Proposition \ref{toricg},  we have the open inclusions of  toric subsets: $$B^s:=B(Y'/X)_+\subset B(Y/X)_+\subset B(Y/X).$$

 By the previous Lemma, $B^s \to B^s/T_B=X_{\Delta'}$ is a geometric quotient.

\end{proof}

\section{Cobordization of locally toric morphisms}
\subsection{Locally toric morphisms of locally toric schemes} 
\subsubsection{Locally toric schemes} 
\begin{definition} \label{ltoric} A normal scheme $X$ over a field $\kappa$ is {\it locally toric} if  any point $p\in X$ admits an  open neighborhood $U$,  and a regular morphism $\phi: U\to X_\sigma=\Spec(\kappa[P_\sigma])$, called {\it a toric chart}. 

An ideal $\cI$ on  a locally toric $X$ is called  {\it locally monomial} if for any point $p\in X$, there exists a toric chart $U\to X_\sigma=\Spec(\kappa[P_\sigma])$, and a monomial ideal $\cI_\sigma\subseteq \kappa[P_\sigma]$, defined by a subset of $P_\sigma$, such that $\cI_{|U}=\cI_\sigma\cdot \cO_{X|U}$.

\end{definition}

\begin{remark} The primary reason we consider locally toric schemes over a field $\kappa$, and not just over $\ZZ$, is that the morphisms to $\Spec(\ZZ)$ are, in general, not flat. Thus the toric charts over $\ZZ$ into $\Spec(\ZZ[P_\sigma])$ which are defined by the monomials in $P_\sigma=\sigma^\vee\cap N$ are not regular (not flat), and some proofs would require a different formalism. 
	
\end{remark}

\subsubsection{Locally monomial valuations}

\begin{definition} Let $X$ be a locally toric scheme. A valuation of $\kappa(X)$ with values in $\ZZ$ will be called {\it locally monomial} if for any point $p$ in the center $Z(\nu)\subset X$, there exists a toric chart $U\to X_\sigma$, and a vector $v\in \sigma\cap \NN$,
such that $\cI_{\nu,a}=\cO_X\cdot \cI_{\val(v),a}$, for any $a\in \ZZ_{\geq 0}$.
	
\end{definition}

\subsubsection{Locally toric morphisms}
\begin{definition}
A proper birational morphism $\pi : Y \to X$ of normal schemes over a field $\kappa$ is called {\it locally toric} if for any point $p\in X$ there is  an open neighborhood $U$, a toric chart $\phi: U\to X_\sigma$, and the  fiber square:
$$
\begin{array}{cccc} \pi^{-1}(U) & \stackrel{\psi}\to & X_\Delta  \\
  \pi_U \dar &   & \pi_A \dar \\ U &\stackrel{\psi}{\to} & X _\sigma  
   \end{array}.
$$
	
	\end{definition}
	
where $\pi_U:=\pi_{|\pi^{-1}(U)} : \pi^{-1}(U)\to U$  is the restriction of $\pi$.

\begin{proposition} Let $\cJ$ be a locally monomial ideal 
on a locally monomial scheme $X$. The normalized blow-up of $\cJ$ is a locally toric morphism. \qed
\end{proposition}

\subsection{Functoriality of cobordization of locally toric morphisms}
\subsubsection{Local toric presentation of cobordization of locally toric morphisms}
\begin{lemma} \label{local} Let $\pi: Y\to X$ be a locally toric proper birational  morphism. Then for any point $p\in X$ there exists an open neighborhood $U$ of $p\in X$, a toric chart $\phi_U: U\to X_\sigma$  and a fiber square $$
\begin{array}{cccc} Y_U:=\pi^{-1}U & \stackrel{\phi}\to & X_\Delta  \\
  \pi_U \dar &   & \pi_A \dar \\ U &\stackrel{\phi_U}{\to} & X_\sigma,  
   \end{array}.
$$ 

such that \begin{enumerate} 
\item There is a bijective correspondence between the irreducible exceptional divisors of $\pi_U$ and $\pi_A$. That is, any irreducible exceptional divisor of $\pi_U$ is the inverse image of an irreducible exceptional divisor of $\pi_A$.
\item There is a bijective correspondence between
the strata of the divisorial stratifications  
	of the exceptional divisors $E_U$ of $Y_U\to U$  and $E^\Delta$ of $X_\Delta\to X_\sigma$, which defines the isomorphism  $\Cl(Y_U/U)\to \Cl(X_\Delta/X_\sigma)$. Moreover, any stratum of the stratification $S_E$ is the inverse image of a stratum in $S_{E^\Delta}$.
	\item  For any $E'_U=\sum n_i(E_{U})_i$ and the corresponding $(E^\Delta)'=\sum n_iE_i^\Delta$ we have
	$\cO_{Y_U}((E^\Delta)')=\cO_Y\cdot (\cO_{X_\Delta}((E^\Delta)')$.
\item $B(\pi_U)=B(\pi_A)\times_{X_\sigma}  U\quad B_+(\pi_U)=B_+(\pi_A)_+\times_{X_\sigma} U.$

\item Any irreducible exceptional Weil divisor $E_i$ of $\pi$ defines a  locally monomial valuation $\nu_i$ with respect to any given toric chart $U\to X_\sigma$ associated with the morphism $\pi$.

	\end{enumerate}
\end{lemma}

\begin{proof} 
 (1) Since $\phi_U$ is regular, the inverse images $\phi_U^{-1}(s_\tau)$ of the toric strata $s_\tau$, where $\tau\leq \sigma$, define a stratification on $U$. Moreover, the induced morphisms on the strata  $\phi_U^{-1}(s_\tau)\to s_\tau$  are regular.
 
 We can assume that the given point $p\in X$ maps to a  point $q\in X_\sigma$, which is in the orbit $\cO_\sigma\subset X_\sigma$. For any $\tau\leq \sigma$, the closure $s_\tau:=\overline{O}_\tau$ of the toric orbit $\cO_\tau$ on $X_\sigma$ is normal. 
Moreover, since $\phi_U$ is regular, the inverse image $\phi_U^{-1}(s_\tau)$ is normal, and thus, it is the disjoint union of the irreducible components of the codimension equal to the codimension of $s_\tau$.

Consequently,  by shrinking $U$ around $p$, if necessary, we can assume that the inverse image of the closures of the toric strata (i.e., the orbits) on $X_\sigma$  are irreducible subsets of $U$.

The inverse image of $E^{\Delta}$ is the union of the normal divisorial components. Their images under $\pi_U$ are of the codimension $\geq 2$. So they are the exceptional divisors of $\pi_U$. Moreover all the irreducible exceptional divisors of $\pi_U$ are contained in $\phi^{-1}(E^{\Delta})$.

The image $\pi_A(E_i^{\Delta})$ contains the orbit $\cO_\sigma$ with $\phi_U^{-1}(\cO_\sigma)\neq \emptyset$. Then, by the assumption, $\phi^{-1}(E_i^{\Delta})\neq \emptyset$. 

We need to show that  each $\phi^{-1}(E_i^{\Delta})$  is an irreducible divisor. 
The image of the exceptional divisor $E^\Delta_i$ under $\pi_A$ defines the closure of the toric orbit$$\pi_A(E^\Delta_i)=\overline{s_{\tau}}=\overline{O_{\tau}}$$ on $X_\sigma$, for some face  $\tau\leq \sigma$. Denote by 
			$q$ the generic point of $E_i^{\Delta}$, and by
			$q^0$ the generic point of $s_{\tau}$.

 For any divisorial component, $E_{ij}$ in $\phi^{-1}(E_i^{\Delta})$, let $p_j$ be its generic point.
 By the assumption $\pi(p_j)\in U$ determines a unique point $p^0$ which is the generic point  of the stratum $s$ on $U$ so that $$\overline{s}=\overline{\phi_U^{-1}}(s_\tau)=\overline{p^0}.$$


 By definition,   the generic point $q$ of the toric divisor $E^\Delta_i$ on $X_\Delta$ is in the fiber 
$F_{q^0}=\pi_A^{-1}(q^0)$.
 Thus the generic points $p_j$ of the components $E_{ij}$ of $\phi^{-1}(E^\Delta_i)$  are in the fiber $$F_{p^0}=\pi^{-1}(p^0_i)=\Spec( \kappa(p^0_i))\times_{\Spec( \kappa(q^0_i))}F_{q^0_i}$$.
 
 Let $\Delta_{\tau}:=\Delta_{|\tau}$ be the restriction of $\Delta$ to $\tau$ which determines   the induced decomposition of ${\tau}$. 
 
 The fiber of $$F_{p^0}=\Spec( \kappa(p^0))\times_{\Spec( \kappa(q^0))}F_{q^0}=Y\times_X \Spec( \kappa(p^0))$$ of $\pi: Y=X\times_{X_\sigma}X_\Delta\to X$ is isomorphic to the fiber of the induced morphism $$X_{\Delta_{\tau}}^{ \kappa(p^0)}\to X_{\tau}^{ \kappa(p^0)}$$ over $p^0=\Spec( \kappa(p^0))$. 
Moreover the natural morphism $F_{p^0}\to F_{q^0}$ is induced by the fiber square
$$
\begin{array}{cccc} X_{\Delta_{\tau}}^{ \kappa(p^0)} & \stackrel{\phi_\Delta}\to & X_{\Delta_{\tau}}^{ \kappa(q^0)} \\
  \pi \dar &   & \pi_A \dar \\ X_{\tau}^{ \kappa(p^0)} &\stackrel{\phi}{\to} & X_{\tau}^{ \kappa(q^0)}\\
 \dar &   &  \dar \\ \Spec({ \kappa(p^0)}) &{\to} & \Spec({ \kappa(q^0)})
   \end{array}.
$$

The above morphism is bijective on the toric orbits and their generic points, as they correspond to the faces of $\Delta$ or respectively $\sigma$.
Then the inverse image of the point $q_j\in F_{q^0}\subset X_{\Delta_{\tau}}^{ \kappa(q^0)}$ corresponds to a unique face in $\Delta(1)$ and a unique  point $p$ in $ F_{p^0}\subset X_{\Delta_{\tau}}^{ \kappa(p^0)}$.

Hence the inverse image of the toric divisor $E_i^{\Delta}$ with the generic point $q$ is  the unique exceptional  divisor $E_i$ with the generic point $p=p_j$ over $q$. 

(2) The same reasoning shows that the inverse image $\phi_\Delta^{-1}(\overline{s^\Delta_j})=\overline{s}_j$ of the closure of a  toric stratum  $\overline{s^\Delta_j}$ on $X_\sigma$  determines a unique stratum $s_j$ on $Y_U$.
 We use  the same  relation for the fibers. 
$$\pi^{-1}({p^0_j})=F_{p^0_j}=\Spec( \kappa(p^0_j))\times_{\Spec( \kappa(q^0_j))}F_{q^0_j}$$
where $p_j$ is the generic point of $s_j$, $q_j=\phi(p_j)$,
$p^0_j=\pi(p_j)$  and $q^0_j=\pi_A(q_j)$.



%

(3) 
 We need to show first that
$$\cO_Y(nE_i)=\cO_Y(\cO_{X_\Delta}(nE^\Delta_i))$$ By the above, the generic point $p$ of $E_i$ is exactly the generic point of the fiber $\phi^{-1}(q)$.  The induced homomorphism of the completions of the local rings is given by
  $$\widehat{\cO_{X_\Delta,q}}\to \widehat{\cO_{Y,p}	}=\widehat{\cO_{X_\Delta,q}}\otimes_{ \kappa(q)}	 \kappa(p)$$
Thus we get  $$m^n_{q}\cdot{\cO_{X_\Delta,p}}=m^n_{q}\subset \cO_{Y,p}	.$$ Both points $p$ and $q$ admit a regular neighborhood and its local rings are DVR defining the valuation $\nu_i$ of $E_i$, and $\nu^\Delta_i$ of $E^\Delta_i$.

One verifies that $\cI_{\nu_i,a,Y}=\cO_Y\cdot \cI^\Delta_{\nu_i,a}$. First observe that the valuation center of $\nu_i$ on $Y$ can be described as $$Z_Y(\nu_i)=V_Y(\cI_{\nu_i,a,Y})=E_i=\phi^{-1}(E_i^\Delta)=V(\cO_Y\cdot \cI^\Delta_{\nu_i,a})$$ 

For any point $p'\in Z(\nu_i)$, and its image $q'=\phi(p)\in 
Z(\nu^\Delta_i)$ we have 
 $$\widehat{\cO_{Y,p}	}=\widehat{\cO_{X_\Delta,p}}\otimes_{ \kappa(p)}	 \kappa(q)[[u_1,\ldots u_k]].$$ 
Consequently   the monomial valuation $\nu^\Delta_i$ on $\cO_{X_\Delta,p'}=\cO_{X_\delta,p'}$ of  $E^\Delta_i$  associated with a vertex of $\Delta$ extends to a certain unique monomial valuation $\nu_i'$
on $\widehat{\cO_{Y,p}}$ such that
$$\widehat{\cI}_{\nu'_i,a,p'}=\cI^{\Delta}_{\nu_i,a}\cdot\widehat{\cO_{Y,p'}	}=\cI^{\Delta}_{\nu_i}\cdot (\widehat{\cO_{X_\Delta,p'}}\otimes_{ \kappa(p')}	 \kappa(q')[[u_1,\ldots u_k]])$$ 
which by flatness implies 
$$\cI_{\nu'_i,a,p'}=\cI^{\Delta}_{\nu_i,a}\cdot{\cO_{Y,p'}}$$ 
Note that the generic point $p$ of $E_i$ specializes at $p'$, and the generic point $q$ of $E^\Delta$ specializes at $q'$. Passing to $p$ and $q$ and localizing we obtain that
$$\cI_{\nu'_i,a,p}=\cI^{\Delta}_{\nu_i,a,q}\cdot \cO_{Y,p}=\cI_{\nu_i,a,p},$$
whence both valuations are equal $\nu_i=\nu'_i$. Thus $\cI_{\nu_i,a,p'}=\cI^{\Delta}_{\nu_i,a,q'}\cdot{\cO_{Y,p'}}$ and the vanishing locus of the ideal 
$$\cI_{\nu_i,a,Y_U}={\cO_{Y_U}}\cdot\cI^{\Delta}_{\nu_i,a,X_\Delta}$$
is irreducible by part (1) and defines the center of the valuation $\nu_i$.

Now, for any effective divisor $(E^\Delta)'=\sum a_iE^\Delta_i$, and its inverse image $(E_U)'=\sum a_i(E_U)_i$
 we have, by flatness \begin{align*} &\cO_{Y_U}(-(E_U)')=\bigcap \cI_{\nu_i,a_i,Y_U}= \cO_{Y_U}\cdot (\bigcap \cI_{\nu_i,a_i,X_\Delta})=\\ & =\cO_{Y_U}\cdot \cO_{X_\Delta}((E^\Delta)')=\cO_{Y_U}\otimes_{\cO_{X_\Delta}} \cO_{X_\Delta}((E^\Delta)')\end{align*}

In general, for any $(E^\Delta)'=\sum a_iE^\Delta_i$, we can find a nontrivial monomial $m\in P_\sigma=\sigma^\vee\cap M$
such that for $n\gg 0$, $$\cO_Y((E^\Delta)')= m^{-n}\cO_Y((E^\Delta)'-n\cdot\divv(m)),$$ where $-((E^\Delta)'-n\cdot \divv(m))$ is effective.
Consequently
\begin{align*}
&\cO_{Y_U}((E_U)')=m^{-n}\cO_{Y_U}((E_U)'-n\cdot \divv_Y(m))=\\
&=\cO_{Y_U}\cdot m^{-n}\cdot \cO_{X_\Delta}((E^\Delta)'-n \cdot \divv(m))=\cO_{Y_U} \cO_{X_\Delta}(E^\Delta)=\cO_{Y_U}\otimes_{\cO_{X_\Delta}} \cO_{X_\Delta}(E^\Delta).
\end{align*}

(4) and (5) Since the morphism $U\to X_\sigma$ is affine, and thus $Y_U\to X_\Delta$ is such we have \begin{align*} &B_{U+}=\Spec_{Y_U}(\bigoplus_{E\in \Cl(Y_U/U)} \cO_{Y_U}(E))=
\Spec_{Y_U}(\cO_{Y_U}\cdot(\bigoplus_{E^\Delta\in \Cl(X_\Delta/X_\sigma)} \cO_{X_\Delta}(E^\Delta))=
\\&=\Spec_{X_\Delta}\cO_{Y_U}\otimes_{\cO_{X_\Delta}}(\bigoplus_{E^\Delta\in \Cl(X_\Delta/X_\sigma)} \cO_{X_\Delta}(E^\Delta))=Y_U\times_{X_\Delta}B_+(\pi_A)_+=\\&
=(U\times_{X_\sigma}X_\Delta)\times_U\times_{X_\sigma}B_+(\pi_A)=  U\times_{X_\sigma}B_+(\pi_A). \end{align*}


By definition, and since all the schemes are normal
$$\pi_{A*}(\cI_{\nu^\Delta_i,a,X_\Delta})=\pi_{A*}(\cO_{X_\Delta}(a,E_i^\Delta))= \cI_{\nu_i^\sigma ,a, X_\sigma}\subset \pi_{A*}(\cO_{X_\Delta}(E^\Delta))=\cO_{X_\sigma}$$ are the toric  ideals generated by monomials associated with the toric valuation $\nu_i^\sigma$.

Similarly

 $$\pi_*(\cI_{\nu_i,a,Y}))=\pi_*(\cO_{Y}(aE))= \cI_{\nu_i,a,X}.$$


By the above and since $\psi$ is flat, we have 
\begin{align*} &\cI_{\nu_i,a,U}=\pi_*(\cO_{Y_U}(-E_i))=\pi_*(\cO_{Y_U}\cdot \cO_{X_\Delta}(-E^\Delta_i))=\\&\pi_*(\cO_{Y_U}\otimes \cO_{X_\Delta}(-E^\Delta_i))= \cO_U\otimes \pi_{A*}(\cO_{X_\Delta}(-E_i^\Delta))=\cO_U\cdot\cI_{\nu_i,a,X_\sigma}
 \end{align*}
is a locally monomial valuation.

 Thus for $E=\sum a_iE_i$, we have 
 $$\pi_*(\cO_{Y_U}(E))=\bigcap \cI_{\nu_i,a_i,U}= \bigcap \cO_U\cdot\cI_{\nu_i,a,X_\sigma}=\cO_U\cdot \bigcap \cI_{\nu_i,a_i,X_\sigma}=\cO_U \cdot \pi_*(\cO_{X_\Delta}(E))$$

Hence, by the above \begin{align*} & B_{U}=\Spec_U(\bigoplus_{E\in \Cl(Y_U/U)} \pi_*(\cO_{Y_U}(E))=\Spec_U(\cO_U\cdot(\bigoplus_{E^\Delta\in \Cl(X_\Delta/X_\sigma)} \pi_*(\cO_{X_\Delta}(E^\Delta))= \\ 
&=\Spec_U(\cO_U\otimes_{\cO_{X_\sigma}}(\bigoplus_{E^\Delta\in \Cl(X_\Delta/X_\sigma)} \pi_*(\cO_{X_\Delta}(E^\Delta))=U\times_{X_\sigma}B(X_\Delta/X_\sigma).
\end{align*}

\end{proof}

\subsubsection{Local description of the exceptional divisor}
As a corollary from Lemma \ref{local} we obtain:
  \begin{lemma} \label{local2}  Let $\pi: Y\to X$ be a locally toric morphism of locally toric schemes.
  Let $\pi_B: 
  B\to X$ be its full cobordization.
  
   For any point $p\in X$, there is a toric chart  $\phi_U: U\to X_\sigma$, such that  for the induced
  morphism  $B_{U}=\pi_{B}^{-1}(U)\to X_\Delta$,
  there is a bijective correspondence between
	the strata $s=\phi^{-1}(s_\tau)$ of the divisorial stratifications of the exceptional divisor
	$D_{B_U}$ on $B_{U}$, (respectively $D_{B_{U+}}$ on $B_{U+}$) and the strata $(s_\tau)$ of the exceptional divisor 
 $D_{B(X_\Delta/X_\sigma)}$ on $B(X_\Delta/X_\sigma)$ (respectively $D_{B(X_\Delta/X_\sigma)_+}$ on $B(X_\Delta/X_\sigma)_+$ ) .
 
 		\end{lemma}
 \begin{proof}
  The reasoning is the same as in the proof of Lemma \ref{local}(2). We can assume, as in  the proof of Lemma \ref{local}(2), that the inverse image $\phi_U^{-1}(s_i)\subset U$ consists of a single stratum.
  
   By Lemmas \ref{local} and \ref{functor}, we  have the following fiber square diagram for the cobordizations, with horizontal morphisms being regular:
$$
\begin{array}{cccc} B_U & \stackrel{\phi}\to & B(X_\Delta/X_\sigma)  \\
  \pi_U \dar &   & \pi_B \dar \\ U &\stackrel{\phi_U}{\to} & X_\sigma=\AA^n,  
   \end{array},
$$
and the analogous fiber square for $B_{U+}$.
Consequently, the inverse image of the exceptional divisor on
 $D_{B(X_\Delta/X_\sigma)}$ is the exceptional divisor $D_{B_U}$. Its components are of the form $V{B(X_\Delta/X_\sigma)}(t_i^{-1})$ and are associated with the components $E^\Delta_i$.  Their inverse images are  the irreducible components $V{B_U}(t_i^{-1}))$ corresponding to the exceptional components $E_i=\phi^{-1}(E^\Delta_i)$.

Since $\phi_U$ is regular, the inverse image $\phi^{-1}(\overline{s})$ of the closure $\overline{s}$ of any stratum $s$ of $D_{B(X_\Delta/X_\sigma)}$ is normal. Thus it is the disjoint union of the irreducible components.

To prove that $\phi^{-1}(\overline{s})$ is irreducible on $Y_U$, we need to show that there is a single generic point $p$ in the fiber $\phi^{-1}(q)$ over the generic point $q$ of $s$, and such that $p$ is of the same codimension in $U$ as $s$ in $B(X_\Delta/X_\sigma)$ . This can be reduced to the problem of the morphism of the fibers
$$\pi^{-1}({p^0})=F_{p^0}\to \pi^{-1}({q^0})=F_{q^0},$$
where $p^0=\pi_B(p)$, and $q^0=\pi_U(q)$ are the generic point of the relevent strata.

 But this follows from  the relation for the fibers of toric morphisms, as in  the proof of Lemma \ref{local}(2), 
$$\pi^{-1}({p^0_i})=F_{p^0_i}=\Spec( \kappa(p^0_i))\times_{\Spec( \kappa(q^0_i))}F_{q^0_i}.$$

 \end{proof}

\subsection{Description  of cobordization of locally toric morphisms}
\subsubsection{Local functoriality of relative Cox spaces for smooth morphims}
\begin{proposition} \label{functor} Let $\pi: Y\to X$ be a proper birational locally toric morphism of locally toric varieties over a field $\kappa$. Let $\phi: X'\to X$ be a regular morphism over $\kappa$, and $\pi': Y'\to X'$ will be the base change. Then 
for any $p'\in Y'$ there are open neighborhoods  $U'$ of $p'$, and $U$ of $p:=\phi(p')$, with the induced smooth morphism $\phi_{|U'}: U'\to U$ such that $$B(Y_U/U)\times_XX'\simeq B(Y'_{U'}/U')\quad B(Y_U/U)_+\times_XX'\simeq B(Y'_{U'}/U')_+$$

Thus the full cobordization and cobordization of proper birational locally toric morphisms  are functorial for regular morphisms up to torus factors.

\end{proposition}
\begin{proof}  This is a direct consequence of Lemma \ref{local}
and  definition   of locally toric morphisms \end{proof}

\subsubsection{Local description  of cobordization of locally toric morphisms}

\begin{lemma} \label{local12} Let $\pi: Y\to X$ be a proper birational locally toric morphism over a field $\kappa$, and $\pi_B: B\to X$ be its full cobordization. Then
\begin{enumerate}
\item $B_+\subset B$ is the natural open immersion.
\item  For any point $p\in X$ there is an open neighborhood $U$ of $p$, with a toric chart $U\to X_\sigma$,   and the torus $$T_{B\setminus B_U}:=\Spec(\,\kappa[x_i, x_i^{-1}\mid E_i\subset B\setminus B_U \,\,]\,),$$ and an induced   regular morphism $$B_U=\sigma^{-1}(U)=B(Y_U/U)\times T_{B\setminus B_U}\,\, \to \,\, X_\sigma \times \AA^k\times T_{B\setminus B_U}$$
\item If $X$ is regular then $B$ is regular.
 \end{enumerate}
\end{lemma}
\begin{proof} By Lemma \ref{local}, the problem reduces locally to a toric situation via toric chart $U\to X_\sigma$.

(1) By Lemma \ref{cover4}, $B(X_\Delta/X_\sigma)_+\hookrightarrow B(X_\Delta/X_\sigma)$ is an open inclusion. Thus, by Lemma \ref{local}, 
 $B_+\subset B$ is  also such.

(2) Also  locally by  Lemma \ref{open},  we can write $B_U=B(Y_U/U)\times T_{B\setminus B_U}$. On the other hand, by Lemma \ref{local}(4) $B(Y_U/U)\to B(X_\Delta/X_\sigma)$ is regular. Finitely
by Lemma \ref{cover4}(2),
$B(X_\Delta/X_\sigma)=X_\sigma\times \AA^k$. 

(3)  Follows from (2).


\end{proof}

\subsection{Local description of cobordization}

\subsubsection{Cobordization of locally monomial maps}

\begin{definition} Let $X$ be a locally monomial scheme over a  field $\kappa$. We say that $u_1,\ldots,u_k$ is a locally toric system of parameters on $X$ if there is  a chart $\phi: U\to X_\sigma$, and a local system of toric parameters $x_1,\ldots,x_k$ on $X_\sigma$, such that $u_i=\phi^*(x_i)$.

\end{definition}

As a Corollary from Proposition \ref{functor}, and Lemma \ref{local12} we obtain the following:
\begin{theorem} \label{regular33} Let $\pi: Y\to X$ be  a proper birational locally toric morphism  of locally toric schemes over a field $\kappa$ . Then  locally on $X$ we can write up to torus factors $$\cA_{Y/X}=\pi_*(\cC_{Y/X})=\cO_X[t_1^{-1},\ldots,t_k^{-1},u_1{\bf t}^{\alpha_1},\ldots,u_k{\bf t}^{\alpha_k}],$$ 
where 
\begin{enumerate}
\item $u_1,\ldots,u_k$ is a {locally toric system of parameters } on an open $U\subset X$ defining a toric chart for the morphism $\pi$, 
\item ${\bf t}^{\alpha_i}:=t_1^{a_{i1}}\cdot\ldots\cdot t_k^{a_{ik}}$, with $a_{ij}:=\nu_i(u_j)\geq 0$. 
\end{enumerate}
In particular, if $X$ is regular then $B$ and $B_+\subset B$ are  regular.\qed
 \end{theorem}
 \begin{proof} We use the fact from Lemma \ref{local}, that the valuations are locally monomial with respect to $u_1,\ldots,u_k$, and  Lemma \ref{cover4}(1).

 \end{proof}



	


\subsubsection{The cobordization of monomial morphisms }\label{local22}
Let 
$Y\to X$ be  a proper birational locally toric morphism over $\kappa$.

Let  $x_1,\ldots,x_n$ be a system of local parameters  at a point $p$ on a locally toric  $X$ defining a toric chart for a  $Y\to X$. Then the full  cobordization  of $\pi:Y\to X$  can be represented as: $$B=\Spec_X(\cO_X[t_1^{-1},\ldots,t_k^{-1},x'_1,\ldots,x'_n]/(x'_1{\bf t}^{-\alpha_1}-x_1,\ldots,x'_k{\bf t}^{-\alpha_k}-x_k).$$
Thus $B=V(x'_1{\bf t}^{-\alpha_1}-x_1,\ldots,x'_k{\bf t}^{-\alpha_k}-x_k)\subset X\times \AA^{n+k}$ is locally a closed  subscheme of $X\times \AA^{n+k}$ defined by a system of local parameters. It is regular for a regular $X$.
Consequently, the full cobordization $B\to X$  can be described  by a single chart up to a torus factor with the following coordinates:
 \begin{itemize} 
\item $t_i^{-1}$  for $i=1,\ldots, k$ is the inverse of the coordinate $t_i$ representing the action of  torus $T=\Spec(\kappa[\Cl(Y/X)]=\Spec(\kappa[t_1,t_1^{-1},\ldots,t_n,t_n^{-1}])$.
\item $x_i' = x_i \cdot {\bf t}^{\alpha_i}$ for $1 \leq i \leq k$, and
\item $x_j' = x_j$ for $j>n$.
\end{itemize}
 The   open subsets  $B_{x'_i}=B\setminus V(x'_i)$, associated with  the  forms $x'_i=x_i{\bf t}^{\alpha_i}$  cover the  cobordization  $B_+=B\setminus V(\cI_{\irr})$ producing several "charts" similarly  to the standard blow-up.  These open affine subsets can be conveniently described by using toric geometry. They correspond  to the maximal faces of the decomposition $\Delta$ of the cone $\sigma$ associated with the local toric chart. 

If $\pi: Y\to X$ is the cobordant blow-up of a locally monomial $\cJ$, where $\codim(V(\cJ)\geq 2$, then the  subset $B_+$, by Lemma \ref{blow}, can be described as $B_+=B\setminus V(\cJ {\bf t}^\alpha)$, where  $\alpha=(a_1,\ldots,a_k)$, and $a_i$ are the coefficients of the exceptional divisor $E=\sum a_iE_i$  of $\pi: Y\to X$. In this case, the charts of $B_+$ can also be interpreted by the vertices of the Newton polytope of $\cJ$.(See Example \ref{111})

\begin{remark} In the particular case, when considering the stack-theoretic quotients of the blow-up of a locally monomial ideal on a regular scheme, one obtains the definition of
a {\it multiple weighted blow-up} $\Bl_{\cJ}=[B_+\sslash T]$ introduced in \cite{AQ} by Abramovich-Quek via the Satriano construction in \cite{Satriano}. The more general definition of $\Bl_{\cJ,b}$ is discussed in Section \ref{AQ}.

\end{remark}

	

\subsubsection{Weighted cobordant blow-ups} \label{weighted}
Recall that the weighted stack-theoretic blow-ups were considered in the context of resolution in \cite{Marzo-McQuillan} and \cite{ATW-weighted}. The definition of the {\it weighted cobordant blow-up } was introduced in \cite{Wlodarczyk22}.
One can view these notions from the more general perspective of Cox cobordant blow-ups or the multiple weighted blow-ups of Abramovich-Quek from \cite{AQ}.
\begin{definition} \label{weighted4}
 Let $(x_1,\ldots,x_k)$ be a partial system of local parameters on a regular scheme $X$.
Let $\cJ$ be a center of the form $(x_1^{a_1},\ldots,x_k^{a_k})$, where $a_1\leq a_2\leq\ldots\leq a_k$ are positive integers, and $k>1$. Let $\pi: Y\to X$ be the normalized blow-up of $\cJ$. By the {\it  weighted cobordant blow-up} of $\cJ$ we mean the cobordization $B_{\cJ+}\to X $ of  $\pi: Y\to X$.
\end{definition}
The corresponding monomial ideal on the toric chart  $\AA_{\kappa}^k$ defines a piecewise linear function $G:=\min_i(a_ie_i^*)$ on the regular coordinate cone $\sigma=\langle e_1,\ldots,e_k\rangle$, where $e_i^*(e_j)=\delta_{ij}$.  The functions $a_ie_i^*$ determine the ray $$\rho:=\{v\in \sigma \mid a_1e_1^*(v)=\ldots
=a_ke_k^*(v)\}.$$ The ray $\rho$ is generated by the primitive vector $$w=(w_1,\ldots,w_k)=w_1e_1+\ldots+w_ke_k,$$ with relatively prime components and such that $$w_1a_1=\ldots=w_ka_k.$$

The normalized blow-up of $\cJ$  is described by the decomposition $\Delta$   of $\sigma$ into  maximal subcones where $G$ is linear. 

Thus $\Delta$ is  the star subdivision $\rho\cdot \langle e_1,\ldots,e_k\rangle$ at a ray $\rho$.
 The vector $w$ determines  the valuation $\nu_E$ of the unique irreducible exceptional divisor.


 Then, by Lemma \ref{cover4}(1),  the  full cobordant  blow-up  of  $X$ at the center $\cJ$  is defined by 
 $$B_\cJ=\Spec_X({\cO_X}[{ t}^{-1}, { t}^{w_1}{x_1},\ldots, { t}^{w_k}{x_k}]). $$ 
Here we have  $$w_i=\nu_E(x_i)=(w,e_i^*)$$
 The cobordant weighted blow-up is simply 
 $(B_\cJ)_+=B\setminus V(\sigma^\circ(\cJ))$, where, by Lemma \ref{blow},  we have
 $\sigma^\circ(\cJ)={ t}^a\cdot \cJ$, where $a=\nu(\cJ)$.
 Thus $$\sigma^\circ(\cJ)=(x_1^{a_1}{ t}^{a_1w_1},\ldots,x_k^{a_k}{ t}^{a_kw_k}).$$
  
 %
 We see that the cobordant weighted blow-up is  the cobordization of an ordinary toric weighted blow-up corresponding to the star subdivision at the center $v\in \sigma$. We will discuss this construction in the context of the blow-ups of valuative $\QQ$-ideals in Section \ref{weighted2}.

Observe that both notions: the one in Definition \ref{weighted4}, and the one given by the formula $B_\cJ=\Spec_X({\cO_X}[{ t}^{-1}, { t}^{w_1}{x_1},\ldots, { t}^{w_k}{x_k}]),$ as in \cite{Wlodarczyk22}, are different in the trivial case $k=1$ and the blow-up of $(x_1^{a_1})$. Then $Y\to X$ is an isomorphism, and $B=B_+=B_-\simeq Y\simeq X$.
However the formula from \cite{Wlodarczyk22} gives us
$$B=\Spec_X({\cO_X}[{ t}^{-1}, { t}^{w_1}{x_1}]),$$ which defines the isomorphism of the quotients: $$B/G_m\simeq B_+/G_m\simeq B_-/G_m\simeq Y\simeq X.$$ In this case, $B_+$  is  a locally trivial $G_m$-bundle. So both constructions of $B_+$ differ locally by the  torus factor. 

\subsection{Geometric quotients for locally toric morphisms}
In general, when considering the cobordization $B_+$ of a locally toric morphism $\pi: Y\to X$, one obtains 
the good quotient $B_+\sslash T\simeq Y$. Proposition \ref{geometric},  below shows that if $X$ is regular then, by replacing $B_+$ with an open subset $B^s\subseteq B_+$ one obtains the geometric quotient $B^s/T$ with a proper birational morphism
$B^s/T\to B_+\sslash T\simeq Y$.
 xConsequently, $B^s$ has a geometric quotient $B^s/T$ with abelian quotient singularities and the transformation $B^s\to X$ can be  be used in the resolution instead of $B_+\to X$.

\begin{lemma} Let $\pi: Y\to X$ be a locally toric morphism, with $X$ regular. 

Then its cobordization $B_+$ determines the  geometric quotient $B_+\to B_+/T\simeq Y$ iff $ Y$ has abelian quotient singularities.

\end{lemma}
\begin{proof} The problem is local and can be reduced to the toric morphism $\pi: X_\Delta\to X_\sigma$ corresponding to the subdivision $\Delta$ of a regular cone $\sigma$. Then, by Lemma \ref{cover4}, the full cobordization $B$ of $\pi$ is a regular scheme corresponding to the cone $\Sigma_{B}$, and $B_+$ is its open toric subscheme. The natural projection $\sigma_{B_+} \to \Delta$ corresponds to the geometric quotient iff
$\Delta$ is a simplicial fan, and thus $Y$ has abelian quotient singularities.
	
\end{proof}

\begin{proposition}\label{geometric} Let $\pi: Y\to X$ be a proper birational locally toric morphism of locally toric schemes over a field, with $X$ regular. Then $B_+=\Cox(Y/X)_+\subset B=\Cox(Y/X)$
contains open maximal subsets $B^s\subset B_+$ admitting geometric quotient $B^s\slash T_B$ with the projective birational morphism $$ B^s\slash T_B\to B_+/T_B=Y.$$
	
\end{proposition}
\begin{proof} Let $E_1,\ldots, E_k$ be the irreducible exceptional divisors of $\pi: Y\to X$, and $\nu_1,\ldots,\nu_k$ be the associated exceptional valuations on $X$. By Lemma \ref{local}(5), the valuations are locally toric on $X$.
Consider the sequence of the blow-ups at the valuations $\nu_i$ as in \cite[Proposition 8.16.6]{Wlodarczyk-functorial-toroidal}. These are precisely the normalized blow-ups of $\cI_{\nu_i,a,X}$ for a sufficiently divisible $a$. 

Locally,  in the compatible toric charts, $U\to X_\sigma$ the sequence of the blow-ups correspond to
a sequence of the star subdivisions at  the vertices $\Ver(\Delta)\setminus \Ver(\sigma)$ (see \cite[Lemma 7.3.9]{Wlodarczyk-functorial-toroidal}). As the result we create a new subdivision $\Delta'$ of $\Delta$ with $\Ver(\Delta')=\Ver(\Delta)$. This decomposition is simplicial. Indeed, let $\delta_0$ be any cone in $ \Delta'$. By the property of the star subdivisions, for any vertex $v_0\in \Ver(\delta)\setminus \Ver(\sigma)$ one can write $\delta_0=\langle v_0 \rangle +\delta_1$, where $\delta_1$ is a face of $\delta_0$ of codimension one in $\delta_0$, and $v_0$ is linearly independent of $\Ver(\delta_1)$. We can run this  argument inductively until we can represent $\delta_0$ as $\delta_0=\langle v_0,\ldots,v_r \rangle +\delta_r$, where $v_0,\ldots,v_r \in \Ver(\delta_0)$  are linearly independent of $\Ver(\delta_r)\subset\Ver(\sigma)$. Thus all the vertices of $\delta_0$ are linearly independent.

By construction,  the valuations of the exceptional divisors corresponding  to $\Ver(\Delta)\setminus \Ver(\sigma)$ remain unchanged.  Then $B(Y/X)=B(Y'/X)$. On the other hand, by the description of the toric case from Lemma \ref{cover4}(4), we obtain the open inclusions $$B^s:=B(Y'/X)_+\subseteq B(Y/X)_+=B_+\subseteq B(Y/X)=B.$$

\end{proof}

\section{Cobordant resolution  of singularities } \label{free}
\subsection{The dual complex of the exceptional divisor}
One can extend the considerations and the results from Section \ref{dual}. 
\subsubsection{The exceptional divisor} Let $\pi: Y\to X$ be a proper birational locally toric morphism, where $X$ is a regular scheme over a field $\kappa$, and $E_1,\ldots, E_k$ be the irreducible components of the exceptional divisor $E$ of $\pi$. Let $\pi_{B}: B\to X$ be the full  cobordization of $\pi$. 
By Theorem \ref{regular}, $B$ is regular and there is an SNC divisor $D=V_{B}(t_1^{-1}\cdot\ldots\cdot t_k^{-1})$ with irreducible components $D_i=V_B(t_i^{-1})$. So is  the divisor $D_+=D_{|B_+}$ on  $B_+$. Moreover, the exceptional divisor $E$ of $\pi: Y\to X$ is  locally toric.
 
 We can associate with the  SNC divisors $D$ on $B$, $D_+$ on $B_+$, and with the divisor $E$ on $Y$  the divisorial stratifications $S_{D}$, $S_{D_+}$, and $S_E$, extending the definitions from Section \ref{dual}.
 The strata of $S_D$ are defined by the  irreducible  components  of the locally closed sets : $$\bigcap_{i\in I} D_i\setminus \bigcup_{j\in J} D_j,$$ where $I\cup J=\{1,\ldots,k\}$.
 Replacing $D_i$ with $D_{i+}$  we obtain the definition for $S_{D+}$.
Consequently, any stratum $s\in S_{D_+}$ extends to a stratum in $S_D$. 

 The closures of the strata  $s^E\in S_E$ are defined by the irreducible components of the intersections  $\bigcap_{i\in I} E_i$. The strata $s_E$ are obtained by removing from $\overline{s}^E$ the proper closed subsets $\overline{s'}^E$.
  The stratifications $S_{D}$, $S_{D_+}$ and $S_E$  determine the dual simplicial complexes $\Delta_D$, $\Delta_{D_+}$, and $\Delta_E$. Since $D$ and $D_+$ are SNC, the simplices in $\Delta_D$, (respectively $\Delta_{D_+}$) are in the bijective correspondence with the strata of $S_D$ (respectively of $S_{D_+}$). Moreover, by the above, $\Delta_{D_+}$ is a subcomplex of $\Delta_{D}$ corresponding to the strata of $S_D$ which intersect $B_+\subset B$. Also, under this identification
  $\Ver(\Delta_{D})=\Ver(\Delta_{D+})$.

 The divisor  $E$  is usually not SNC, and the strata alone  do not determine the faces of $\Delta_E$.
 The vertices $v_i$ of $\Delta_E$  correspond to the divisors $E_i\leftrightarrow v_i$.
 The simplices $\sigma=\Delta(e_i\mid i\in I)$ in $\Delta_E$  correspond to the 
 pairs $(s^E_\sigma, E_\sigma)$ consisting of a stratum $s_\sigma\in S_E$ and a 
 collection of divisors $E_\sigma=\{E_i\mid i\in I)$, such that $\overline{s_\sigma}^E$ is an irreducible component of $\bigcap_{i\in I} E_i$. Thus, in this case, the correspondence between the faces of $\Delta_E$ and the strata of $S_E$ is not bijective, and the closures of strata could be represented by the intersections of components $\bigcap_{i\in I} E_i$ defined by different subsets $I$. (See also Section \ref{dual}.)
 
 Summarizing we have
  \begin{lemma} A simplex $\sigma$ in $\Delta_{B}$ (respectively  $\Delta_{D_+}$, $\Delta_{E}$)  is represented  by a pair $$(\{D_i\mid i\in I\},\quad (\bigcap D_i)_0)$$ consisting of a collection of the irreducible divisors $D_i$, (respectively $D_{i+}$,  $E_i$) which have a nonempty intersection and an irreducible  component $(\bigcap D_i)_0$ of  $\bigcap D_i$ (respectively $\bigcap D_{i+}$,\quad   $\bigcap E_{i}$). \qed	
\end{lemma}

  \begin{corollary} \label{local3} With the previous assumptions and notations:
 \begin{enumerate}
 \item There is a bijective correspondence between the divisors $D_i$, $D_{i+}$, $E_i$, and the valuations $\nu_i$.
 \item 
 \begin{enumerate}
  
 \item If $s\in S_D$ then $s$ is a component of a locally closed set
 $$\bigcap_{D_i\supseteq s}D_i\setminus \bigcap_{D_i\not\supseteq s}D_i$$ 
\item  The image $\pi_B(s)$ is closed. It is an irreducible component  of the closed set
$$ \pi_B(\bigcap_{D_i\supseteq s}D_i\setminus \bigcap_{D_i\not\supseteq s}D_i)=$$ $$
\bigcap_{D_i\supseteq s} Z_X(\nu_i)=\bigcap_{D_i\supseteq s} \pi_B(D_i)=\bigcap_{D_i\supseteq s} \pi(E_i)$$   
 where  $Z_X(\nu)$ denotes the center of a valuation $\nu$ on $X$.
 Moreover, the sets $\bigcap_{D_i\supseteq s} D_i$ are locally irreducible over $X$. 
 
 \item The morphism $\pi_B$ determines 
  a bijective correspondence between the strata defined by the irreducible components of $\bigcap_{D_i\supseteq s}D_i\setminus \bigcap_{D_i\not\supseteq s}D_i$ and
 the irreducible components of $\bigcap_{D_i\supseteq s}Z_X(\nu_i)$.  
\end{enumerate}

 \item The morphism $\pi_{B_+,Y}$ determines a bijective correspondence between the components of $\bigcap_{i\in I} D_{i+}$ and the components of $\bigcap_{i\in I} E_i$. This correspondence   
 defines the isomorphism of the dual  complexes
 $\Delta_{D_+}\simeq \Delta_E$. 
 \item The morphism of the stratifications $S_{D_+}\to S_D$  maps  a stratum $s_+$ of $S_{D_+}$ into an open subset of a stratum $s$ of $S_D$ . It determines the inclusion  of the 
 dual complexes $\Delta_{B_+}\hookrightarrow \Delta_B$.
 \end{enumerate}

 \end{corollary}
\begin{proof} (1) The correspondence follows from Lemmas \ref{Ex2}, and \ref{irreducible}.

(2)-(5) By Lemma \ref{local} and  \ref{local2}, we can reduce the situation locally to the toric case, where we use Lemmas \ref{closed}, and \ref{special5}, and Corollaries \ref{special15} \ref{special6} and \ref{special7}.

	\end{proof}

\subsubsection{Dual complex of valuations of a locally toric morphism} 
 
Let $N=\{\nu_1,\ldots,\nu_k\}$ be the set of the exceptional valuations of $\pi:Y\to X$. The vertices of $\Ver(\Delta_E)$, and thus of $\Ver(\Delta_B)$ and  $\Ver(\Delta_{B+})$  are 
in the bijective correspondence with the valuations in $N$, and the exceptional divisors $E_i$, $D_{i+}$, and $D_i$: $$\nu_i\leftrightarrow E_i=Z_Y(\nu_i)\leftrightarrow D_{i+}\leftrightarrow  D_i.$$

 Consequently, one can associate with the faces of $\Delta_E$, $\Delta_{B+}$, and $\Delta_B$ the subsets of $N$.
 This determines   the complexes $\Delta^N_B, \Delta^N_{B+}, \Delta^N_{E}$, called the {\it dual valuation complexes}, together
 with natural isomorphisms of the simplicial complexes $$\Delta_B\to \Delta^N_B,\quad,\Delta_{B_+}\to \Delta^N_{B+},\quad \Delta_E\to \Delta^N_E$$
 Then, by Lemma \ref{local3}, $\Delta^N_E=\Delta_{B_+}$ determine the same subcomplex of $\Delta^N_B$.
 
 The simplices of the valuation complexes will be called the {\it valuation faces}. 
 The valuation faces come with natural face inclusions inherited from $\Delta^N_B, \Delta^N_{B+}, \Delta^N_{E}$.
 

By Lemma \ref{local3} we get:



\begin{lemma}\begin{enumerate}
 \item A valuation face $\sigma$ in $\Delta^N_{B}$ is represented by a pair $(\omega, Z^0_X(\omega))$  defined by the collection of valuations $\omega=\omega
_\sigma$ in $N$, such that $$Z_X(\omega):=\bigcap_{\nu\in \omega} Z_X(\nu)\neq \emptyset,$$ and an irreducible  component  $Z^0_X(\omega)$ of $Z_X(\omega)$. 
\item A simplex   $\sigma$ of $\Delta^N_E=\Delta^N_{B+}$  corresponds  to a subset  $\omega \subset N$,  such that $$Z_Y(\omega):=\bigcap_{\nu\in \omega} Z_Y(\nu)\neq \emptyset,$$  
and an irreducible component, denoted as $Z^0_Y(\omega)$ of the set $$Z_Y(\omega):=\bigcap_{\nu\in \omega} Z_Y(\nu).$$  
\end{enumerate} 
The face relations are given by the inclusions of the sets of valuations and the associated components.	\qed
\end{lemma}

\begin{remark}Thus, the dual valuation complexes could be thought of as ordinary dual  complexes of the exceptional divisors with the associated valuation structure so that the vertices define the relevant exceptional valuations, and the faces determine the sets of  the valuations.
\end{remark}


 \subsubsection{Dual complex  associated with a locally monomial ideal} 

If $\cJ$ is locally monomial  ideal on a regular scheme, such that $\codim(V(\cI))\geq 2$, then one can associate with $\cJ$ the normalized blow-up $\pi: Y\to X$, and the full cobordant blow-up $\pi_B: B\to X$ of  $\cJ$. The morphism  $\pi: Y\to X$ is locally toric, and  we shall call the dual complexes $\Delta_D$, $\Delta_{D+}\simeq \Delta_E$ and the corresponding dual valuation complexes $\Delta^N_D$, $\Delta^N_{D+}\simeq \Delta^N_E$ {\it associated  with $\cJ$}.

\subsection{Graded rings defined by the valuations}


\subsubsection{Graded rings defined by valuations} 
\label{grad}

 In the considerations below, let  $\omega=\{\nu_1,\ldots,\nu_r \}$ be a set of valuations on a regular scheme $X$. We associate with each valuation $\nu_i$ a dummy variable $t_i$ for $i=1,\ldots,r$. 
Set
$${\bf t}:=(t_1,\ldots,t_k) \quad \mbox{and} \quad  {\bf t}^{-1}:=(t_1^{-1},\ldots,t_k^{-1}).$$


Consider the partial componentwise order on $\ZZ_{\geq o}^r$.
For $\alpha:=(a_1,\ldots,a_r)\in \ZZ^r_{\geq 0}$ we define
 the ideals \begin{align}\cJ_{\omega}^\alpha:=\bigcap_{\nu_i\in \omega} \cI_{\nu_i,a_i}\subset \cO_X, \quad\quad \cJ_{\omega}^{>\alpha}:=\sum_{\beta>\alpha} \cJ_{\omega}^\beta. \end{align}  This determines  the $\ZZ^k_{\geq 0}$-graded Rees algebra $$\cA_{\omega}:=\bigoplus_{a\in \ZZ_{\geq 0}} \cJ_{\omega}^\alpha {\bf t}^\alpha\subset \cO_X[t],$$
 where ${\bf t}^{\alpha}=t_1^{a_1}\cdot\ldots\cdot t_r^{a_r}$, and the  associated gradation \begin{align}\gr_{\omega}(\cO_X)=\bigoplus_{a\in \ZZ_{\geq 0}} (\cJ_{\omega}^\alpha/\cJ_{\omega}^{>\alpha}){\bf t}^\alpha=\cA_{\omega}/(\cA_{\omega}\cap {\bf t}^{-1}\cA_{\omega}])=\cA_{\omega}[{\bf t}^{-1}]/ ({\bf t}^{-1}\cdot \cA_{\omega}[{\bf t}^{-1}]).\end{align} 
In particular, for $\alpha=0=(0,\ldots,0)$ we have locally on $X$:$$\cJ_{\omega}:=\cJ_{\omega}^{>0}=\cI_{Z_X(\omega)},$$
  where $$Z_X(\omega):=\bigcap_{i=1}^k Z_X(\nu_i).$$
  
 Then $\gr_{\omega}(\cO_X) $ is a sheaf of graded $\cO_X/\cJ_\omega=\cO_{V(\cJ_{\omega})}$-modules.   \begin{lemma} Assume the valuations in the set $\omega=\{\nu_1,\ldots,\nu_r \}$  are monomial for a certain partial system of local parameters $u_1,\ldots,u_n$ on  a regular scheme $X$.
 Then 
 \begin{enumerate}
 \item $\cJ_{\omega}=\sum_{i=1}^k \cI_{\nu_i,1}=(u_j\mid \nu_i(u_j)>0, {\rm for\,\, some} \,\,\nu_i\in \omega )$ , and

 \item $\gr_{\omega}(\cO_X)=\cO_{V(\cJ_{\omega})}[u_1{\bf t}^{\alpha_1},\ldots,u_k{\bf t}^{\alpha_k}],$ where $u_i\in \cJ^{\alpha_i}_{\omega}\setminus \cJ_{\omega}^{>\alpha_i}$, and $\alpha_i=(a_{i1},\ldots,a_{in})$, with
$\nu_i(u_j)=a_{ij}\in \ZZ_{\geq 0}$ for $i=1,\ldots,r$, and $j=1,\ldots,n$

  \end{enumerate} 	
   \end{lemma}

 \begin{proof} (1) Note that $\cI_{Z_X(\nu_i)}=(u_j\mid \nu_i(u_j)>0)$. Thus 
 $$\cJ_{\omega}= \cI_{Z_X(\omega)}=\sum_{\nu\in \omega}  \cI_{Z_X(\nu_i)}=(u_j\mid \nu_i(u_j)>0, {\rm for\,\, some} \,\,\nu_i\in \omega )$$


 (3) By definition of $A_\omega$, the equality (5), and the Proof of Lemma \ref{cover4}(1).
 \begin{align*}&\cA_{\omega}[{\bf t}^{-1}]=(\bigoplus_{a_i\in \ZZ_{\geq 0}} \,\, \bigcap^r_{i=1}\, \cI_{\nu_i,a_i}\,\,\cdot t_1^{a_1}\cdot\ldots\cdot t_r^{a_r})[t_1^{-1},\ldots,t_r^{-1}]=\\&=\cO_X[t_1^{-1},\ldots,t_r^{-1},u_1{\bf t}^{\alpha_1},\ldots,u_n{\bf t}^{\alpha_n}])=\cO_X[{\bf t}^{-1},u_1{\bf t}^{\alpha_1},\ldots,u_n{\bf t}^{\alpha_k}],\end{align*}	
	where $\alpha_i=(a_{i1},\ldots,a_{in})$, and
$\nu_i(u_j)=a_{ij}\in \ZZ_{\geq 0}$ for $i=1,\ldots,r$, and $j=1,\ldots,n$.
Thus by the equality (6): 
\begin{align*}& \gr_{\omega}(\cO_X)=\cA_{\omega}[{\bf t}^{-1}]/ ({\bf t}^{-1}\cdot \cA_{\omega}[{\bf t}^{-1}]=\\
 &= (\cO_X[{\bf t}^{-1},u_1{\bf t}^{\alpha_1},\ldots,u_n {\bf t}^{\alpha_k}])/({\bf t}^{-1})\simeq \\ & =(\cO_X/\cJ_\omega)[u_1{\bf t}^{\alpha_1},\ldots,u_n{\bf t}^{\alpha_k}].
	\end{align*}

 \end{proof}

  We shall call the corresponding scheme 
 $$\NN_{\omega}(X):=\Spec_{V(\cJ_\omega)}(\gr_{\omega}(\cO_X))=\Spec_{Z_X(\omega)}(\gr_{\omega}(\cO_X))$$ the {\it weighted normal bundle} of $X$ at the set of valuations $\omega$.  
 
 One can extend this to any valuation face in $\Delta_B^N$, associated with a full cobordant blow-up $B\to X$.

\begin{definition} \label{w} By the {\it weighted normal bundle} of $X$ at the valuation face $\omega\in \Delta_B^N$ we mean the scheme 
 $$\NN_{\omega}(X):=\Spec_{Z^0_X(\omega)}(\gr_{\omega}(\cO_X))$$	 over the component $Z^0_X(\omega)$ of $Z_X(\omega)$ associated with the face $\omega$.
\end{definition}

As  $\Delta_E^N\subset \Delta_B^N$ the above definition is also valid for any valuation face $\omega\in \Delta_E^N$. 

\subsubsection{The ideals of the initial forms} \label{initial}
With any function $f\in \cO_{X,p}$, regular at $p\in V(\cJ)$, such that $f\in \cJ_{\omega}^\alpha\setminus \cJ_{\omega}^{>\alpha}$, for a certain $a\in \NN$ one can associate the unique homogenous element, called  the {\it initial form} $$\inn_\omega(f)=(f+\cJ_\omega^{>\alpha}) \in (\cJ_\omega^\alpha/\cJ_\omega^{>\alpha}) {\bf t}_{\omega}^\alpha\subset \gr_{_\omega}(\cO_X).$$ Similarly, we associate with an
ideal sheaf $\cI$, the filtration $\cI_\omega^\alpha:=\cI\cap \cJ_\omega^\alpha$ and set $\cI^{>\alpha}_\omega=\cI\cap \cJ_\omega^{>\alpha}$.

We define the 
 {\it ideal of the initial forms} of $\cI$ to be the ideal 
$$\inn_\omega(\cI)=\bigoplus_{\alpha\in \ZZ{\geq 0}} \cI_\omega^\alpha/\cI^{>\alpha}_\omega =\bigoplus_{\alpha\in \ZZ_{\geq 0}} (\cI^\alpha_\omega+\cJ_{\omega}^{>\alpha})/\cJ_{\omega}^{>\alpha}\subset \gr_{{\omega}}(\cO_X)$$
on $N_{\cJ}(X)$.

For the ideal sheaf $\cI$, its {\it weak ideal of the initial forms on $N_{\cJ}(X)$} is given by
$$\inn^\circ_\omega(\cI)=\gr_{{\omega}}(\cO_X)\cdot \cI_\omega^{\alpha_0}/\cI^{>\alpha_0}_\omega\subset \gr_{{\omega}}(\cO_X),$$
where $\cI\subset \cJ_{\omega}^{\alpha_0}$, and $\cI \not\subset  \cJ_{\omega}^{>\alpha_0}$.
 
 \begin{remark} For any function $f\in \cO_X$,
 $$\inn_\omega(f)=\inn_\omega(\cO_X\cdot f)=\inn^\circ_\omega(\cO_X\cdot f).$$	
 \end{remark}

\subsubsection{Composition of gradations}

\begin{lemma} \label{special} Let $\omega=\{\nu_1,\ldots,\nu_r\}$ be a  set of  valuations which are monomial for a common partial system of local parameters  $u_1,\ldots,u_n$ on a regular $X$. Consider its partition into subsets $\omega_1=\{\nu_1,\ldots,\nu_s\}$, and  $\omega_2=\{\nu_{s+1},\ldots,\nu_r\}$.
Let  ${\bf t}:=(t_1,\ldots,t_r)$ ( respectively ${\bf t}_{\omega_1}:=(t_1,\ldots,t_s)$) be the set of the unknowns $t_i$ associated to the valuations $\nu_i\in \omega$, ( respectively and $\nu_i\in \omega_1$).
Let $\cJ_{\omega_1}=(u_1,\ldots,u_{\ell})$.
Then 
\begin{enumerate}
	\item The set  $\omega_2=\{\nu_{s+1},\ldots,\nu_r\}$ determines the set of monomial valuations $\overline{\omega_2}=\{\overline{\nu}_{s+1},\ldots,\overline{\nu}_r\}$
on the multi-graded ring 	$$\gr_{\omega_1}(\cO_X)=\cO_{V(\cJ_{\omega_1})}[u_1 {\bf t}_{\omega_1}^{\alpha_1},\ldots,u_\ell{\bf t}_{\omega_1}^{\alpha_n}]\simeq \cO_{V(\cJ_{\omega_1})}[u_1 ,\ldots,u_{\ell}],$$	
with the ideals 	$$\cI_{\overline{\nu}_i,a_i}=\inn_{\omega_1}(\cI_{\nu_i,a_i})$$

\item 
$\gr_{\overline{\omega_2}}(\gr_{\omega_1}(\cO_X)\simeq \gr_{\omega}(\cO_X).$

\item If $\cI\subset \cO_X$ then 
\begin{enumerate}
\item $\inn_\omega(\cI)=\inn_{\overline{\omega_2}}(\inn_{\omega_1}(\cI)).$
\item	$\inn^\circ_\omega(\cI)=\inn^\circ_{\overline{\omega_2}}(\inn^\circ_{\omega_1}(\cI)).$

\end{enumerate}

\end{enumerate}

\end{lemma}

\begin{proof} (1) For $j\leq \ell$, $\inn_{\omega_{w_1}}(u_j)$ is identified with 
$u_j$ in $\cO_{V(\cJ_{\omega_1})}[u_1 ,\ldots,u_{\ell}]$. Otherwise  if $j> \ell$, then $\inn_{\omega_{w_1}}(u_j)$ 
is  a parameter in  $\cO_{V(\cJ_{\omega_1})}=\cO_X/\cJ_{\omega_1}=\cO_X/(u_1 ,\ldots,u_{\ell})$.
Consequently $\nu_j$  determine the monomial valuations $\overline{\nu}_j$ on $\gr_{\omega}(\cO_X)$ with $\inn_{\omega_1}(\cI_{\nu_j,a})=(\cI_{\overline{\nu}_j,a}).$ 


(2) and (3) For the  multiindex $\alpha=(\alpha_1,\alpha_2)$, where $\alpha_i$ correspond to $\omega_i$ for $i=1,2$, consider a function $ f\in \cJ_{\omega}^\alpha\setminus \cJ_{\omega}^{>\alpha}$:
$$f\in \cJ_{\omega}^\alpha\setminus \cJ_{\omega}^{>\alpha}=\cJ_{\omega_1}^{\alpha_1}\cap \cJ_{\omega_2}^{\alpha_2}\setminus ( \cJ_{\omega_1}^{>\alpha_1}\cap \cJ_{\omega_2}^{\alpha_2}+\cJ_{\omega_1}^{\alpha_1}\cap \cJ_{\omega_2}^{>\alpha_2})$$

The ideal $\inn_{\omega_1}(\cJ_{\omega_2}^{\alpha_2})\subset \gr_{\omega_1}(\cO_X)$  is homogenous and 
$\inn_{\omega_1}(f)$ is in  $\alpha_1$-gradation of $\inn_{\omega_1}(\cJ_{\omega_2}^{\alpha_2})_{\alpha_1}\subset (\gr_{\omega_1}(\cO_X))_{\alpha_1}$:
$$\inn_{\omega_1}(\cJ_{\omega_2}^{\alpha_2})_{\alpha_1}=\frac{\cJ_{\omega_1}^{\alpha_1}\cap \cJ_{\omega_2}^{\alpha_2}}{\cJ_{\omega_1}^{>\alpha_1}\cap \cJ_{\omega_2}^{\alpha_2}}\subseteq 
(\gr_{\omega_1}(\cO_X))_{\alpha_1}=\frac{\cJ_{\omega_1}^{\alpha_1}}{\cJ_{\omega_1}^{>\alpha_1}}$$ 
 and 
 $$\inn_{\omega_1}(\cJ_{\omega_2}^{>\alpha_2})_{\alpha_1}=\frac{\cJ_{\omega_1}^{\alpha_1}\cap \cJ_{\omega_2}^{>\alpha_2}+\cJ_{\omega_1}^{>\alpha_1}\cap \cJ_{\omega_2}^{\alpha_2}}{\cJ_{\omega_1}^{>\alpha_1}\cap \cJ_{\omega_2}^{\alpha_2}}$$
Consequently, by the above, \begin{align*} & \inn_{\overline{\omega_2}}(\inn_{\omega_1}(f))\in \inn_{\overline{\omega_2}}(\inn_{\omega_1}(\cJ_{\omega_2}^{\alpha_2})_{\alpha_1})=\frac{(\inn_{\omega_1}(\cJ_{\omega_2}^{\alpha_2}))_{\alpha_1}}{(\inn_{\omega_1}(\cJ_{\omega_2}^{>\alpha_2}))_{\alpha_1}}=\\ &=\frac{\cJ_{\omega_1}^{\alpha_1}\cap \cJ_{\omega_2}^{\alpha_2}}{\cJ_{\omega_1}^{\alpha_1}\cap \cJ_{\omega_2}^{>\alpha_2}+\cJ_{\omega_1}^{>\alpha_1}\cap \cJ_{\omega_2}^{\alpha_2}}=\frac{\cJ_{\omega}^{\alpha}}{\cJ_{\omega}^{>\alpha}}=(\gr_{\omega_1}(\cO_X))_{\alpha}\end{align*}
On the other hand the initial form $$\inn_{\omega}(f)\in \frac{\cJ_{\omega}^{\alpha}}{\cJ_{\omega}^{>\alpha}}=(\gr_{\omega_1}(\cO_X))_{\alpha},$$
determines the same element: \begin{align*}\inn_{\omega}(f)=\inn_{\overline{\omega_2}}(\inn_{\omega_1}(f))\in (\gr_{\omega}(\cO_X))_{\alpha}=\gr_{\overline{\omega_2}}((\gr_{\omega_1}(\cO_X))_{\alpha_1})_{\alpha_2},\end{align*}
which implies  (3).

 \end{proof}

\subsection{The weighted normal bundles at valuations}


The following  extends a classical result of Huneke-Swanson on extended Rees algebras and smooth blow-ups \cite[Definition 5.1.5]{HS}, and the recent results of Rydh in \cite{Rydh-proj} and W\l odarczyk  in \cite[Lemma 5.1.4]{Wlodarczyk22} on the weighted normal cone.

\begin{lemma} \label{regul} Let $\pi: Y\to X$ be a locally toric proper birational morphism to
 a regular scheme $X$  over a field $\kappa$, with the exceptional components $E_i$, for $i=1,\ldots,k$ and let $\pi_B: B\to X$ be its full cobordization.  Then  for any  stratum $s=s_\omega \in S_{D}$ of the exceptional divisor $D=V(t_1^{-1}\cdot\ldots\cdot t_k^{-1})$ on $B$ and the corresponding valuation face $\omega$ in $\Delta_B^N$ there is an isomorphism:
$$s\simeq \NN_{\omega}(X)\times \check{\bf t}_{{\omega}}.$$
where $T_{\check{\bf t}_{\omega}}:=\Spec(\kappa[\check{\bf t}_{\omega},\check{\bf t}^{-1}_{\omega}])$, for  the set $\check{\bf t}_{\omega}$ of  the unknowns corresponding to the remaining  exceptional valuations
which are not in $\omega$.
\end{lemma}
\begin{proof} 
 
  We can  replace $X$ with its  open subset  and assume $Z_X(\omega)$ is  irreducible so that $ \pi_B(s)=Z^0_X(\omega)=Z_X(\omega)$ .

   By separating variables into ${\bf t}_{\omega}$ and $\check{\bf t}_{\omega}$ we can factor any monomial ${\bf t}^{\alpha}={\bf t}_{\omega}^{\alpha}\cdot \check{\bf t}_{\omega}^{\alpha}$ uniquely into the product of the relevant monomials ${\bf t}_{\omega}^{\alpha}$ and $\check{\bf t}_{\omega}^{\alpha}$ respectively in ${\bf t}_{\omega}$ and $\check{\bf t}_{\omega}$.  Then  by Corollary \ref{local3}(2), we can  write $ \overline{s}=V({\bf t}^{-1}_{\omega})$ in a neighborhood of  $\overline{s}$, and $s=V({\bf t}^{-1}_{\omega})\setminus V(\check{\bf t}^{-1}_{\omega})$ as there is only one component of $V({\bf t}^{-1}_{\omega})$ mapping to $ \pi_B(s)=Z^0_X(\omega)=Z_X(\omega)$.
    
By definition $B=\Spec(\cA[\bf{t}^{-1}])$, where    
 $$\cA=\bigoplus_{\alpha\in \ZZ^k_{\geq 0}} \cJ^\alpha {\bf t}^\alpha,\quad \cJ^\alpha:=\bigcap^k_{i=1}  \cI_{\nu_i,a_i}\subset \cO_X. $$

 Let $$\cA_\omega=\bigoplus_{\alpha} \cJ^\alpha_\omega {\bf t}^\alpha_\omega,\quad \cJ^\alpha_\omega:=\bigcap_{{\nu_i\in \omega} } \cI_{\nu_i,a_i}\subset \cO_X. $$
Thus for the open subset $B_{\omega}\subset B$ where $\check{\bf t}_{\omega}^{-1}$ are invertible we can write:
 $$B_{\omega}:=B_{\check{\bf t}_{\omega}^{-1}}=\Spec_X(\cA[\bf{t}^{-1}][\check{\bf t}_{\omega}])=\Spec_X(\cA_\omega[{\bf t}_{\omega}^{-1}])[\check{\bf t}_{\omega},\check{\bf t}_{\omega}^{-1}])$$ 
Consequently $s=V_{B_{\omega}}({\bf t}_{\omega}^{-1})\subset B_{\omega} $, by Corollary \ref{local3}(2), 
 so we can write 
\begin{align*} &\cO_{s}= \cA_\omega[{\bf t}_{\omega}^{-1}][\check{\bf t}_{\omega},\check{\bf t}_{\omega}^{-1}]/({\bf t}_{\omega}^{-1})=(\cA_{\omega}[{\bf t}_{\omega}^{-1}]/(({\bf t}_{\omega}^{-1}\cdot\cA_{\omega}[{\bf t}_{\omega}^{-1}]))[\check{\bf t}_{\omega},\check{\bf t}_{\omega}^{-1}]=\\ & =(\cA_{\omega}/(({\bf t}_{\omega}^{-1}\cdot\cA_{\omega})\cap \cA_{\omega}))[\check{\bf t}_{\omega},\check{\bf t}_{\omega}^{-1}]=  \gr_{\omega}(\cO_X)[\check{\bf t}_{\omega},\check{\bf t}_{\omega}^{-1}].
	\end{align*}	
The latter equality follows from Section \ref{grad}.	
\end{proof}

\subsubsection{The weak and the strict transforms and the ideal of the initial forms}
The identification from Lemma \ref{regul} can be extended to the strict transforms of the ideals.
The following  generalizes the  result from \cite[Lemma 5.1.4]{Wlodarczyk22}  for the weighted blow-ups.
\begin{lemma} \label{regular3} 
\label{grad2} Let $X$ be a regular scheme over a field
Let $B\to X$ be the full cobordant blow-up of a locally monomial center $\cJ$. Let $\cI\subset \cO_X$ be an ideal sheaf on $X$.
Let $\sigma^s(\cI)\subset \cO_B$ be the strict transform of $\cI$, and $\sigma^\circ(\cI)\subset \cO_B$ be its weak transform (see Definition \ref{strict}).
Then for any $s\in S_D$,
the natural isomorphism $$\cO_{s}\simeq \cO_B[\check{\bf t}_{\omega},\check{\bf t}_{\omega}^{-1}]/({\bf t}_{\omega}^{-1})\to \gr_{\omega}(\cO_X)[\check{\bf t}_{\omega},\check{\bf t}_{\omega}^{-1}]$$ takes 
\begin{enumerate}

\item $\sigma^s(\cI)_{|s}$ onto
   $\inn_{\omega}(\cI)[\check{\bf t}_{\omega},\check{\bf t}_{\omega}^{-1}]\subset \gr_{\omega}(\cO_X)[\check{\bf t}_{\omega},\check{\bf t}_{\omega}^{-1}]$. 	

\item $\sigma^\circ(\cI)_{|s}$ onto
   $\inn^\circ_{\omega}(\cI)[\check{\bf t}_{\omega},\check{\bf t}_{\omega}^{-1}]\subset \gr_{\omega}(\cO_X)[\check{\bf t}_{\omega},\check{\bf t}_{\omega}^{-1}]$. 
\end{enumerate}
\end{lemma}

\begin{proof} Let $f\in \cI$ such that  $f\in  \cJ^{\alpha}\setminus \cJ^{>{\alpha}}$.
By the definition of $$\cO_B=(\bigoplus_{\alpha\geq 0} \cJ^{\alpha}{\bf t}^{\alpha})[{\bf t}^{-1}]$$ we conclude that $\sigma^s(f)={\bf t}^{\alpha}f.$

 Then   $f\in  \cJ^{\alpha}_{\omega}\setminus \cJ_{\omega}^{>{\alpha}}$, and in a neighborhood of $s$ we have that $\check{\bf t}^{-1}_{\omega}$ is invertible.
Then the strict transform $$\sigma^s(f)={\bf t}^{\alpha}f={\bf t}_{\omega}^{\alpha}\check{\bf t}_{\omega}^{\alpha}f\in \cJ^{\alpha} {\bf t}^{\alpha}\subset \cJ_{\omega}^{\alpha} {\bf t}_{\omega}^{\alpha}[\check{\bf t}_{\omega},\check{\bf t}_{\omega}^{-1}],$$ and its reduction modulo $({\bf t}_{\omega}^{-1}\cO_B\cap \cJ_{\omega}^{{\alpha}}){\bf t}^\alpha=\cJ_{\omega}^{>{\alpha}}{\bf t}^\alpha$ can be written as the homogenous element 
$$\sigma^s(f)={\bf t}^{\alpha}f+{\bf t}^\alpha \cJ_{\omega}^{{>a}}={\bf t}_{\omega}^{\alpha}\check{\bf t}_{\omega}^{\alpha}f+{\bf t}_{\omega}^{\alpha}\check{\bf t}_{\omega}^{\alpha}\cdot\cJ_{\omega}^{{>a}}$$
 in $$\cO_B[\check{\bf t}_{\omega},\check{\bf t}_{\omega}^{-1}]/({\bf t}_{\omega}^{-1}\cdot\cO_B[\check{\bf t}_{\omega},\check{\bf t}_{\omega}^{-1}])=\gr_{\omega}(\cO_X)[\check{\bf t}_{\omega},\check{\bf t}_{\omega}^{-1}]$$ in the gradation \begin{align*} & {\cJ^{\alpha}_{\omega} {\bf t}_{\omega}^{\alpha}[\check{\bf t}_{\omega},\check{\bf t}_{\omega}^{-1}]}/{({\bf t}_{\omega}^{-1}\cO_B[\check{\bf t}_{\omega},\check{\bf t}_{\omega}^{-1}]\cap \cO_B[\check{\bf t}_{\omega},\check{\bf t}_{\omega}^{-1}])}\\&=
 (\cJ^{\alpha}_{\omega}/\cJ_{\omega}^{{>a}}){\bf t}_{\omega}^{\alpha}[\check{\bf t}_{\omega},\check{\bf t}_{\omega}^{-1}]\subset \gr_{\omega}(\cO_X)[\check{\bf t}_{\omega},\check{\bf t}_{\omega}^{-1}] \end{align*}
On the other hand $f$ determines its initial form $$\inn_\omega(f)=(f+\cJ_{\omega}^{>\alpha}){\bf t}^{\omega}\in (\cJ_{\omega}^{\alpha}/\cJ_{\omega}^{>\alpha}){\bf t}^\alpha_{\omega},$$ and thus, by the above $\sigma^s(f)$ 
 naturally and bijectively corresponds to $$\check{\bf t}_{\omega}^{\alpha}\inn_\omega(f)\in (\cJ^{\alpha}_{\omega}/\cJ_{\omega}^{>{\alpha}}) {\bf t}_{\omega}^{\alpha}[\check{\bf t}_{\omega},\check{\bf t}_{\omega}^{-1}]\subset \gr_{\omega}(\cO_X)[\check{\bf t}_{\omega},\check{\bf t}_{\omega}^{-1}].$$
	The latter differs from $\inn_\omega(f)$ by the unit $\check{\bf t}_{\omega}^{\alpha}$: $$\check{\bf t}_{\omega}^{\alpha}\inn_{\omega}(f)\sim \inn_{\omega}(f).$$
\end{proof}

\subsection{Cobordant resolution by locally monomial centers}\label{almost}

\subsubsection{Weighted normal cone}

\begin{definition} Let $X$ be a regular scheme over a field.
Let $Y\subset X$ be a  closed reduced subscheme with the ideal $\cI_Y$. Let $\omega$ be a set of monomial valuations for a partial system of local parameters. 
The subscheme   $C_{\omega}(Y)=V(\inn_{\omega}\cI_Y)\subset \NN_{\omega}(X)$ will be called  the {\it weighted normal cone} of $Y$ at $\omega$. 
	
\end{definition}

\begin{lemma}\label{special2} Let $X$ be a regular universal catenary scheme over a field. Let $Y\subset X$ be a subscheme of pure codimension $d$. Let 
$\omega$ be a set of monomial valuations for 
a common partial local system of parameters $u_1,\ldots,u_k$ on $X$.
Then $C_\omega(Y)$ is of pure codimension $d$ in $\NN_\omega(X)$.

\end{lemma}
\begin{proof} Let $\omega=\{\nu_1,\ldots,\nu_r\}$ and $\omega_1=\{\nu_1,\ldots,\nu_{r-1}\}$ be its subset. Then, by Lemma \ref{special},   we can write $$\inn_\omega(\cI)=\inn_{\nu_r}(\inn_{\omega_1}(\cI)),$$
where $\nu_r$ is monomial on  $$\NN_{\omega_1}(X)=\Spec(\gr_{\omega_1}(\cO_X))=\Spec(\cO_{V(\cJ_{\omega_1})}[u_1 ,\ldots,u_{\ell}])$$
is also universally catenary. Here we assume without loss of generality that $\cJ_{\omega_1}=u_1 ,\ldots,u_{\ell}$ for $\ell\leq k$.

By the inductive argument for $\inn_{\omega_1}(\cI)$ on $\NN_{\omega_1}(X)$ we can reduce the situation to a single monomial valuation $\nu=\nu_r$.

 Let $\nu(u_1)=w_1,\ldots,\nu(u_k)=w_k$, and find   some integers $a_1,\ldots,a_k$ such that  $$a_1w_1=\ldots=a_kw_k,$$

Consider the full cobordant	blow-up $B$ of $\cI=(u_1^{a_1},\ldots,u_k^{a_k})$ 
$$B=\Spec_X(\cO_X[{t}^{-1},u_1{t}^{w_1},\ldots,u_k{t}^{w_k}].$$
We apply the argument from \cite[Theorem 5.2.1]{Wlodarczyk22}. 
By the assumption, $B$ is catenary.
Let $d$ be the codimension of $Y$ in $X$.
Then for the morphism $$\pi_{B_-}: B_-=B\setminus V({t}^{-1})=X\times \GG_m\to X,$$ the inverse image $\pi_{B_-}^{-1}(Y)$ is  of pure codimension $d$ in $B_-$. So it is its scheme-theoretic closure $Y':=\overline{\pi_{B_-}(Y)}$, which is the strict transform $V(\sigma^s(\cI))$ of $Y$. 

Note that ${t}^{-1}$ is not  a zero divisor in $$Y'=V(\sigma^s(\cI))=\Spec(\cO_B/\sigma^s(\cI)),$$  since ${t}^{-1}f\in \sigma^s(\cI)$ implies $f\in \sigma^s(\cI)$, by definition of the strict transform.

Then, by the Krull Hauptidealsatz, we have that each component of $Y'\cap V({t}^{-1})$ is of codimension $1$ in  $Y'$, and of codimension $d+1$ in $B$.  We conclude that each component of $$Y'\cap V({t}^{-1})=V(O_{{t}^{-1}}\cdot \sigma^s(\cI))=C_\nu(Y)\subset \NN_\nu(X)$$ is of codimension $d$ in $V({t}^{-1})=\NN_\nu(X)$.

\end{proof}

\subsubsection{Cobordant resolution}
  
  For any scheme $Y$, let $\Sing(Y)$ denote its singular locus. For any ideal $\cI$ on $X$ by $\Sing(V(\cI))$ we mean the singular locus of the scheme $$V(\cI)=\Spec_X(\cO_X/\cI).$$
The following theorem extends \cite[Theorem 5.2.2]{Wlodarczyk22}.
\begin{theorem} \label{regular} Let $X$ be a regular universally catenary  scheme over a field.
Let $Y\subset X$ be an integral, closed subscheme of  pure codimension $d$ defined by $\cI_Y$. 

 
 Assume 
 there is a locally monomial ideal $\cJ\supset \cI_Y$ on $X$, with the cosupport $V(\cJ)$ of codimension $\geq 2$, and 
 with the associated exceptional divisor $E$ on the normalized blow-up $\sigma:Y=bl_{\cJ}(X)\to X$, and the dual valuation complex $\Delta_E^N$ such that 
\begin{enumerate} 
\item $\Sing(V(\cI_Y))\subseteq V(\cJ)$.

\item For any  valuation face $\omega\in \Delta^N_E\subset \Delta^N_B$, and   the ideal  $\inn_{\omega}(\cI_Y)\subset \gr_{\cJ_\omega}(\cO_X) \footnote{Definition \ref{w}}$  we have $$\Sing_{\NN_{\omega}(X)}(V(\inn_{\omega}\cI_Y))\subseteq V(\inn^\circ_{\omega}\cJ) .$$
\end{enumerate}
(respectively \\ \indent (2') For any  valuation face $\omega\in \Delta^N_E\subset \Delta^N_B$, we have $$\Sing_{\NN_{\omega}(X)}(V(\inn^\circ_{\omega}\cI_Y))\subseteq V(\inn^\circ_{\omega}\cJ).) $$

Then the cobordant blow-up $B_+\to X$ of $\cJ$ defines a cobordant resolution of $Y$. That is, the strict transform $Y'=V(\sigma^s(\cI_Y))$ of $Y$ (respectively the weak transform $Y'=V(\sigma^\circ(\cI_Y))$ of $Y$) is a regular subscheme of $B_+$ of the codimension equal to the codimension of $Y$ in $X$.
	
\end{theorem}
\begin{proof} 
The problem is local on $X$. 
Thus, up to a torus factor, we can assume that the full cobordant blow-up of $\cJ$ is given locally on $X$ by $$\sigma: B=\Spec(\cO_X[{\bf t}^{-1},{\bf t}^{\alpha_1}u_1,\ldots,{\bf t}^{\alpha_k}u_k])\to X.$$
Then for the restriction morphism $\pi_{B-}: B_-=X\times T\to X$, the inverse image $\pi_{B_-}^{-1}(Y)$ is irreducible of codimension $d$. So is its closure $Y':=\overline{\sigma^{-1}(Y)}=V(\sigma^c(\cI_Y))$, which is the strict transform of $Y$. 
Since $V(\cJ)$ is of codimension $\geq 2$, the divisor $D=V_B({\bf t}^{-1})$ is exceptional for $B\to X$.
Observe that $$\Sing(Y')\setminus D=\Sing(Y')\cap B_-\subset V_{B_-}^\circ(\cJ)\subseteq  V_B(\sigma^\circ(\cJ))$$

On the other hand, the exceptional divisor $D_+=D_{|B_+}$  is the union of the strata $s_+\in S_{D+}$. By Corollary \ref{local3}(4),(5), each such a stratum $s_+$ extends to $s\in S_D$, and corresponds to the valuation face $\omega\in \Delta^N_{D_+}=\Delta^N_E\subset \Delta^N_D$.

Since the singular locus of $V(\inn_{\omega}\cI_Y)$ is contained in $V(\inn_\omega^\circ(\cJ))$ and by Lemmas
 \ref{regul}, \ref{grad2}, we have 
 $$\Sing(Y'\cap s)=\Sing_{\NN_{\omega}(X)}(V(\inn_{\omega}\cI_Y))\times\check{\bf t}_\omega\subseteq V(\inn^\circ_{\omega}\cJ)\times\check{\bf t}_\omega= V_B(\sigma^\circ(\cJ)_{|s}),$$ 
 Then
using Lemmas \ref{special2} and \ref{regul} we conclude that  the subscheme  $$Y'\cap s\simeq V(\inn_{\omega}\cI_Y))\times\check{\bf t}_\omega $$ is of pure codimension $d$ in $s\simeq \NN_{\omega}(X)\times\check{\bf t}_\omega$, and $$(Y'\cap s)\cap B_+=((Y'\cap s)\setminus V_B(\sigma^\circ(\cJ))$$ is regular of codimension  $d$ in $s$.


 Hence for $p\in ((Y'\cap s)\setminus V_B(\sigma^\circ(\cJ))$, we can find parameters $v_1,\ldots v_d\in  \cO_{s,p} \cdot \cI_{Y'}=(\cO_{B,p}\cdot \cI_{Y'})/({\bf t}_{\omega}^{-1})$ at $p$ which  vanish on $Y'\cap s$. But these parameters come from local parameters in $\cI_{Y'}$ on $B$ at $p$. So they define a regular subscheme $Y''$ of $B_+$ of codimension $d$, containing locally $Y'$. Thus $Y''$ locally coincides with $Y'$ which must be  regular at $p\in s\setminus V_B(\sigma^\circ(\cJ))$.
 Consequently $\Sing(Y')$ is contained in  $V_B(\sigma^\circ(\cJ))$, and, by Lemma \ref{blow},  $Y'$ is a regular subscheme of $B_+=B\setminus V_B(\sigma^\circ(\cJ))$  of codimension $d$. 
 
 The proof for the weak transform $\sigma^\circ(\cI)$ (with stronger assumptions in condition (2')) is identical.

\end{proof}
As a corollary, we obtain the following:
\begin{theorem} \label{regul22} Let $X$ be a smooth variety over a field $\kappa$ of any characteristic.
Let $Y\subset X$ be a closed integral subscheme of $X$ 
Assume 
 there is a locally monomial ideal $\cJ\supset \cI_Y$ on $X$ , with the cosupport $V(\cJ)$ of codimension $\geq 2$, and
 with the associated exceptional divisor $E$ on the normalized blow-up $\sigma:Y=bl_{\cJ}(X)\to X$, and the dual valuation complex $\Delta_E^N$ such that 
\begin{enumerate} 
\item $\Sing(V(\cI_Y))\subseteq V(\cJ)$.

\item For any  valuation face $\omega\in \Delta^N_E$,   the ideal  the singular locus $$\Sing_{\NN_{\omega}(X)}(V(\inn_{\omega}\cI_Y))\subseteq V(\inn^\circ_{\omega}\cJ) .$$

\end{enumerate}

Then there is a resolution of  $Y$ at $Z$, that is, a projective birational morphism $\phi: Y^{res}\to Y$ from a smooth variety $Y^{res}$  with the exceptional locus $Z\subset Y$, such that $\phi^{-1}(Z)$ is an SNC divisor on $Y'$.

\end{theorem}
\begin{proof} Take the cobordant resolution $B_+\to X$ from Theorem \ref{regular}. We use Section \ref{local22} to embed  cobordant blow-up $B_+$ as a smooth subspace of the relative affine space $\AA^n_X$. This implies that $B_+\sslash T$ is locally toric.



The locally toric singularities of  $B_+\sslash T$
 can be canonically resolved by the combinatorial method of \cite[Theorem 7.17.1]{Wlodarczyk-functorial-toroidal}. This produces the projective birational resolution $Y'\to Y$ of $Y$ such that the inverse image of the singular point is an SNC divisor.

\end{proof}

\subsection{Resolution of hypersurfaces via the Newton method}





\subsubsection{The Newton polytope of a monomial ideal}
Let $X=\AA^k_Z=\Spec(\cO_Z[x_1,\ldots,x_k]),$
where $Z$ is a smooth scheme over $\kappa$, and $\cI\subset \cO_Z[x_1,\ldots,x_k]$ is an ideal. One can  extend the notion of the Newton polytope of monomial ideals
$I=({\bf x}^{\alpha_1},\ldots, {\bf x}^{\alpha_k})\subset \kappa[x_1,\ldots,x_n]$  
considered previously in Section \ref{Newt} in the case  $Z=\Spec(\kappa)$, where $\kappa$ is a field.

As before, by the associated {\it  Newton polytope} of $\cI$ we mean
 $${\rm P}_\cI:=\con(\alpha_1+\QQ_{\geq 0}^n,\ldots,\alpha_k+\QQ_{\geq 0}^n)\subseteq \QQ_{\geq 0}^n$$ 

Conversely, with   a polytope $P\subset \QQ_{\geq 0}^n$ we associate the monomial ideal $$\cI_P:=(x^\alpha\mid \alpha \in P\cap \ZZ^n).$$

\subsubsection{The initial forms defined by faces}
\begin{definition} (see \cite{AQ})
If  $\cI=(x^\alpha)_{\alpha\in A}\subset \cO_Z[x_1,\ldots,x_n]$ is a monomial ideal and $P_\cI$ is its Newton
polytope and $P\leq P_\cI$ be its face, we define  the {\it initial form with respect to a face of the $P_I$} to be:

$$\inv_{P}(\cI):=( {\bf x}^{\alpha} \mid {\alpha\in A\cap P}).$$
\end{definition}

The definition  $\inv_{P}(\cI)$  is a particular case of the notion o the initial form $\inv^\circ_{\omega}(\cI)$
with respect to a valuation face $\omega\in \Delta^N_E$.  

 By Corollary \ref{use} we obtain
 \begin{lemma} \label{use4} Let $\Delta^N_E$ be the valuation dual complex associated with a monomial ideal $\cI$, and let  $P=P_\cI$ be its Newton polytope.

 
 Any valuation face $\omega\in \Delta^N_E$ of the associated dual valuation complex $\Delta^N_E$ defines 
 the induced face   $$P_\omega:=P\cap \bigcap_{\nu\in \omega} H_\nu$$ of the Newton polytope $P$, for the supporting hyperplanes $H_\nu$ associated with $\nu$ and we have:
  $$\inv^\circ_{\omega}(\cI)=\inv_{P_\omega}(\cI).$$
  
  Conversely for any supporting  face ${P'}$ of $P$ there is 
 a valuation face $\omega\in \Delta^N_E$, such that $\inv^\circ_{\omega}(\cI)=\inv_{P_\omega}(\cI).$ \qed
 	
 \end{lemma}

 \begin{remark} The above correspondence is not bijective.
 Several valuation faces $\omega$ could define the same supporting face of $P$.
 The information encoded in the dual valuation complex is richer and can be applied to a more general setting.
 	
 \end{remark}
 \subsubsection{The Newton polytopes of polynomials and ideals}
 
 By the {\it Newton polytope} of the function $$f=\sum c_{\alpha} {\bf x}^\alpha\in  \cO(Z)[x_1,\ldots,x_k],$$ where $c_\alpha\in \cO(Z)$ we mean the Newton polytope of the monomial ideal $$\cJ_f:=(x^\alpha\mid c_{\alpha}\neq 0),$$ generated by the exponents $\alpha$  occurring in the presentation of $f$ with nonzero coefficients. 
 Note that $\cJ_f$ is the smallest monomial ideal which contains $f$.

This definition can be extended to any ideal $\cI\subset \cO_Z[x_1,\ldots,x_k]$. We associate with $\cI$ the monomial ideal $\cJ=\cJ_\cI$ generated by $\cI_f$, where $f\in \cI$. The Newton polytope $P_\cI$ of $\cI$ is simply the Newton polytope  of the monomial ideal $\cJ$.

If $P\leq P_\cI$ is a face of the Newton polytope $P_\cI$, and    $f=\sum_{\alpha\in A} c_{\alpha}{\bf x}^{\alpha}\in \cI$ we put $$\inv_{P}(f):=\sum_{\alpha\in A\cap P} c_{\alpha}{\bf x}^{\alpha}.$$
Then $\inv_{P}(\cI)$ is the ideal generated by $\inv_{P}(f)$, where $f\in \cI$. 

Let $\Delta^N_E$ be the dual valuation complex associated with a monomial ideal $\cJ\subset \cO_Z[x_1,\ldots,x_k]$. Recall that, by Section \ref{initial}, and using identification : \\ $\gr_{\omega}\cO_Z[x_1,\ldots,x_k]=\cO_Z[x_1,\ldots,x_k]$, for any valuation face $\omega\in \Delta^N_E$ we write $$\inv_{\omega}(f):=\sum_{\alpha\in A_{\omega,f}} c_{\alpha}{\bf x}^{\alpha}\in \gr_\omega(\cO_Z[x_1,\ldots,x_k])=\cO_Z[x_1,\ldots,x_k].$$

where $$A_{\omega,f}=\{\alpha\in A  \mid \nu(x^\alpha)=\nu(f),  \nu \in \omega\}.$$

 Similarly for the  ideal  $\cI\subset \cO_Z[x_1,\ldots,x_k]$ the ideal  of the initial forms $\inv_{\omega}^\circ(\cI)$
is generated by all $\inv_{\omega}(f)$, where $f\in \cI$, and $\nu(f)=\nu(\cI)$ for all  $\nu \in \omega$.
 
 The following is an immediate consequence of Lemma \ref{use4}, and the above:
 \begin{lemma} \label{use3} Let $P_f$ (respectively $P_\cI$) be the Newton polytope of $f=\sum c_\alpha x^\alpha$ (respectively of an ideal $\cI\subset \cO_Z[x_1,\ldots,x_k]$), and let $\cJ_f$ (respectively $\cJ_\cI$) be the associated monomial ideal. Then for any valuation face $\omega\in \Delta^N_E$ of the associated dual valuation complex $\Delta^N_E$ and the corresponding face $P_\omega$ of  $P$. 
 
  $$\inv_{\omega}(f)=\inv_{P_\omega}(f).$$
  $$\mbox{(respectively}\quad \inv_{\omega}^\circ(\cI)=\inv_{P_\omega}(\cI)).$$ \qed
 	
 \end{lemma}


\subsubsection{Resolution by  the Newton polytopes}
The following is a particular case of Theorem \ref {regular} for hypersurfaces, written in a more straightforward setup. 

\begin{theorem} \label{A} $X=\AA^n_Z=\Spec \cO_Z[x_1,\ldots,x_k]$, where $Z$ is a regular scheme over  a field $\kappa$. Let $$f=\sum_{\alpha\in A_f} c_{\alpha}{\bf x}^{\alpha}\in \cO_Z[x_1,\ldots,x_k]$$  where   $c_{\alpha}\neq 0$  for $\alpha\in A_f$. Let $\cJ=({\bf x}^\alpha\mid \alpha\in A_f)^{\sat}$ be the induced monomial ideal, and $P_f=P_\cJ$  be its Newton polytope. 
Assume that
\begin{enumerate}
\item The cosupport $V(\cJ)$ is of codimension $\geq 2$,
 \item $\Sing(V((f))\subseteq  V(\cJ).$ 
\item For any supporting  face $P$ of $P_f$, $\Sing(V(\inn_{P}(f))\subset V(\inn_{P}(\cJ))$. 
\end{enumerate}
\item 
 Then the cobordant blow-up $B_+\to X$ of $\cJ$ resolves the singularity of $V(f)$.
That is, the strict transform $Y'=V(\sigma^s(f))$ of $Y$ (which coincides with   the weak transform $V(\sigma^\circ(f))$ of $Y$) is a regular subscheme of $B_+$. 	

\end{theorem}

\begin{proof}

By Lemmas \ref{use3}, \ref{use4}, and the assumption (3) we get that	$$\Sing(\inv_\omega(f))=\Sing(\inv^\circ_\omega(f))\subset  V(\inn^\circ_\omega(\cJ)),$$ for any $\omega\in \Delta^N_E$ and the corollary follows
from Theorem \ref{regular}. \end{proof}

\begin{remark} The theorem shows that in the case of hypersurface $V(f)$ the critical combinatorial information is 
related to the faces $P$ of the Newton polytope $P_f$. Generally,  one considers  the  dual valuation complex $\Delta^N_E$ associated with the ideal $\cI$. 
In such a case, $\inv_P(f)$ is replaced with more general $\inv_\omega(\cI)$, and the role of the Newton polytope of a monomial ideal is limited (see Theorems \ref{regular}, Theorems \ref{res2}). However, it still can be used in the context of the order of the ideals in $\cO_Z[x_1,\ldots,x_k]$ (see Theorems  \ref{res3}, \ref{res4}).
\end{remark}



One can easily  extend these results to the products of schemes:
\begin{theorem} \label{B}  Let $X=\prod_Z X_j$, where each  $X_j=\AA^{n_j}_Z=\Spec(\cO_Z[{\bf x}_{j}])=\Spec \cO_Z[x_{j1},\ldots,x_{jk_j}]$, where $Z$ is a regular scheme over  a field $\kappa$ for $j=1,\ldots, r$. Let $$f_j=\sum_{\alpha\in A_{f_j}} c_{j\alpha}{\bf x}^{j\alpha}\in \cO_Z[x_{j1},\ldots,x_{jk_j}]$$  where   $c_{j\alpha}\neq 0$ for $\alpha\in A_{f_j}$. Let $$\cJ_j=({\bf x}_j^{\alpha}\mid \alpha\in A_{f_j})^{\sat}\subset \cO_Z[{\bf x}_{j}]$$ be the induced monomial ideal, and $P_{f_j}:=P_{\cJ_j}$  be its Newton polytope in $\QQ^{k_j}$. 
Assume that for any $j=1,\ldots,r$
\begin{enumerate}
\item The cosupport $V(\cJ_j)$ is of codimension $\geq 2$,
\item $\Sing(V(f_j))\subseteq  V(\cJ_j).$ 
\item For any supporting  face $P$ of $P_{f_j}$,$\Sing(V(\inn_{P}(f))\subseteq V(\inn_{P}(\cJ_i))$. 
\end{enumerate}
\item 
 Then the cobordant blow-up $B_+\to X$ of $\prod_{j=1}^r\cO_X\cdot \cJ_i$ resolves the singularity of $V(f_1,\ldots,f_k)$. That is, the strict transform $Y'=V(\sigma^s((f_1,\ldots,f_k))$ of $Y$  is a regular subscheme of $B_+$. \end{theorem}

\begin{proof}
The  cobordant blow-up of $\prod_{j=1}^r\cO_X\cdot\cJ_i$   is equal to the product over $Z$ of the cobordant blow-ups $B_{j+}$ of  $\cJ_i$ on  $\Spec(\cO_Z[{\bf x}_j])$, each of which is smooth over $Z$.

\end{proof}

\subsection{The Abramovich-Quek resolution}

The following result is  due to Abramovich-Quek (with some minor modifications):

\begin{corollary}\cite[Theorem 5.1.2]{AQ} \label{AQ21} Let $X=\AA^n_Z=\Spec \cO_Z[x_1,\ldots,x_n]$, where $Z$ is a regular scheme over a field. Consider the induced SNC divisor $D:=V(x_1\cdot\ldots\cdot x_n)$.
Let $$f=\sum_{\alpha\in A_f} c_{\alpha}{\bf x}^{\alpha}\in \cO(Z)[x_1,\ldots,x_k]$$  where   $c_{\alpha}\neq 0$  for $\alpha\in A_f$. Let $\cJ=\cJ_f:=({\bf x}^\alpha\mid \alpha\in A_f)$ be the associated monomial ideal, and $P_f:=P_\cJ$  be its Newton polytope. 
Assume that the cosupport $V(\cJ)$ is of codimension $\geq 2$, and
 for any   face $P$ of $P_f$, the ideal $(\inn_P(f))$ determines a smooth subscheme outside of $D$. 

 Then the cobordant blow-up $B_+\to X$ of $\cJ$ resolves the singularity of $V(f)$. That is, the strict transform $Y'=V(\sigma^s(f))=V(\sigma^\circ(f))$ of $Y$  is a regular subscheme of $B_+$. 

\end{corollary}
\begin{remark} Note that unlike in the original formulation the coefficients
$c_{\alpha}$ are not necessarily invertible. Theorem \ref{AQ21}	 is further generalized for the ideals in the context of order. See Remark \ref{order2}.
\end{remark}

\begin{proof} Let $\sigma_0^\vee=\QQ_{\geq 0}^n=\langle x_1,\ldots,x_n \rangle$ be the cone corresponding to the ring $$\cO_Z[x_1,\ldots,x_n].$$ 



It suffices to  show that  conditions (2), (3) of Theorem \ref{A} are satisfied.

To prove condition (3)
let $P'$ be any supporting face of $P_f$. By  Lemma \ref{strata}, there is a stratification of $X$ with  strata ${s_\tau}$, where $\tau$ is a face of $\sigma_0$ .  It is determined by the pull-back of the orbit stratification on $X_{\sigma_0}$, via $X=X_{\sigma_0}\times Z \to X_{\sigma_0}$.

Assume that $\overline{s_\tau}$ is not in $V(\inn_{P'}^\circ(\cJ))$. This means, by Corollary \ref{orbit}, that $\tau^*$ intersects $P'$ so we
consider the face $P:=P'\cap \tau^*$. 
Moreover, by Lemma \ref{strata}, we can write the closure of the stratum $s_\tau$ as $$\overline{s_\tau}:=V(x_i \mid x_i \not \in \tau^*).$$ 
On the other hand, since $P\subset \tau^*$, the polynomial $\inn_{P}(f)\in \cO_Z[x_i\in \tau^*]\subset \cO_Z[x_1,\ldots,x_n]$ can be identified with $$\inn_{P}(f)_{|\overline{s_\tau}}\in \cO_Z[x_1,\ldots,x_n]/(x_i \mid x_i \not \in \tau^*)\simeq \cO_Z[x_i\in \tau^*].$$
Now $\inn_{P}(f)$ is simply equal to
$$\inn_{P}(f)_{|\overline{s_\tau}}=\inn_{P'}(f)_{|\overline{s_\tau}},$$


By the assumption  $\inn_{P}(f)\in \cO_Z[x_i\in \tau^*]$ is a local parameter on 
\begin{align*}
&\Spec(\cO_Z[x_1,\ldots,x_n]\setminus V(\prod {x_i})\quad =\\&(\Spec(\cO_Z[x_i\in \tau^*])\setminus V(\prod_{x_i\in  \tau^*} x_i ))\quad \times \quad (\Spec(\cO_Z[x_i\not\in \tau^*])\setminus V(\prod_{x_i\not\in  \tau^*} x_i )),\end{align*}
and on 
\begin{align*} &(\Spec(\cO_Z[x_i\in \tau^*])\setminus V(\prod_{x_i\in  \tau^*} x_i ))\simeq \\ &  \simeq V(x_i \mid x_i \not \in \tau^*)\setminus V(\prod_{x_i\in  \tau^*} x_i ))=\overline{s_\tau}\setminus V(\prod_{x_i\in  \tau^*} x_i ))=s_\tau\end{align*}
 Consequently 
 $\inn_{P}(f)_{|\overline{s_\tau}}=\inn_{P'}(f)_{|\overline{s_\tau}}$
defines a local parameter on the stratum $s_\tau$.
This implies that $\inn_{P'}(f)$ is a local parameter on all strata $s_\tau$ outside of $V(\inn^\circ _{P'}(\cJ))$.

The proof of condition (2) is similar.
Consider any face $\tau$ of $\sigma_0$.
If $\overline{s_\tau}$ is not in $V(\cJ)$, then, by Corollary \ref{orbit},  $\tau$ intersects $P_f$ so we
consider the face $P:=P_f\cap \tau^*$. 
Consequently, by the assumption
$$(\inn_{P}(f))_{|\overline{s_\tau}}=f_{|\overline{s_\tau}}\in \cO_Z[x_i\in \tau^*]$$
defines a local parameter on the stratum $$s_\tau=\overline{s_\tau}\setminus V(\prod_{x_i\in  \tau^*} x_i ),$$ 
This implies that $f$ is  a local parameter on all strata $s_\tau$ outside of $V(\cJ)$, showing  condition (2) of Theorem \ref{A} and completing  the proof.

\end{proof}

\begin{corollary} \label{BB}  Let  $Z$ be  a regular scheme over  a field $\kappa$, and  $X=\prod_Z X_j$, where  $X_j:=\AA^{n_j}_Z=\Spec \cO(Z)[x_{j1},\ldots,x_{jk_j}]$ for $j=1,\ldots, r$. Let $$f_j=\sum_{\alpha\in A_{f_j}} c_{j\alpha}{\bf x}^{j\alpha}\in \cO_Z[x_{j1},\ldots,x_{jk_j}]$$  where   $c_{j\alpha}\neq 0$  for $\alpha\in A_{f_j}$. Let $\cJ_j:=({\bf x}_j^{\alpha}\mid \alpha\in A_{f_j})$ be the induced monomial ideal, and $P_{f_j}:=P_{\cJ_j}$  be its Newton polytope in $\QQ^{k_j}$. 
Assume that for any $j=1,\ldots,r$ the cosupport $V(\cJ_j)$ is of codimension $\geq 2$,
and  for any   face $P$ of $P_{f_j}$, the ideal $(\inn_P(f_j))$ determines a smooth subscheme outside of $V(x_{j1}\cdot\ldots\cdot x_{jk_j})$. 
Then the cobordant blow-up $B_+\to X$ of $\prod_{j=1}^r\cO_X\cdot\cJ_i$ resolves the singularities of $V(f_1,\ldots,f_k)$. That is, the strict transform $Y'=V(\sigma^s((f_1,\ldots,f_k))$ of $Y$  is a regular subscheme of $B_+$. 
\end{corollary}

\subsubsection{Examples of resolution}
\begin{theorem} \label{res}

Let $X=\Spec_Z (\cO_Z[x_1,\ldots,x_k])$, where $Z$ is a smooth variety over  a field $\kappa$. 
Consider the closed subscheme $Y$  on $X$ defined by a function $f\in H^0(X,\cO_X)$ of  the form 
$$f=\sum_{i=1}^k c_{\alpha_i}(v) {\bf x}^{\alpha_i},$$ 
where $c_{\alpha_i}(v)\in \cO(Z)^*$ are invertible.
 
Assume that for the presentation of $f$, one of the following holds:
\begin{itemize}
\item  $\cha(\kappa)=0$, and
for any $\alpha_i$ except possibly one, there is a
variable $x_{j_i}$ such that a power $x_{j_i}^{a_{j_i}}$ of $x_{j_i}$, occurs in ${\bf x}^{\alpha_i}$ and $x_{j_i}$ does not occur in the others ${\bf x}^{\alpha_j}$ for $j\neq i$.

\item  $\cha(\kappa)=p$, and for any $\alpha_i$ except possibly one there is a
variable $x_{j_i}$ such that a power    $x_{j_i}^{a_{j_i}}$ of $x_{j_i}$, occurs in ${\bf x}^{\alpha_i}$, with $p\nmid a_{j_i}$ and $x_{j_i}$ does not occur in the others ${\bf x}^{\alpha_j}$ for $j\neq i$ except as some $k\cdot p$-th power for $k\in \NN$.


\end{itemize}
	
	Then the cobordant blow-up $B_+\to X$ of $\cJ=({\bf x}^{\alpha_1},\ldots, {\bf x}^{\alpha_k})$ resolves singularity, so that the strict transform $\pi_B^s(f)$ determines a regular subscheme $\sigma^s(Y)$ of $B_+$.
\end{theorem}

\begin{proof}
Let $\cD(f)$ be the ideal generated by $f$, and all the derivatives  $D(f)$. At any point $p$ of $\Sing(f)$ , we have that $\ord_p(f)\geq 2$ which implies  $\Sing(V(f))=V(\cD(f))$.
 But the ideal $\cD(f)$ contains all but possibly one monomials ${\bf x}^{\alpha_i}\sim x_{j(i)}D_{x_{j(i)}}(f)$  .  Since $f=\sum_{i=1}^k c_{\alpha_i}(v) {\bf x}^{\alpha_i},$ and all but at most one monomial   
	${\bf x}^{\alpha_i}$ are in $\cD(f)$ we conclude that $\cD(f)\supseteq \cJ=({\bf x}^{\alpha_1},\ldots, {\bf x}^{\alpha_k}).$ So $$\Sing(Y)=V(\cD(f))\subseteq V(\cJ).$$ 

Similarly  $$\Sing(\inn_P(f))=V(\cD(\inn_P(f))\subseteq V((\inn^\circ_P(\cJ)).$$

Thus the conditions of Theorem \ref{A} 
are satisfied. 


\end{proof}

\begin{example}   \label{1111} Let $f=x_1^{a_1}+\ldots+x_k^{a_k}\in \kappa[x_1,\ldots,x_k]$, where the characteristic $p$ divides at most one $a_i$.
Then the cobordant blow-up of $(x_1^{a_1},\ldots,x_k^{a_k})$ resolves singularity. 
By Example \ref{weighted}, it is given by $$B=\Spec_X(\cO_X[{\bf t}^{-1},x_1{\bf t}^{w_1},\ldots,x_k{\bf t}^w_k])$$
 $$B_+=B\setminus V_B(\sigma^s(\cJ))= B\setminus V_B(x_1{\bf t}^{w_1},\ldots,x_k{\bf t}^{w_k}).$$

The morphism $B_+\to X$ is interpreted  in Section \ref{weighted2} as  the cobordant blow-up of 
 the weighted center $\cJ=(x_1^{1/w_1},\ldots,x_k^{1/w_k})$, such that $\cO_{B_+}\cdot \cJ=\cO_{B_+}\cdot {\bf t}^{-1}$.
\end{example}

\begin{example} $$x_1^p+ax_2^px_3+bx_1x_4^px_5^{p^2}\in \kappa[x_1,x_2,x_3,x_4,x_5,], $$  where $a,b\in \kappa^*$can be resolved by the single cobordant blow-up of $$\cJ=(x_1^p,x_2^px_3,x_1x_4^px_5^{p^2})$$ over a field $\kappa$ of   characteristic $p$.  Here for $x_2^px_3$ the variable $x_3$ does not occur in the other terms, and for $x_1x_4^px_5^{p^2}$ the coordinate $x_1$ occurs in the other terms as $x_1^p$ -power or does not show at all.

\end{example}

\begin{example}
$$x_1^2x_2^5+7x_4^7x_3^5+25 x_1x_3^6  \in \kappa[x_1,x_2,x_3,x_4] $$  
can be resolved by the cobordant blow-up of  $$\cJ=(x_1^2x_2^5, x_4^7x_3^5, x_1x_3^6)$$
	over a field $\kappa$ of   ${\cha(\kappa)}\neq 5,7$. We use $x_2$ for 
$x_1^2x_2^5$, and $x_4$ for $x_4^7x_3^5$.
\end{example}

\begin{theorem} \label{res2} Let $Z$ be a smooth variety over a field $\kappa$.
Let  $$X=\Spec_Z(\cO_Z[x_1,\ldots,x_n])=\Spec_Z(\cO_Z[{\bf  x}_1,\ldots,{\bf  x}_r]),$$ where 
$${\bf  x}_i:=(x_{k_{i-1}},\ldots,x_{k_{i}-1}),$$ for $k_0=1<k_1<\ldots <k_r=n+1$.
Consider the closed subscheme $Y$ of $X$   defined by  the set of the polynomial functions $f_j\in H^0(X,\cO_X)$, where $j=1,\ldots, r$ of  the form 
$$f_j=\sum_{i=1}^{r_j} c_{\alpha_{ij}}(v) {\bf  x}_j^{\alpha_{ij}},$$ 
where $c_{\alpha_{ij}}(v)\in \cO(Z)^*$ are invertible.
 
Assume that for any $j=1,\ldots, r$ and for the presentation of $f_j$ 
  one of the following holds:
\begin{itemize}
\item  $\cha(\kappa)=0$, and
for any $\alpha_{ij}$ except possibly one, there is a
variable $x_{j_i}$ such that a power $x_{j_i}^{a_{j_i}}$ of $x_{j_i}$, occurs in ${\bf x}^{\alpha_{ij}}$ and $x_{j_i}$ does not occur in the others ${\bf x}^{\alpha_{i'j}}$ for $i'\neq i$.
\item  $\cha(\kappa)=p$, and for any $\alpha_{ij}$ except possibly one there is a
variable $x_{j_i}$ such that a power $x_{j_i}^{a_{j_i}}$ of $x_{j_i}$ occurs in ${\bf x}^{\alpha_{ij}}$ , with $p\nmid a_{j_i}$  and $x_{j_i}$ does not occur in the others ${\bf x}^{\alpha_{i'j}}$ for $i'\neq i$  except as some $k\cdot p$-th power for $k\in \NN$.
\end{itemize}
	
	Then the cobordant blow-up $B_+\to X$ of $$\cJ=\prod_{j=1}^r({\bf x}_j^{\alpha_1},\ldots, {\bf x}_j^{\alpha_k})$$ resolves singularity, so that the strict transform $\sigma^s(f_1,\ldots,f_r)$ determines a smooth subvariety of $B_+$.
\end{theorem}

\begin{proof}  The space $X$ can be written as the fiber product $$X=\prod_Z\AA_Z^{k_i-k_{i-1}}=\prod_Z \Spec_Z(\cO_Z[{\bf x}_j]).$$
The  cobordant blow-up of $\cJ$  is equal to the product over $Z$ of the cobordant blow-ups $B_{j+}$ of  $({\bf x}_j^{\alpha_1},\ldots, {\bf x}_j^{\alpha_k})$ on  $\Spec_Z(\cO_Z[{\bf x}_j])$, and   each of $B_{j+}$ is smooth over $Z$ by Theorem \ref{res}.

\end{proof}

\begin{example}
The system of equations
 \begin{align*} 
&x_1^p+ax_1x_2^px_3+bx_4x_5^px_6^{p^2}&=0\\
&y_1^{p^3}+cy_2^{p^2}y_3y_6+dy_1y_4^py_5^{p^2}y_6^2&=0
\end{align*}

in $$ \kappa[x_1,\ldots,x_6,y_1,\ldots,y_5],$$  where $a,b,c,d\in \kappa^*$, can be resolved by the single cobordant blow-up of $$\cJ=(x_1^p,x_2^px_3,x_1x_4x_5^px_6^{p^2})\cdot (y_1^{p^3},y_2^{p^2}y_3,y_4y_5^py_6^{p^2})$$ in characteristic $p$.

\end{example}

\begin{example}   \label{11} Let $f_j=x_{1j}^{a_1}+\ldots+x_{k_jj}^{a_k}\in \kappa[x_{ij}]$, where $j=1,\ldots,k$ and the characteristic $p$ divides at most one $a_{ij}$ for any $j$.
Then the cobordant blow-up of $\prod_j(x_{1j}^{a_1},\ldots,x_{k_jj}^{a_k})$ resolves singularity of $V(f_j)_{j=1,\ldots,k}$. 

\end{example}

\subsection{Partial resolution by the order}
The method can be linked to different invariants, particularly  to the order $$\ord_p(\cI):=\max\{k\mid \cI_p\subset \m_p^k\},$$
where $m_p\subset \cO_{X,p}$ is the maximal ideal of a  point $p\in X$.
\begin{definition} Let $\cI$ be an ideal on a regular scheme $X$, and $d\in \NN$ be an integer. We define
$$\supp(\cI,d):=\{p\in X\mid \ord_p(\cI)\geq d\}.$$
\end{definition}

The following theorem  extends \cite[Lemma 5.3.1]{Wlodarczyk22}:

\begin{theorem} \label{res22} Let $\cI$ be an ideal on a regular scheme $X$ over a field, and let  $d\in \NN$ be any natural number. Assume that there exists a locally monomial center ${\cJ}$, with $\codim(V(\cJ)\geq 2$,  with the associated dual valuation complex $\Delta^N_E$, and such that \begin{enumerate}
\item  $\supp(\cI,d)\subseteq V(\cJ)\subset X,\quad $ 
 \item 
 $\supp(\inn_{\omega}(\cI),d)\subseteq V(\inn^\circ_{\omega}(\cJ))\subset \NN_{\omega}(X)$, for any $\omega\in \Delta^N_E$. 
 \end{enumerate}
 (respectively \\ \indent (2') $\supp(\inn^\circ_{\omega}(\cI),d)\subseteq V(\inn^\circ_{\omega}(\cJ))\subset \NN_{\omega}(X)$, for any $\omega\in \Delta^N_E$.) 

\bigskip
 Then  for the cobordant blow-up $\sigma_+: B_+\to X$ of $\cJ$, the maximal order of the strict transform $\sigma^s(\cI)$ (respectively the weak transform $\sigma^\circ(\cI)$) on $B_+$ is strictly smaller than $d$.

\end{theorem}
\begin{proof} Let $q\in D_+=D\setminus V_B(\sigma^\circ(\cJ))$, where $D=V(t_1^{-1}\cdot\ldots t_k^{-1})$ is the exceptional divisor of $B\to X$. Then there is $\omega\in \Delta^N_E$, and the corresponding stratum $s$ in $S_D$ such that $$q\in s \setminus V(\sigma^\circ(\cJ)).$$  By  Lemmas \ref{regul},  \ref{regular3}, there is a  natural isomorphism $s\to\Spec(\gr_{\omega}(\cO)[\check{\bf t}_{\omega},\check{\bf t}_{\omega}^{-1}])$, which takes
$\sigma^s(\cI)_{|s}$ to $\inn_{\omega}(\cI)[\check{\bf t}_{\omega},\check{\bf t}_{\omega}^{-1}]$ (and $\sigma^\circ(\cJ)_{|s}$ to $\inn^\circ_{\omega}(\cJ)[\check{\bf t}_{\omega},\check{\bf t}_{\omega}^{-1}]$). 
Consequently  $$\ord_q(\sigma^s(\cI))\leq \ord_q(\sigma^s(\cI)_{|s})=\ord_q(\inn_{\omega}(\cI)[\check{\bf t}_{\omega},\check{\bf t}_{\omega}^{-1}])<d.$$  
If 
	$$q\in B\setminus D \setminus V(\sigma^s(\cJ))=B_-\setminus V(\sigma^s(\cJ))=(X\setminus V(\cJ))\times  T,$$ then since $\pi_B(q)\in X\setminus V(\cJ)$ we conclude that $$\ord_q(\sigma^s(\cI))=\ord_{\pi_B(q)}(\cI)<d.$$
	The proof for $\sigma^\circ(\cI)$ is the same.
\end{proof}
	
\subsection{The Newton method of decreasing order}\label{order}

As a corollary from Theorem \ref{res22}, and Lemma \ref{use4} we obtain:
\begin{theorem} \label{res3} $X=\AA^n_Z=\Spec \cO_X[x_1,\ldots,x_k]$, where $Z$ is regular over a field $\kappa$ of characteristic $p$. 
 Let $\cI\subset  \cO_Z[x_1,\ldots,x_k]$ be an ideal, and $\cJ=\cJ_\cI$, be its associated monomial ideal with the Newton polytope $P_\cJ=P_\cI$ and $d\in \NN$ be any natural number 
 	such that 
	\begin{enumerate}
	\item $\codim(V(\cJ))\geq 2$.
	\item $\supp(\cI,d)  \subseteq V(\cJ)$.

	\item for any supporting face $P$ of $P_\cJ$,   $\supp(\inn_P(\cI),d) \subset V(\inv_P(\cJ)).$
	\end{enumerate}
	Then  the maximal order of the weak transform $\sigma^\circ(\cI)$ on $B_+$ under cobordant blow-up of $\cJ$ is strictly smaller than $d$.  \qed
\end{theorem}

Thus we get
 \begin{theorem} \label{res4} $X=\AA^n_Z=\Spec \cO_X[x_1,\ldots,x_k]$, where $Z$ is regular over a field $\kappa$ of characteristic $p$. 
 Let $\cI\subset  \cO_Z[x_1,\ldots,x_k]$ be an ideal, and $\cJ=\cJ_\cI$, be its associated monomial ideal with with $\codim(V(\cJ)\geq 2$, and let  $P_\cJ=P_\cI$  
 	be its Newton polytope and $d\in \NN$ be any natural number
 	such that 
for any face $P$  of $P_\cJ$,    $$\supp(\inn_P(\cI),d) \subset D:=V(x_1\cdot,\ldots\cdot x_k).$$


	Then  the maximal order of the weak transform $\sigma^\circ(\cI)$ under cobordant blow-up $B_+\to X$ of $\cJ$ is strictly smaller than $d$.  \qed
\end{theorem}

\begin{proof} The proof  uses similar arguments as the proof of Corollary \ref{AQ21}. We need to show that the conditions of Theorem \ref{res3} are satisfied
 
 For condition (3) of Theorem \ref{res3}, let $P'$ be any supporting face of $P_\cI$. Consider the closure of the stratum $\overline{s_\tau}$, where $\tau$ is a face of $\sigma_0$. 

If ${s_\tau}$ is not in $V(\inn_{P'}^\circ(\cJ))$
consider the face $P:=P'\cap \tau^*$. Then  $$\supp(\inn_{P'}(\cI))_{|\overline{s_\tau}},d)=\supp(\inn_{P}(f)_{|\overline{s_\tau}},d)$$

is contained  in $\overline{s_\tau}\setminus s_\tau=V(\prod_{x_i\in \tau^*} x_i )$ so it is not in $s_\tau$. 
This implies that  $\supp(\inn_{P'}(\cI)),d)$ is
contained  $V(\inn^\circ _{P'}(\cJ)$.

The proof of condition (2) of Theorem \ref{res3} is the same, except  we replace $P'$ with $P_\cJ$,  $\inn_{P'}^\circ(\cJ)$ with $\cJ=\inn_{P_\cJ}^\circ(\cJ)$ , and $P$ with $P=P_\cJ\cap \tau^*$.

Thus the corollary is  a consequence of Theorem \ref{res3}.
\end{proof}

\begin{remark} \label{order2}
Theorems \ref{res3}, \ref{res4} generalize respectively Theorem \ref{A}, and  Corollary \ref{AQ21}.
We put $\cI=(f)$, and $d=2$. Then $\Sing(V(f))=\supp((f),2)$.
	
\end{remark}

\begin{example} (See also Example \ref{111})

Let $Y\subset X=\Spec \kappa[x,y,z]$ be described by the ideal $$\cI=(x^k+xy+y^{l},\quad  z^{kl}+x^{k-2}z^{{kl}-1}+y^{k-2}z^{{kl}-1})$$ of order $2$, where $\gcd(k,l)=1$. Consider the corresponding  admissible monomial ideal $$\cJ=(x^k,xy,y^{l},z^{kl},x^{k-2}z^{{kl}-1},y^{k-2}z^{{kl}-1})$$
and  its associated  {\it Newton polytope}  $P$ generated by the exponents $$(k,0,0),(1,1,0),(0,l,0),(0,0,kl)\subset \sigma^\vee=\langle e_1^*,e_2^*,e_3^*\rangle.$$ 

This corresponds to two supporting faces $P_1,P_2$ defined, respectively,  by $$(k,0,0),(1,1,0),(0,0,kl),$$ and $$(1,1,0),(0,l,0),(0,0,kl).$$ 
They intersect at the face $P_{12}=(1,1,0),(0,0,kl)$.

The faces  $P_1,P_2$ corresponds to  the primitive vectors $v_1=(a_1,a_2,a_3)$ such that $$a_1k=a_1+a_2=a_3kl,$$ and 
$v_2=(b_1,b_2,b_3)$, where $$b_1+b_2=lb_2=klb_3$$ in the dual plane.  So $$v_1=(l,l(k-1),1),\quad v_2=(k(l-1),k,1).$$ This defines the set of two extremal valuations $\nu_1,\nu_2$.

Then $$\inn_{P_1}(\cI)=(x^k+xy,z^{kl}),\quad \inn_{P_2}(f)=(xy+y^l,z^{kl}),\quad \inn_{P_{12}}(f)=(xy,z^{kl}).$$ By considering the ideals of the derivatives $\cD(\inn_{P}(f))$ we  see that in all cases $$\supp(\inn_{P}(\cI),2)=V(x,y,z)$$
Similarly $\supp(\cI,2)=V(x,y,z)$.

The cobordant blow-up of $\cJ=(x^k,x^2y,y^l)$ is described as
\begin{align*}&B=\Spec_X(\cO_X[t_1^{-1},t_2^{-1},\,\,\,xt_1^lt_2^{k(l-1)},\,\,\,yt_1^{l(k-1)}t_2^k,\,zt_1t_2]=\\&\Spec(\kappa[t_1^{-1},t_2^{-1},\, xt_1^lt_2^{k(l-1)},\,\,yt_1^{l(k-1)}t_2^k,\,\,zt_1t_2])\end{align*}$$B_+=B\setminus V(\sigma^s(\cJ))=B\setminus {\bf t}^\alpha\cJ,$$
where $\sigma^s(\cJ)={\bf t}^\alpha\cJ$, and the coefficients are given by the exceptional divisor $E=\alpha_1E_1+\alpha_2E_2$ of the toric normalized blow-up of $\cJ$. 
  $$\alpha_1=\nu_1(f)=a_1k=a_1+a_2=a_3kl=kl,$$ and 
$$\alpha_2=\nu_2(f)=b_1+b_2=lb_2=klb_3=kl$$

Thus $$B_+= B\setminus V((x^k,xy,y^{l},z^{kl})\cdot t_1^{kl}t_2^{kl}),$$
By Lemma \ref{res2}, the cobordant blow-up $B_+\to X$ of $\cJ=(x^k,xy,y^{l},z^{kl})$ decreases the order of $\cI$ to 1. \end{example}




\section{Generalized cobordant blow-ups and $\QQ$-ideals}
\subsection{Cobordization with respect to  subgroups  $\Gamma\subset \Cl(Y/X)\otimes \QQ$}
\begin{definition} Let $\pi: Y\to X$ be a proper birational morphism. Let $ \Gamma \subset \Cl(Y/X)\otimes \QQ$ be  a finitely generated subgroup. We define the {\it full cobordization (resp. cobordization of $\pi$) with respect to $\Gamma$} to be  $$B=B^\Gamma:=\Spec_X(\pi_*(\bigoplus_{E\in \Gamma} \cO_X(E))\quad B_+=B^\Gamma_+=\Spec_Y(\bigoplus_{E\in \Gamma} \cO_Y(E)) .$$
	
\end{definition}

\begin{proposition} The natural morphism $B_+^\Gamma\to B^\Gamma$ is an open immersion if locally on $X$ there are   forms $F=fx^{-E}$, with $E\in \Gamma$  such that $X_{F}$ are open  affine and cover $X$.

\end{proposition}
\begin{proof} The proposition follows from the first part of the proof of Proposition \ref{cover2}.
	
\end{proof}

\begin{definition} Let $\pi: Y\to X$ be the normalized blow-up of the an $\cI$ on a normal scheme $X$. Let $\Gamma\subset \Cl(Y/X)\otimes \QQ$ be a finitely generated subgroup. Then we define the {\it full cobordant blow-up of $\cI$ with respect to $\Gamma$} (resp. the {\it  cobordant blow-up of $\cI$ with respect to $\Gamma$})  to be  the
{ full cobordization (resp. cobordization of) $\pi$  with respect to $\Gamma$}.


\end{definition}
\begin{proposition} \label{small} 
$\pi: Y\to X$ be the normalized blow-up of an ideal $\cJ$ on a  normal scheme $X$, and let $E^0$ be the exceptional Cartier divisor such that
$\cO_Y(-E^0)=\cO_Y\cdot \cJ$.
If  a finitely generated group $\Gamma\subset \Cl(Y/X)\otimes \QQ$ contains  divisor $E^0$, then  $B^\Gamma_+=B^\Gamma\setminus V(\cI {\bf t}^{-E_0})$.

\end{proposition}

\begin{proof} The proof is identical to the proof of Lemma \ref{blow}.

\end{proof}

\subsection{Simple cobordant blow-up of ideal $\cI$}\label{blow3}
\begin{definition}  Let $\pi: Y\to X$ be the normalized blow-up  of an ideal $\cI$ on a normal scheme $X$, with the exceptional divisor $E^0$, such that $\cO_Y(-E^0)=\cO_Y\cdot\cI$.
By the {\it simple  cobordant blow-up}
of  $\cI$ on $X$ 
we mean  the cobordization $B_+^\Gamma$ of $\pi: Y\to X$ with respect to the subgroup $\Gamma= \ZZ\cdot E^0 \subset \Cl(Y/X)$ generated by $E^0$.
	
\end{definition}
\begin{lemma} The simple cobordant blow-up of $\cI$ is given by $$B=\Spec_X(\cO_X[{\bf t}^{-1},\cI t])^{\inte},\quad  B_+=B\setminus V(\cI t).$$
	
\end{lemma}
\begin{proof}
It follows that $\cO_Y(-nE^0)=\cO_Y\cdot\cI^n$. Moreover $\pi_*(\cO_Y(-nE^0))=(\cI^n)^{\inte}$ is the integral closure of $\cI^n$.

 Consequently $$B=B^\Gamma_\cI=\Spec_X(\pi_*(\bigoplus_{n\in\ZZ} \cO_Y(nE^0){\bf t}^{nE_0})=\Spec_X(\cO_X[{ t}^{-1},\cI t])^{\inte},$$ under the identification of ${\bf t}^{E_0}$ with $t^{-1}$.
By Proposition \ref{small},
$$B_+=B\setminus V(\sigma^s(\cI))=B\setminus V(\cI t).$$
and thus is described by the standard Rees extended algebra.
\end{proof}

\subsection{Cobordant blow-ups of $\QQ$-ideals} 

\subsubsection{Valuative $\QQ$-ideals}
The {\it valuative $\QQ$-ideals}  were introduced in \cite{ATW-weighted}.  Here we consider its particular version considered in \cite{Wlodarczyk22}.
\begin{definition} By {\it   valuative $\QQ$-ideals}, or, simply, {\it  $\QQ$-ideals} on a normal scheme $X$ we mean the equivalence classes of formal expressions $\cI^{1/n}$, where $\cI$ is the ideal on $X$, and $n\in\NN$. We say that two $\QQ$-ideals  $\cI^{1/n}$,  and $\cJ^{1/m}$ are equivalent if   the integral closures of $\cI^m$, and $\cJ^n$ are the same. 
\end{definition}
In particular, if $D$ is a Cartier effective divisor on $X$ then any $\QQ$-Cartier effective divisor $\frac{1}{m}\cdot D$ determines  the $\QQ$-ideal $\cO_X(-D)^{\frac{1}{m}}$.  


By the {\it vanishing locus} of $\cJ=\cI^{1/n}$ we mean $V(\cJ)=V(\cI)$.

One can define the operation of addition and multiplication on $\QQ$-ideals: $$\cI^{1/n}+\cJ^{1/m}:=(\cI^m+\cJ^n)^{1/mn},\quad \cI^{1/n}\cdot\cJ^{1/m}=(\cI^m\cdot\cJ^n)^{1/mn}.$$

For any valuative $\QQ$-ideal $\cJ=\cI^{1/n}$ on $X$ we define the associated ideal of sections on $X$: $$\cJ_X:=\{f\in \cO_X \mid f^{n}\in \cI^{\inte}\},$$
where $\cI^{\inte}$ is the integral closure of $\cI$.
In particular, for the effective Cartier divisor $D$, we have the equalities $$(\cO_X(-D)^{1/m})_X=\cO_X(-\frac{1}{m}D)=\{f\in \cO_X\mid \divv(f)-\frac{1}{m}D\geq 0\}.$$

With any  valuative  $\QQ$-ideal $\cJ$  we  associate the {\it Rees algebra} on $X$ :
$$\cO_X[\cJ t]_X=\bigoplus_{n\in \ZZ_{\geq 0}} (\cJ^n)_X{ t}^n \subset \cO_X[t],$$
and the {\it extended Rees algebra} on $X$:$$\cO_X[{ t}^{-1},\cJ t]_X=\bigoplus_{n\in \ZZ_{\geq 0}} \cJ^n_X { t}^n \oplus \bigoplus_{-n\in \ZZ_{<0}}{ t}^{-n} \subset \cO_X[t,{ t}^{-1}].$$

\subsubsection{Cobordant blow-up of $\QQ$-ideals}
 \label{Q}
Let  $\cJ=\cI^{1/m}$ be  a $\QQ$-ideal on $X$. Consider the normalized blow-up $\pi: Y\to X$ of $\cI$, with the exceptional divisor $E^0$ such that $\cO_Y(-E^0)=\cO_Y\cdot \cI$. 

  Then $\cO_Y\cdot \cI^{1/m}$ is the $\QQ$-ideal $\cO_Y(-E^0)^{1/m}$, which corresponds to the $\QQ$-Cartier exceptional divisor $\frac{1}{m}E^0$ on $Y$.

Consequently, by the {\it blow-up} of the $\QQ$-ideal $\cJ=\cI^{1/m}$ we mean the the normalized blow-up $\pi: Y\to X$ of $\cI$, with the associated  $\QQ$-Cartier divisor $\frac{1}{m}E^0$.

\begin{definition} By the {\it simple cobordant blow-up/full cobordant blow-up of the $\QQ$-ideal $\cJ=\cI^{1/m}$} we mean the cobordization/full cobordization of 
	the normalized blow-up $Y\to X$ of $\cJ$ with respect to the  group $\Gamma=\ZZ\cdot \frac{1}{m}E^0\subset \Cl(Y/X)\otimes \QQ$ generated by $\frac{1}{m}E^0$. 
\end{definition}

\begin{lemma} Let $\sigma: B\to X$ be the simple full cobordant blow-up of the $\QQ$-ideal $\cJ=\cI^{1/m}$ on a normal scheme $X$.   Then
$$B=\Spec_X(\cO_X[{t}^{-1},\cJ t])_X$$
\begin{enumerate}
\item $B_+=B\setminus V(\cJ t)$
\item $\cO_{B_+}\cdot \cJ={t}^{-1}\cdot \cO_{B_+}$	
\end{enumerate}

 \end{lemma}

\begin{proof} Let $\pi: Y\to X$ be the normalized blow-up of $\cI$,  $E^0$ is  the exceptional divisor on $Y$ such  
 $\cO_X\cdot\cI=\cO_Y(-E_0)$. Thus, by \cite[Proof of Lemma 2.1.4]{Wlodarczyk22},  $$\pi_*(\cO_Y(-\frac{n}{m}E^0)=( f\in \pi_*(\cO_Y)=\cO_X :  f^n\in \pi_*(\cO_Y(-mE_0))=(\cI^m)^\inte)=\cJ^m_X .$$ giving the formula for $B$: 
$$B=\Spec_X(\pi_*(\bigoplus_{n\in \ZZ} \cO_Y(n\cdot (1/m)\cdot E^0){t}^n))=\Spec_X(\cO_X[{t}^{-1},\cJ t])_X.$$

By Proposition \ref{small}, $$B_+=B\setminus V(\cI\cdot {\bf t}^{-E^0})=B\setminus V(\cJ {\bf t}^{-(1/m)E^0})=B\setminus V(\cJ t),$$ 
as $(1/m)E^0$ generates $\Gamma$, and ${\bf t}^{(1/m)E^0}$ corresponds to ${t^{-1}}$.
 Thus the inverse image of $$\cO_{B_+}\cdot \cI^{1/m}=\cO_{B_+}\cdot \cJ={ t}^{-1}\cO_{B_+}\cdot \cJ t=\cO_{B_+}\cdot { t}^{-1}$$ is a Cartier exceptional divisor. We use here the fact that the $\QQ$-ideal $\cJ t_{|B_+}=\cO_{B_+}$, as $\cJ t=(\cI t^a)^{1/a}=(\cI\cdot t^{-E^0})^{1/a}$ is trivial on $B_+=B\setminus V(\cI t^{-E^0})$.
\end{proof}





\subsection{Weighted cobordant blow-ups revisited} \label{weighted2}

Let $\pi_B: B\to X$ be the simple  cobordant blow-up of the weighted center $\cJ=(u_1^{1/w_1},\ldots,u_k^{1/w_k})$, where $u_1,\ldots,u_k$ is a partial system of local parameters on a regular scheme. Assume, first,  that the weights $w_i$ are relatively prime.
The center $\cJ$ can be written as $\cJ=\cI^{1/m}$, where $$\cI=(u_1^{m/w_1},\ldots,u_k^{m/w_k}),$$ is the ideal, and the weights $w_i| m$. 

Let $E^0$ be the exceptional divisor of the blow-up $\pi: Y\to X$ of  $\cJ$. Let $\nu_{E^0}$ be the associated exceptional valuation. Using  the toric chart, defined by $u_1,\ldots,u_k$ one reduces the situation to the blow-up
of the toric $\QQ$-ideal $\cJ$ on a toric variety $X_\sigma$, where
$\sigma=\langle e_1,\ldots,e_k\rangle$ is regular.  The $\QQ$-ideal
 $\cJ=\cI^{1/m}$ defines a piecewise linear convex function $$F_\cJ=\min (\frac{1}{w_1}e_1^*,\ldots,\frac{1}{w_k}e_k^*)=1/m\cdot \min (\frac{m}{w_1}e_1^*,\ldots,\frac{m}{w_k}e_k^*).$$  The normalized blow-up of $\cJ$ defines a decomposition $\Delta$ of $\sigma$ into the maximal subcones where $F_\cJ$ is linear. 
 Let $w:=(w_1,\ldots,w_k)$ , and  $$F_\cJ(e_i)=0,\ldots, F_\cJ(w)=1.$$
 Then $\Delta$ is the star subdivision at $\langle w\rangle$. Moreover, the vector $m w$ corresponds to $mF_\cJ=F_\cI$ in the sense that they define the same Weil divisors
 , and $w$ and corresponds to $E^0$, so the valuation $\nu_{E^0}$ on $Y$ is associated with $w$. In particular, $$\nu_{E^0}(u_i)=e_i^*(w)=w_i.$$
 

Consequently, $\cO_Y\cdot(u_1^{m/w_1},\ldots,u_k^{m/w_k})=\cO_Y(-mE^0)$ is the ideal of the Cartier divisor $mE^0$ on $Y$, associated with the integral function $mF_\cJ$, and the $\QQ$-ideal $\cJ=(u_1^{1/w_1},\ldots,u_k^{1/w_k})$ corresponds  to the $\QQ$-divisor $(1/m)\cdot mE^0=E^0$ which is a Weil divisor. 

 
The  cobordant blow-up associated with the group $\Gamma=\ZZ \cdot E^0=\Cl(Y/X)$  is given by the standard formula from Theorem \ref{regular33}: 

\begin{align*} &B=\Spec_X(\pi_*(\bigoplus_{n\in \ZZ} \cO_Y(nE^0)){ t}^n=\Spec_X(\cO_X[{t}^{-1},\cJ t])_X=\\ &=\Spec_X(\bigoplus_{a_i\in \ZZ} \,\,  \cI_{\nu,a_i}\,\,\cdot t_1^{a_1}\cdot\ldots\cdot t_k^{a_k})= \\&= \Spec_X({\cO_X}[{ t}^{-1}, {u_1}{ t}^{w_1},\ldots, {u_k}{ t}^{w_k}]),\end{align*}
where $w_i=\nu_{E^0}(u_i)$.
	



In general,   for arbitrary weights, the  simple cobordant blow-up of $(u_1^{1/w_1},\ldots,u_k^{1/w_k})$ is associated with the group $\Gamma=\ZZ \cdot \frac{1}{w^0} E^0=\frac{1}{w^0}\cdot \Cl(Y/X)$, where $w^0:=\gcd(w_1,\ldots,w_k)$, and with the valuation $w^0\nu$, with $$w^0\nu(x_i)=w_i.$$
 Now  $$\cI_{w^0\nu,a}=(u_1^{b_1}\cdot\ldots \cdot u_k^{b_k}) \mid \sum_{j=1}^k {b_{j}}w_{i}\geq a).$$ Comparing gradations we  see $$\bigoplus_{a\in \ZZ}  \cI_{w^0\nu,a}t^{a}=\cO_X[t^{-1},u_it^{w_{i}}].$$
 Then
 \begin{align*}&B=\Spec_X(\pi_*(\bigoplus_{n\in \ZZ} \cO_Y(n\cdot (1/w^0)\cdot E^0)t^n=\\
 & =\Spec_X(\bigoplus_{a_i\in \ZZ}  \cI_{w_0\nu,a}\,\,\cdot t^{a})=
 \Spec_X({\cO_X}[t^{-1}, t^{ w_1}{x_1},\ldots, t^{w_k}{x_k}]). \end{align*} 

By the above $B_+=B\setminus V(\cJ t)=B\setminus V(t^{w_1}{x_1},\ldots, t^{w_k}{x_k})$, and $\cO_{B_+}\cdot\cJ=\cO_{B_+}\cdot t^{-1}$.

These weighted cobordant blow-ups were studied in \cite{Wlodarczyk22} and used for the resolution  of varieties in characteristic zero and some classes of singularities in positive and mixed characteristic. 
 To a great extent, they are equivalent to  the stack-theoretic weighted blow-ups  introduced and considered  in \cite{Marzo-McQuillan}, and \cite{ATW-weighted}.

\subsection{Multiple weighted cobordant blow-ups of Abramovich-Quek}\label{AQ}

In the paper \cite{AQ}, the authors consider the generalization of the weighted blow-ups, so-called,  {\it multi-weighted blow-ups} ${Bl_{\cJ,b}}$, associated with a $\QQ$-ideal $\cJ$ and a vector $b=(b_1,\ldots,b_k)$. They are constructed locally in toric charts in the language of  fantastacks and stack-theoretic quotients via Satriano combinatorial approach \cite{Satriano}. The {\it multi-weighted blow-ups} are used  to prove the logarithmic resolution
on smooth toroidal ambient Artin stacks in  characteristic zero. 

We give here a geometric interpretation of this construction 
in the language of cobordizations with respect to a subgroup. In particular, this approach  does not rely on coordinates or combinatorics.

Let $\pi: Y\to X$ be
the normalized blow-up  of a  locally monomial center $\cJ$ on a regular scheme over a field. Denote by $E_1,\ldots, E_k$ the irreducible exceptional divisors. Let $\nu_1,\ldots,\nu_k$  be the associated exceptional valuations.  
We consider the full cobordant blow-up  of $\cJ$ with respect to the subgroup $$\Gamma_b:=\ZZ \frac{1}{b_1}E_1\oplus\ldots\oplus\ZZ \frac{1}{b_k}E_k\subset \Cl(Y/X)\otimes \QQ,$$ for any positive integers $b_1,\ldots,b_k$, and $b=(b_1,\ldots,b_k)$. Write $$B=\Spec_X(\pi_*(\bigoplus_{E\in \Gamma_b} \cO_X(E)),\quad B_+=\Spec_Y(\bigoplus_{E\in \Gamma_b} \cO_Y(E)) .$$

 The generators $\frac{1}{b_1}E_1,\ldots,\frac{1}{b_k}E_k $ are associated with the monomial valuations $$\nu^b_1:=b_1\nu_1,\ldots,\nu^b_k:=b_k\nu_k.$$ 
 Then  locally on $X$ using the Proposition \ref{valu}, and the proof of Lemma \ref{cover4}(1)
 we can write \begin{align*}& B=\Spec_X(\bigoplus_{a_i\in \ZZ} \,\, \bigcap^k_{i=1}\, \cI_{\nu^b_i,a_i}\,\,\cdot t_1^{a_1}\cdot\ldots\cdot t_k^{a_k})=\\ &= \bigcap^k_{i=1} \cO_X[t_i^{-1},u_jt_i^{\nu^b_i(u_j)}][\check{\bf t}_i, \check{\bf t}_i^{-1}]= \\ &=
 \Spec_X(\cO_X[t_1^{-1},\ldots,t_k^{-1},u_1{\bf t}^{\alpha^b_1},\ldots,u_k{\bf t}^{\alpha^b_k}]), \end{align*} 
where  
\begin{itemize}
\item $\check{\bf t}_i:=t_1,\ldots,\check{t_i},\ldots,t_k$
\item $u_1,\ldots,u_k$, is a system of coordinates on open $U\subset X$ defining   monomial generators for $\cJ$, and
 \item ${\bf t}^{\alpha^b_i}:=t_1^{a^b_{i1}}\cdot\ldots\cdot t_k^{a^b_{ik}},\quad \mbox{with}\quad a^b_{ij}:=\nu^b_i(u_j)=b_i\nu_i(u_j)\geq 0.$ 
\end{itemize}
Note that under this  correspondence  $t_i^{-1}\mapsto {\bf t}^{\frac{1}{b_i}E_i}.$

Let \begin{align} E^0=a_1E_1+\ldots+a_kE_k\end{align} be the exceptional divisor of $\pi: Y\to X$, for the relevant $a_i\in \ZZ_{\geq 0}$, such that $\cO_Y(-E^0)=\cO_Y\cdot \cJ$.
By Proposition \ref{small}, $$B_+ = B\setminus V(\cJ {\bf t}^{-E^0})= B\setminus V(\cJ {\bf t}^{\alpha^b}),$$
where ${\bf t}^{\alpha^b}$ corresponds to ${\bf t}^{-E_0}$, under $t_i^{-1}\mapsto {\bf t}^{\frac{1}{b_i}E_i}$.
Thus by (7),
 $\alpha^b=(b_1a_1,\ldots,b_k a_k),$ and 
\begin{align*}&
B_+ = B\setminus V(\cJ {\bf t}^{-E^0})= B\setminus V(\cJ t_1^{b_1a_1}\cdot\ldots\cdot t_k^{b_ka_k}). \end{align*} 

In particular, if $X$ is regular over a field and $\cJ$  is a locally monomial ideal  on $X$, then the full cobordant blow-up $B$ of $\cJ$  with respect to $\Gamma_b$  is regular.

\subsubsection{Multiple weighted blow-ups associated with 
$\QQ$-ideals}

Consider the normalized blow-up $\pi: Y\to X$ of a monomial $\QQ$-ideal $\cJ$, with the associated exceptional divisor $E^0=a_1E_1+\ldots+a_kE_k$ with  rational, positive coefficients $a_i$, as in Section \ref{Q}. We choose $b=(b_1,\ldots,b_k)$ with $b_i\in \ZZ_{>0}$, such that $$\Gamma_b=\ZZ \frac{1}{b_1}E_1\oplus\ldots\oplus\ZZ \frac{1}{b_k}E_k\subset \Cl(Y/X)\otimes \QQ,$$ is the minimal subgroup of $\Cl(Y/X)\otimes \QQ$ containing $E^0$.

Thus any monomial $\QQ$-ideal $\cI$ and $b=(b_1,\ldots,b_k)\in \ZZ^k_{>0}$ determines a unique associated cobordant blow-up  $B_+\to X$ with respect to the group $\Gamma_b$. This  way, taking the stack-theoretic quotient, we obtain the Abramovich-Quek {\it multiple weighted blow-up $[B_+\sslash T]\to X$} from \cite{AQ},  which is necessarily regular for a regular $X$.



\bibliographystyle{amsalpha}


\end{document}